\documentclass[10pt]{article}
\usepackage{latexsym,amsfonts,amssymb,amsmath,amsthm}
\usepackage{graphicx}
\usepackage{longtable}
\usepackage{cite}
\usepackage[dvipsnames]{xcolor}
\IfFileExists{bbm.sty}{\usepackage{bbm}}{%
  \let\mathbbm\mathbf}
\usepackage{tikz}
\usetikzlibrary{arrows.meta}
\definecolor{FigBirth}{HTML}{276F9A}
\definecolor{FigDeath}{HTML}{BD6241}
\definecolor{FigPotential}{HTML}{256A73}
\definecolor{FigInk}{HTML}{29343C}
\definecolor{FigGuide}{HTML}{9CA7AD}
\tikzset{
  paper axis/.style={FigInk!85,line width=0.55pt,
    -{Stealth[length=1.7mm,width=1.15mm]}},
  paper guide/.style={FigGuide,line width=0.45pt,dash pattern=on 1.5pt off 2pt},
  paper flow/.style={FigInk,line width=0.8pt,
    -{Stealth[length=1.8mm,width=1.3mm]}}
}

\usepackage{tocloft}
\usepackage{amsfonts}
\usepackage[colorlinks=true,linkcolor=black,urlcolor=blue,citecolor=red]{hyperref}
\hypersetup{linkcolor=blue,pdftitle={Quasi-stationary distributions and sharp metastable asymptotics for birth-death processes with a strong Allee effect},pdfauthor={Tian Hou, Kaige Yan and Dun Zhou}}

\begin{document}
	\setlength{\baselineskip}{13pt}

\parindent 0.5cm
\evensidemargin 0cm \oddsidemargin 0cm \topmargin 0cm \textheight
22cm \textwidth 16cm \footskip 2cm \headsep 0cm
	
\newtheorem{theorem}{Theorem}[section]
\newtheorem{lemma}[theorem]{Lemma}
\newtheorem{proposition}[theorem]{Proposition}
\newtheorem{corollary}[theorem]{Corollary}
\newtheorem{property}[theorem]{Property}

\newtheorem{definition}{Definition}[section]
\newtheorem{example}{Example}[section]
\newtheorem{remark}{Remark}[section]

\numberwithin{equation}{section}

% Main results labeled by capital letters
% Main results use a common alphabetical counter
\newtheorem{mainthm}{Theorem}
\renewcommand{\themainthm}{\Alph{mainthm}}

\newtheorem{mainlem}[mainthm]{Lemma}
\newtheorem{maincor}[mainthm]{Corollary}
\newtheorem{mainremark}[mainthm]{Remark}
\def\p{\partial}
\def\I{\textit}
\def\R{\mathbb{R}}
\def\C{\mathbb{C}}
\def\u{\underline}
\def\l{\lambda}
\def\a{\alpha}
\def\e{\epsilon}
\def\ls{\lambda^*}
\def\D{\displaystyle}
\def\wyx{ \frac{w(y,t)}{w(x,t)}}
\def\imp{\Rightarrow}
\def\tE{\widetilde{E}}
\def\tX{\widetilde{X}}
\def\tH{\widetilde{H}}
\def\tu{\widetilde{u}}
\def\d{\mathcal{D}}
\def\aa{\mathcal{A}}
\def\DH{\mathcal{D}(\tH)}
\def\bE{\bar E}
\def\bH{\bar H}
\def\M{\mathcal{M}}
\renewcommand{\labelenumi}{(\arabic{enumi})}

\def\disp{\displaystyle}
\def\undertex#1{$\underline{\hbox{#1}}$}
\def\card{\mathop{\hbox{card}}}
\def\sgn{\mathop{\hbox{sgn}}}
\def\exp{\mathop{\hbox{exp}}}
\def\OFP{(\Omega,{\cal F},\PP)}
\newcommand\JM{Mierczy\'nski}
\newcommand\RR{\ensuremath{\mathbb{R}}}
\newcommand\CC{\ensuremath{\mathbb{C}}}
\newcommand\QQ{\ensuremath{\mathbb{Q}}}
\newcommand\ZZ{\ensuremath{\mathbb{Z}}}
\newcommand\NN{\ensuremath{\mathbb{N}}}
\newcommand\PP{\ensuremath{\mathbb{P}}}
\newcommand\abs[1]{\ensuremath{\lvert#1\rvert}}

\newcommand\normf[1]{\ensuremath{\lVert#1\rVert_{f}}}
\newcommand\normfRb[1]{\ensuremath{\lVert#1\rVert_{f,R_b}}}
\newcommand\normfRbone[1]{\ensuremath{\lVert#1\rVert_{f, R_{b_1}}}}
\newcommand\normfRbtwo[1]{\ensuremath{\lVert#1\rVert_{f,R_{b_2}}}}
\newcommand\normtwo[1]{\ensuremath{\lVert#1\rVert_{2}}}
\newcommand\norminfty[1]{\ensuremath{\lVert#1\rVert_{\infty}}}

\title{Quasi-stationary distributions and sharp metastable asymptotics for birth--death processes with a strong Allee effect}
\author{
Tian Hou$^1$\footnote{Supported by the National Natural Science Foundation of China (No.12501215) and Basic Research Program of Jiangsu (No.BK20251429).},
Kaige Yan$^1$,
Dun Zhou$^1$\footnote{Corresponding author: zhoudun@njust.edu.cn (D. Zhou). Supported by the National Natural Science Foundation of China (Nos.12671213,12331006).}\\
\vspace{0.5em}
\begin{tabular}{c}
\footnotesize $^1$ School of Mathematics and Statistics, Nanjing University of Science and Technology\\[-2pt]
\footnotesize Nanjing, Jiangsu 210094, P. R. China
\end{tabular}
}
\date{}
\maketitle
% insert the table of contents
%\tableofcontents
	
%---------------------SECTION DIVIDE LINE---------------------------
	
\begin{abstract}
We study a class of density-dependent birth--death processes with a
strong Allee effect and an absorbing extinction state. The associated
deterministic system is bistable: both the extinction state and a positive
equilibrium are locally asymptotically stable, separated
by an unstable Allee threshold. Let $K$ denote the population-size
scaling parameter. As $K\to\infty$, we derive a sharp
Eyring--Kramers-type asymptotic formula, including the leading
prefactor, for the principal eigenvalue of the killed generator, together with
a uniform approximation of the corresponding positive
eigenvector over the entire state space. As a consequence, we obtain
a sharp asymptotic formula for the mean extinction time under the
quasi-stationary distribution.

In the bistable regime, the initial population determines the competition
between early extinction and positive metastability. After the initial relaxation and before
extinction from the positive stable region, the law is approximated in
total variation by an initial-state-dependent convex combination of
the quasi-stationary distribution and the Dirac mass at extinction.
The mixing coefficient is identified precisely as the probability of reaching the positive
stable region before absorption. Moreover, we identify a probabilistic
transition layer of width $O(\sqrt K)$ around the Allee threshold,
within which the mixing coefficient converges to an explicit Gaussian
profile. In particular, a population starting at the threshold
reaches the positive metastable region or becomes extinct during the
initial transient phase with asymptotically equal probabilities.

We further establish existence and uniqueness of the quasi-stationary distribution, quantitative convergence to the quasi-stationary regime, and a discrete Gaussian approximation of the quasi-stationary distribution centered at the
positive stable population level. Our analysis combines a self-adjoint
spectral realization of the killed generator, matched asymptotic analysis of
the associated second-order recurrence, and discrete Laplace
estimates for the reversible weights.
\end{abstract}

\tableofcontents

\section{Introduction}\label{Introduction}

\subsection{Background and motivation}
Stochastic population models provide a natural framework for modeling
demographic fluctuations that are absent from deterministic equations. Let $(X_t)_{t\geq0}$ be a Markov process on a measurable state space $(E,\mathcal{E})$, where
$0\in E$ is the absorbing extinction state. Define
\[
E^*:=E\setminus\{0\},
\quad
T_0:=\inf\{t\geq0:X_t=0\}.
\]
Suppose that extinction occurs almost surely from every nonzero initial state, that is,
\[
\mathbb{P}_x(T_0<\infty)=1,
\quad x\in E^*.
\]
Then
\[
\left\|\mathbb{P}_x(X_t\in\cdot)-\delta_0\right\|_{\mathrm{TV}}=\mathbb{P}_x(T_0>t)\to0
\quad\text{as }t\to\infty
\]
for all $x\in E^*$, where $\delta_0$ denotes the Dirac mass. As a consequence, the unique stationary distribution of such a Markov process is exactly $\delta_0$. Although this stationary distribution describes the eventual fate of the population, it provides no information regarding the potentially long-lived and biologically relevant behavior of the population prior to extinction.

Quasi-stationary distributions provide a natural framework for describing
this pre-extinction regime. For a probability measure $\nu$ on $E^*$,
let
\[
\mathbb{P}_\nu(\,\cdot\,)
:=
\int_{E^*}\mathbb{P}_x(\,\cdot\,)\,\nu(dx)
\]
denote the law of the process with initial distribution $\nu$. The
measure $\nu$ is called a quasi-stationary distribution if, for every
$t>0$,
\[
\mathbb{P}_\nu(T_0>t)>0
\]
and
\[
\mathbb{P}_\nu
\bigl(X_t\in A\mid T_0>t\bigr)
=
\nu(A)
\]
for every measurable set $A\subseteq E^*$. Equivalently,
\[
\mathbb{P}_\nu(X_t\in A,\ T_0>t)
=
\mathbb{P}_\nu(T_0>t)\nu(A).
\]
Thus, a quasi-stationary distribution describes the typical population
profile conditioned on non-extinction. It also provides access to several
fundamental quantities, including the characteristic extinction time, the
rate of convergence toward the quasi-stationary regime, and the transient
behavior connecting the initial population to either extinction or a
long-lived metastable state.

Since Yaglom's pioneering work on subcritical Galton--Watson processes
\cite{Yaglom1947}, the existence, uniqueness, spectral representation, and
domains of attraction of quasi-stationary distributions have been extensively
studied for absorbing Markov chains and birth--death processes; see, among
others,
\cite{seneta1966quasi,darroch1967quasistationary,
cavender1978quasistationary,ferrari1991some,ferrari1995existence,
van1991quasi,coolen2006quasikilling,coolen2006quasi,
Gao2015quasistationary}.

For a continuous-time absorbed Markov process, the survival probability under
a quasi-stationary distribution is necessarily exponential. More precisely,
if $\nu$ is a quasi-stationary distribution, then there exists a constant
$\theta>0$ such that
\[
    \mathbb{P}_{\nu}(T_0>t)=e^{-\theta t},
    \quad
    \mathbb{E}_{\nu}[T_0]=\frac{1}{\theta};
\]
see, for example,
\cite{darroch1967quasistationary,van1991quasi}.
When $\nu$ is associated with the principal eigenvalue
$-\rho_1$ of the killed generator, one has $\theta=\rho_1$, and hence
\[
    \mathbb{P}_{\nu}(T_0>t)=e^{-\rho_1t},
    \quad
    \mathbb{E}_{\nu}[T_0]=\frac{1}{\rho_1}.
\]

In the self-adjoint discrete-spectrum setting considered in this paper, the
quasi-stationary distribution is determined by the positive eigenvector
associated with $-\rho_1$, while the higher eigenvalues govern the relaxation
of the conditioned process toward quasi-stationarity. In particular, if
\[
    -\rho_1>-\rho_2>-\rho_3>\cdots
\]
denote the eigenvalues of the killed generator, then the spectral gap $\rho_2-\rho_1$
determines the leading spectral rate of convergence after conditioning on
survival; see
\cite{van1991quasi,Chazottes2016}.
Thus, in the present framework, the study of quasi-stationary distributions is
naturally connected with the principal eigenvalue, the corresponding positive
eigenvector, the mean extinction time, and the spectral gap of the killed
generator.
Birth--death processes provide a natural microscopic model for addressing
these questions, since each transition represents the birth or death of a
single individual. The origins of these models can be traced back to Yule's pure-birth process
\cite{yule1925mathematical}, while the basic framework for continuous-time
birth--death processes was established by Feller
\cite{feller1939DieGrundlagen}. Kendall
\cite{kendall1948generalized} subsequently obtained explicit transition
probabilities for linear birth--death models, and the spectral representation
developed by Karlin and McGregor \cite{karlin1957classification} laid the
foundation for the analysis of state-dependent birth--death processes. In
particular, the spectral structure of the killed birth--death generator makes
it possible to study quasi-stationarity, extinction-time asymptotics, and
conditional convergence within a unified framework.

A particularly important mechanism in population dynamics is the
strong Allee effect, whereby the per-capita growth rate becomes
negative below a critical population density. Since the pioneering
work of Allee \cite{allee1938social}, this phenomenon has been studied
in a variety of single-species, spatial, discrete-time, and structured
population models
\cite{dennis1989allee,lewis1993allee,amarasekare1998allee,
jang2006allee,cushing2014backward,kramer2009evidence}.
At the deterministic level, a strong Allee effect produces a
bistable phase portrait in which the extinction state and a positive
equilibrium are both locally asymptotically stable and are separated
by an unstable threshold. In the corresponding stochastic system,
this bistability creates a competition between rapid absorption and
long-lived positive metastability. The purpose of this paper is to
quantify this competition through sharp spectral, extinction-time,
and quasi-stationary asymptotics.

\subsection{Related work and the discrete--continuous distinction}

A general theory of quasi-stationarity has been developed for broad classes
of absorbed Markov processes. Under suitable Lyapunov and minorization
conditions, one can establish the existence and uniqueness of a
quasi-stationary distribution, exponential convergence of the conditioned
process, and the existence of the associated $Q$-process
\cite{ChampagnatVillemonais2016,ChampagnatVillemonais2021,
ChampagnatVillemonais2023}. These general results provide a probabilistic
framework for quasi-stationarity, but do not usually yield sharp small-noise
asymptotics for the principal eigenvalue, the spectral gap, or the profile of
the quasi-stationary distribution.

More precise results have been obtained for one-dimensional absorbed
diffusions. Ji et al. \cite{Ji2021} studied the multiscale transient dynamics
of absorbed singular diffusions and derived small-noise asymptotics for the
principal eigenvalue and the gap between the first two eigenvalues.
Shen et al. \cite{ShenWangYi} investigated the concentration of
quasi-stationary distributions on invariant sets of the limiting deterministic
flow. More recently, Qi et al. \cite{Qi2024} established large-deviation
principles, refined concentration estimates, and asymptotic results for
extinction times, eigenvalues, and eigenfunctions. These results mainly apply
to one-dimensional small-noise diffusions under suitable regularity, boundary,
and stability assumptions. The bistable case in which both the absorbing
state and a positive equilibrium are stable, separated by an unstable
equilibrium, was considered in \cite{Yan2026quasistationary}.

Sharp quasi-stationary asymptotics have also been obtained for discrete
population models. For a class of density-dependent logistic birth--death
processes in which the origin is repelling at the deterministic level,
Chazottes et al. \cite{Chazottes2016} derived sharp estimates for the
principal eigenvalue and the associated eigenvector, quantified the rate of convergence
toward the quasi-stationary distribution, and obtained a Gaussian approximation centered at the positive stable equilibrium.

Although density-dependent birth--death processes admit diffusion
approximations, their quasi-stationary spectral problems are not equivalent.
Moreover,
diffusion approximations valid on finite time intervals do not automatically
preserve exponentially long extinction times, rare transitions across an
unstable threshold, or the low-lying spectrum of the killed generator. These
differences make a separate analysis of discrete birth--death processes
necessary.

\subsection{Model, assumptions, and notation}
\label{section of notations}
\label{subsec:notation}

Write $\mathbb{N}:=\{0,1,2,\ldots\}$,
$\mathbb{N}^*:=\{1,2,\ldots\}$, and $\mathbb{R}_+:=[0,\infty)$.
Let $(X_t^K)_{t\ge0}$ be a birth--death process on $\mathbb{N}$
with absorbing state $0$. For a population of size $n\in\mathbb{N}^*$,
the total birth and death rates are
\[
\lambda_n
=
n\widetilde{\lambda}\left(\frac nK\right),
\quad
\mu_n
=
n\widetilde{\mu}\left(\frac nK\right),
\]
respectively, where $K>1$ is the population-size scaling parameter.
We impose the following assumptions on the positive per-capita rate
functions $\widetilde{\lambda}$ and $\widetilde{\mu}$.

\begin{enumerate}
\item[\textbf{(A1)}]
The functions $\widetilde{\lambda}$ and $\widetilde{\mu}$ belong to
$C^2(\mathbb{R}_+)$, and
\[
0<\widetilde{\lambda}(0)<\widetilde{\mu}(0).
\]

\item[\textbf{(A2)}]
Both $\widetilde{\lambda}$ and $\widetilde{\mu}$ are increasing.
There exist exactly two points $0<x_1<x_2$ such that
\[
\widetilde{\lambda}(x_i)
=
\widetilde{\mu}(x_i),
\quad i=1,2.
\]
The two intersections are transversal and satisfy
\[
\widetilde{\lambda}'(x_1)
>
\widetilde{\mu}'(x_1),
\quad
\widetilde{\lambda}'(x_2)
<
\widetilde{\mu}'(x_2).
\]
Moreover,
\[
\begin{cases}
\widetilde{\lambda}(x)<\widetilde{\mu}(x),
    & x\in(0,x_1),\\[1mm]
\widetilde{\lambda}(x)>\widetilde{\mu}(x),
    & x\in(x_1,x_2),\\[1mm]
\widetilde{\lambda}(x)<\widetilde{\mu}(x),
    & x\in(x_2,\infty),
\end{cases}
\]
and the function
\[
x\longmapsto
\ln\frac{\widetilde{\mu}(x)}
          {\widetilde{\lambda}(x)}
\]
is strictly decreasing on $(0,x_1)$.
\item[\textbf{(A3)}]
The following conditions hold:
\[
\lim_{x\to\infty}
\frac{\widetilde{\lambda}(x)}
     {\widetilde{\mu}(x)}
=0,
\quad
\sup_{x\in\mathbb{R}_+}
\frac{\widetilde{\mu}'(x)}
     {\widetilde{\mu}(x)}
<\infty,
\]
and
\[
\int_{x_2}^{\infty}
\frac{dx}{x\widetilde{\mu}(x)}
<\infty.
\]
\end{enumerate}

Define $H:\mathbb{R}_+\to\mathbb{R}$ by
\begin{equation}\label{Hx}
H(x)
=
\int_{x_2}^{x}
\ln\frac{\widetilde{\mu}(s)}
          {\widetilde{\lambda}(s)}\,ds.
\end{equation}

We further impose the following regularity condition.

\begin{enumerate}
\item[\textbf{(H)}]
The function $H$ is three times differentiable and satisfies
\[
\sup_{x\in\mathbb{R}_+}
(1+x^2)|H'''(x)|<\infty.
\]
\end{enumerate}

We write
\[
h(x):=H'(x)=\ln\frac{\widetilde{\mu}(x)}{\widetilde{\lambda}(x)},
\quad F(x):=H(0)-H(x),
\quad V(x):=x\bigl(\widetilde{\lambda}(x)-\widetilde{\mu}(x)\bigr).
\]
Thus $H(x_2)=0$, $F(0)=0$, and $F(x_2)=H(0)$. The quantity
$c:=H(0)$ may have either sign or vanish.

For later use, set
\[
n_1(K):=\lfloor x_1K\rfloor,
\quad
n_2(K):=\lfloor x_2K\rfloor.
\]
The corresponding rounding errors are
$\epsilon_i(K):=Kx_i-n_i(K)\in[0,1)$, $i=1,2$.
We also use the tail cutoff
\[
    x_{\frac12}
    :=
    \max\left\{
        x>0:
        \frac{\widetilde{\lambda}(x)}
             {\widetilde{\mu}(x)}
        =
        \frac12
    \right\},
    \quad
    n_{\frac12}(K)
    :=
    \left\lfloor x_{\frac12}K\right\rfloor.
\]
Under the assumptions above,
\[
    0<x_1<x_2<x_{\frac12}<\infty.
\]

Fix $K_0>1$ sufficiently large such that
$1\le n_1(K)<n_2(K)<n_{\frac12}(K)$ for every $K\ge K_0$.
All asymptotic assertions below concern $K\to\infty$ with $K\ge K_0$;
increasing $K_0$ when necessary does not alter the assumptions.
Unless stated otherwise, dependence on $K$ is suppressed in the notation.

Define the reversible weights by
\[
\pi_1:=\frac1{\mu_1},
\quad
\pi_n
:=
\frac{\lambda_1\cdots\lambda_{n-1}}
     {\mu_1\cdots\mu_n},
\quad n\geq2.
\]
The reciprocal conductances are
\[
w_n:=\frac1{\lambda_n\pi_n},\quad n\ge1.
\]
Define $u^0=(u_n^0)_{n\in\mathbb{N}^*}$ by
\begin{equation}\label{u_n^0}
u_n^0
=
u_1^0
\left(
1+\sum_{j=1}^{n-1}\frac1{\lambda_j\pi_j}
\right),
\quad n\geq1,
\end{equation}
where
\[
u_0^0=0,
\quad
u_{n_1(K)}^0=1,
\quad
u_1^0
=
\left(
1+\sum_{j=1}^{n_1(K)-1}
\frac1{\lambda_j\pi_j}
\right)^{-1}.
\]

In Section~\ref{sec:preliminaries}, we construct a self-adjoint
realization $L$ of the killed birth--death generator in the weighted space
$\ell^2(\pi)$, whose notation is specified in the table below.
As shown in Lemma~\ref{operatorL},
the spectrum of $L$ is discrete
and can be written as
\[
-\rho_1>-\rho_2>-\rho_3>\cdots\to-\infty.
\]
Here $\rho_1>0$, the principal eigenvalue $-\rho_1$ is simple
and admits a strictly positive eigenvector.

For $j\in\mathbb{N}$, let $T_j:=\inf\{t\ge0:X_t^K=j\}$,
with $\inf\emptyset:=\infty$. $\mathbb{P}_n$ and
$\mathbb{E}_n$ denote probability and expectation, respectively,
for the process started from $X_0^K=n$.
For $A\subseteq\mathbb{N}^*$, the killed transition kernel is
\[
P_t(n,A):=\mathbb{P}_n(X_t^K\in A,\ T_0>t).
\]
In particular, $P_t(n,\mathbb{N}^*)=\mathbb{P}_n(T_0>t)$, and
$P_t(n,\cdot)/P_t(n,\mathbb{N}^*)$ is the law conditioned on
survival. We also write $P_tf(n)=\sum\limits_{j\ge1}P_t(n,\{j\})f(j)$.

The remaining notation used throughout the paper is collected below.
The constructions and estimates for the spectral and probabilistic
quantities are given in the indicated results.
\begingroup
\small
\renewcommand{\arraystretch}{1.2}
\setlength{\LTpre}{0.5\baselineskip}
\setlength{\LTpost}{0.5\baselineskip}
\begin{longtable}{@{}p{0.24\textwidth}@{\hspace{0.025\textwidth}}p{0.735\textwidth}@{}}
\hline
Symbol & Definition or meaning \\ \hline
\endfirsthead
\hline
Symbol & Definition or meaning (continued) \\ \hline
\endhead
\hline
\endfoot
\multicolumn{2}{@{}l}{\textit{Sequence spaces and operators}} \\[2pt]
$\ell^2(\pi)$
& Weighted Hilbert space of sequences $u=(u_n)_{n\ge1}$ such that
  $\sum\limits_{n\ge1}\pi_n|u_n|^2<\infty$. \\[3pt]
$\langle u,v\rangle_\pi$, $\|u\|_\pi$
& $\langle u,v\rangle_\pi=\sum\limits_{n\ge1}\pi_n\overline{u_n}v_n$ and
  $\|u\|_\pi=\langle u,u\rangle_\pi^{1/2}$. \\[3pt]
$\|u\|_\infty$
& Supremum of $|u_n|$ over the index set under consideration. \\[3pt]
$\mathfrak{D}$, $\mathcal{D}(L)$
& Space of finitely supported sequences on $\mathbb{N}^*$ and domain
  of the closed generator $L$, respectively; the latter is characterized
  in \eqref{eq:generator-domain}. \\[3pt]
$\mathbbm{1}$, $\mathbbm{1}_A$
& Constant sequence equal to one and indicator of $A$, respectively. \\[4pt]
\hline
\multicolumn{2}{@{}l}{\textit{Spectral quantities}} \\[2pt]
$-\rho_j(K)$, $j\ge1$
& Eigenvalues of $L$, ordered by $0<\rho_1(K)<\rho_2(K)<\cdots$;
  see Lemma~\ref{operatorL}. \\[3pt]
$u^0=(u_n^0)$
& Zero-energy sequence defined in \eqref{u_n^0}, with
  $u_0^0=0$ and $u_{n_1(K)}^0=1$. \\[3pt]
$\Phi=(\Phi_n)$
& Truncated profile $\Phi_n=u_{\min\{n,n_2(K)\}}^0$;
  see \eqref{Phi=}. \\[3pt]
$\phi=(\phi_n)$
& Strictly positive principal eigenvector satisfying
  $L\phi=-\rho_1\phi$, normalized by $\phi_{n_1(K)}=1$;
  see \eqref{phi n}. \\[4pt]
\hline
\multicolumn{2}{@{}l}{\textit{Probability measures and spectral coefficients}} \\[2pt]
$\nu=\nu^K=(\nu_n)$
& Quasi-stationary distribution in Theorem~\ref{UniqueQSD}, with
  $\nu_n=\pi_n\phi_n/\langle\phi,\mathbbm{1}\rangle_\pi$. \\[3pt]
$\delta_0$
& Dirac probability measure at the absorbing state $0$. \\[3pt]
$\alpha_n(K)$
& $\alpha_n(K)=\Phi_n/u_{n_2(K)}^0$. For $1\le n\le n_2(K)$,
  this equals $\mathbb{P}_n(T_{n_2(K)}<T_0)$; for $n>n_2(K)$,
  it equals one. See \eqref{eq:alpha-hitting}. \\[3pt]
$S_K$
& Total reversible weight $S_K=\sum\limits_{j\ge1}\pi_j$. \\[3pt]
$B_K$
& Uniform spectral amplitude
  $B_K=\bigl(S_K/\min\limits_{1\le j\le n_{\frac12}(K)}\pi_j\bigr)^{1/2}$. \\[3pt]
$d_n(K)$
& Principal-projection coefficient
  $d_n(K)=\phi_n\langle\phi,\mathbbm{1}\rangle_\pi/\|\phi\|_\pi^2$.
  The quantities $S_K$, $B_K$, and $d_n(K)$ are used in
  Section~\ref{M-result}. \\
\end{longtable}
\endgroup

For probability measures $\mu$ and $\nu$ on $\mathbb{N}$, our total
variation convention is
\[
\|\mu-\nu\|_{\mathrm{TV}}
:=\frac12\sum_{n\ge0}|\mu_n-\nu_n|
=\sup_{A\subseteq\mathbb{N}}|\mu(A)-\nu(A)|.
\]
The same convention applies to a finite signed measure $\zeta$ of
total mass zero: $\|\zeta\|_{\mathrm{TV}}:=\frac12\sum\limits_{n\ge0}|\zeta_n|$.
For measures on $\mathbb{N}^*$, the sum and supremum are restricted
to $\mathbb{N}^*$. When used for probability laws on the metastable
time scale, $\mu_K\approx\nu_K$ means that
$\|\mu_K-\nu_K\|_{\mathrm{TV}}\to0$ in the stated limit.

The standard normal distribution function is
\[
\mathcal{N}(z):=\frac1{\sqrt{2\pi}}\int_{-\infty}^{z}e^{-\frac{s^2}{2}}\,ds.
\]

Operator norms are induced by the stated domain and range norms;
$\|A\|_{X\to Y}$ denotes the norm of $A:X\to Y$, and $I$ denotes
the identity operator on the space in use. Auxiliary integer cutoffs
such as $k$, $m$, and $N$, and error terms denoted by $\varepsilon$,
are defined locally where they occur.

The logarithm $\ln$ is natural. Empty sums equal zero and empty
products equal one. Generic positive constants, written $C$, $C'$,
or with a local subscript, may change from line to line. Unless
specified otherwise, they may depend on the fixed rate functions
and fixed auxiliary parameters, but not on $K$ or on indices and
times over which an estimate is stated to be uniform.
Asymptotic notation refers to $K\to\infty$ unless another limit is
specified. For $b_K>0$, we use
\[
\begin{aligned}
a_K=O(b_K)
&\quad\Longleftrightarrow\quad |a_K|\le Cb_K
   \text{ for all sufficiently large }K,\\
a_K=o(b_K)
&\quad\Longleftrightarrow\quad a_K/b_K\to0\text{ as }K\to\infty.
\end{aligned}
\]
An $O(b_K)$ term may have either sign; one-sided bounds
are written with explicit positive constants. For families indexed
by $n\in I_K$, uniform $O(b_K)$ means
$\sup\limits_{n\in I_K}|a_{K,n}|\le Cb_K$, and uniform $o(b_K)$ means
$\sup\limits_{n\in I_K}|a_{K,n}|/b_K\to0$. For positive quantities,
$a_K\asymp b_K$ means $C^{-1}b_K\le a_K\le Cb_K$, whereas
$a_K\sim b_K$ means $a_K/b_K\to1$ in the limit under consideration.

\subsection{Main results}\label{M-result}

We first establish sharp asymptotics for the principal eigenvalue $-\rho_1$ and the corresponding eigenvector $\phi=(\phi_n)_{n\in\mathbb{N}^*}$. 

\begin{mainthm}\label{theorem of principal eigenvalue and eigenfunction}
	For all sufficiently large $K$, one has
	\begin{equation}\label{rho1}
        \rho_1(K)=\frac{x_2\widetilde{\lambda}(x_2)\sqrt{-H''(x_1)H''(x_2)}}{2\pi} \exp\left(-KH(x_1)\right)\left(1+O\left(\frac{(\ln K)^3}{\sqrt{K}}\right)\right).
	\end{equation}
	Moreover, for all sufficiently large $K$ and $n\in\mathbb{N}^*$, 
	\[
	|\phi_n(K)-\Phi_n(K)|\le O(1)\rho_1(K)K\Phi_n(K),
	\]
	where  $\phi=(\phi_n)_{n\in\mathbb{N}^*}$ is the eigenvector corresponding to the principal eigenvalue $-\rho_1$, and $\Phi=(\Phi_n)$ is given by
	\begin{equation}\label{Phi=}
		\Phi_n(K)=
		\begin{cases}
			u_n^0, & \mbox{if } 1\le n\le n_2(K),\\
			u_{n_2(K)}^0, & \mbox{if } n\ge n_2(K).
		\end{cases}
	\end{equation}
	In particular, for all sufficiently large $K$, one has
	\[
	\sup_{n\in\mathbb{N}^*}|\phi_n(K)-\Phi_n(K)|\le O(1)\rho_1(K)K.
	\]
\end{mainthm}

Recall that $T_0:=\inf\{t\ge0:X^K_t=0\}$ is the absorption time. In what follows, we characterize the unique quasi-stationary distribution
$\nu^K=(\nu_n^K)_{n\in\mathbb{N}^*}$ of the process $(X_t^K)_{t\geq0}$. For simplicity, we write $\nu$ in place of $\nu^K$ whenever the dependence on $K$ is clear from the context, except in Theorem \ref{Gaussian approximation}, where the dependence is made explicit.

\begin{mainthm}\label{UniqueQSD}
	For all sufficiently large $K$, the birth--death process $(X_t^K)_{t\geq0}$ admits a unique quasi-stationary distribution $\nu=(\nu_n)_{n\in\mathbb{N}^*}$, which is given by
	\[
	\nu_n=\frac{\pi_n\phi_n}{\langle\phi,\mathbbm{1}\rangle_{\pi}}\text{ for each }n\in\mathbb{N}^*.
	\]
\end{mainthm}

\begin{maincor}
For all sufficiently large $K$, the mean extinction time of $(X_t^K)_{t\geq0}$ satisfies
\[
\mathbb{E}_{\nu}[T_0]=\frac{2\pi}{x_2\widetilde{\lambda}(x_2)\sqrt{-H''(x_1)H''(x_2)}} \exp\left(KH(x_1)\right)\left(1+O\left(\frac{(\ln K)^3}{\sqrt{K}}\right)\right).
\]
\end{maincor}

Recall $c=H(0)$, the tail cutoff $n_{\frac12}(K)$ and the fixed lower bound
$K_0$ from Section~\ref{section of notations}. For the uniform estimates, set
\[
M_K:=\sum_{j=n_{\frac12}(K)}^\infty
\frac{1}{\lambda_j\pi_j}\sum_{p=j+1}^\infty\pi_p,
\quad c_0:=\left(\sup_{K\ge K_0}M_K\right)^{-1}>0.
\]
The positivity of $c_0$ is proved below in
Proposition~\ref{proposition7.3}. We also use
\[
S_K:=\sum_{j=1}^\infty\pi_j,\quad
B_K:=\left(\frac{S_K}{\min\limits_{1\le j\le n_{\frac12}(K)}\pi_j}\right)^{\frac12},
\quad
d_n(K):=\phi_n\frac{\langle\phi,\mathbbm{1}\rangle_\pi}{\|\phi\|_\pi^2}.
\]
These quantities are finite, $d_n(K)>0$, and $B_K\ge1$.

We use the total variation convention of Section~\ref{subsec:notation}. Denote by \(P_t\) the semigroup generated by the process $(X_t^K)_{t\geq0}$. The following theorem quantifies the rate of convergence towards the
quasi-stationary distribution and describes the metastable mixture.
\begin{mainthm}\label{theorem of TV}
Let $u^0$ be defined by \eqref{u_n^0}, and set
\[
\alpha_n(K):=\frac{\Phi_n}{u_{n_2(K)}^0}
=\begin{cases}
\frac{u_n^0}{u_{n_2(K)}^0},&n\le n_2(K),\\
1,&n\ge n_2(K).
\end{cases}
\]
There exists a constant $C>0$, independent of $K,n,t$, such that
\begin{equation}\label{equation TV}
\begin{split}
&\left\|\mathbb{P}_n(X_t^K\in\cdot)
-\left[\alpha_n(K)\nu+(1-\alpha_n(K))\delta_0\right]\right\|_{\mathrm{TV}}\\
&\quad\le C\left[K\rho_1e^{-\rho_1t}+1-e^{-\rho_1t}
+e^{-\frac{c_0}{4}t}+R(t,K,\rho_2)\right],
\end{split}
\end{equation}
where
\[
R(t,K,\rho_2):=B_K e^{-\rho_2\frac{t}{2}}.
\]
This holds for every sufficiently large $K$, every $n\ge1$, and every $t\ge0$.
Moreover,
\begin{equation}\label{ptnAptnN}
\left\|\frac{P_t(n,\cdot)}{P_t(n,\mathbb{N}^*)}-\nu\right\|_{\mathrm{TV}}
\le\min\left\{1,\frac{O(1)}{d_n(K)}
\left[e^{-\frac{c_0}{8}t}
+B_K e^{-\frac{\rho_2-\rho_1}{2}t}\right]\right\}.
\end{equation}
Both rates on the right-hand side are positive, since $c_0>0$ and
$\rho_2>\rho_1$.
The amplitudes satisfy $B_K\le O(1) e^{C' K}$ for some \(C'>0\) and
$d_n(K)=\alpha_n(K)(1+O(K\rho_1(K)))$, uniformly in $n$.
\end{mainthm}

A sharper version, which records the pointwise estimate, is given in Remark \ref{remark pointwise}. 

\begin{mainremark}
{\normalfont
The estimate in Theorem~\ref{theorem of TV} is most informative on
an intermediate metastable time scale. Such a scale exists under
the standing assumptions. Indeed, Lemma~\ref{lem:coarse-spectral-separation}
gives a constant $c_{\mathrm{sp}}>0$, independent of $K$, such that
\[
\rho_2(K)\ge\frac{c_{\mathrm{sp}}}{K}
\quad\text{for all sufficiently large }K.
\]
Together with $B_K\le O(1)e^{C'K}$ and \eqref{rho1}, this implies that
the choice $t_K=K^3$ satisfies
\[
e^{-\frac{c_0}{4}t_K}+B_K e^{-\frac{\rho_2(K)}{2} t_K}\to0,
\quad \rho_1(K)t_K\to0.
\]
More generally, the same conclusions hold whenever
$t_K/K^2\to\infty$ and $\rho_1(K)t_K\to0$. The first condition
is a sufficient, nonoptimal relaxation scale, while the second
ensures that extinction from the positive metastable region remains
unlikely. The following conclusions apply to any sequence $t_K$
for which the relaxation terms in \eqref{equation TV} vanish and
$\rho_1(K)t_K\to0$.

For $n\leq n_2(K)$, Theorem~\ref{theorem of TV} then yields
\[
    \mathbb{P}_n\bigl(X_{t_K}^K\in\cdot\bigr)
    \approx
    \alpha_n(K)\nu(\cdot)
    +
    \bigl(1-\alpha_n(K)\bigr)\delta_0(\cdot),
\]
where
\[
    \alpha_n(K)
    =
    \mathbb{P}_n\bigl(T_{n_2(K)}<T_0\bigr)
\]
is the probability of reaching the positive metastable region before
absorption; see \eqref{eq:alpha-hitting}. Thus, the quasi-stationary and absorbing components 
correspond to trajectories that enter the positive 
metastable region and to those absorbed during 
the initial transient phase, respectively.

The mixing coefficient exhibits a nontrivial transition across an
$O(\sqrt K)$ neighborhood of the Allee threshold. Indeed, suppose
that a sequence $(n_K)_{K\geq1}$ satisfies
\[
    \frac{n_K-n_1(K)}{\sqrt K}\to z
\]
for a constant $z\in\mathbb{R}$. By
Corollary~\ref{prop:transition-profile},
\[
    \alpha_{n_K}(K)
    \to
    p(z)
    :=
    \mathcal{N}\left(\sqrt{-H''(x_1)}\,z\right),
\]
where $\mathcal{N}$ is the standard normal distribution function
defined in Section~\ref{subsec:notation}. Consequently, provided that the corresponding
initial-relaxation terms vanish,
\[
    \left\|
        \mathbb{P}_{n_K}\bigl(X_{t_K}^K\in\cdot\bigr)
        -
        \left[
            p(z)\nu(\cdot)
            +
            \bigl(1-p(z)\bigr)\delta_0(\cdot)
        \right]
    \right\|_{\mathrm{TV}}
    \to0.
\]

In particular, taking $n_K=n_1(K)$ gives $z=0$, and hence
\[
    \alpha_{n_1(K)}(K)\to \mathcal{N}(0)=\frac12.
\]
More precisely, Proposition~\ref{un2K=2+O} gives
\[
    \alpha_{n_1(K)}(K)
    =
    \frac12
    +
    O\left(\frac{(\ln K)^3}{\sqrt K}\right).
\]
Therefore,
\[
    \mathbb{P}_{n_1(K)}
    \bigl(X_{t_K}^K\in\cdot\bigr)
    \approx
    \frac12\nu(\cdot)
    +
    \frac12\delta_0(\cdot).
\]
Thus, a population starting at the Allee threshold reaches the
positive metastable region or becomes extinct during the initial
transient phase with asymptotically equal probabilities. In this
sense, $n_1(K)$ is both the deterministic separatrix and the center
of the probabilistic transition layer between the two competing
transient outcomes.

For initial states already in the positive stable region, $n\geq n_2(K)$, one has $\alpha_n(K)=1$, and therefore
\[
    \mathbb{P}_n\bigl(X_{t_K}^K\in\cdot\bigr)
    \approx
    \nu(\cdot).
\]
Hence, after the initial relaxation and before the extinction time
scale $\rho_1(K)^{-1}$, the law of a population starting in the positive stable
region is well approximated by the quasi-stationary distribution.
}
\end{mainremark}

Finally, the quasi-stationary distribution admits a discrete Gaussian
approximation near the positive stable equilibrium.
\begin{mainthm}\label{Gaussian approximation}
Set $\sigma^2=1/H''(x_2)$, and let $G^K$ be the probability measure
on $\mathbb{N}^*$ defined by
\[
G_n^K:=\frac{1}{Z(K)}
\exp\left(-\frac{(n-n_2(K))^2}{2K\sigma^2}\right),
\quad
Z(K):=\sum_{n=1}^\infty
\exp\left(-\frac{(n-n_2(K))^2}{2K\sigma^2}\right).
\]
Then $Z(K)=\sqrt{2\pi K}\sigma+O(1)$, and for every sufficiently large $K$,
\[
\|\nu^K-G^K\|_{\mathrm{TV}}
\le O\left(\frac{(\ln K)^4}{\sqrt K}\right).
\]
\end{mainthm}

\subsection{Comparison with previous results, technical difficulties,
and novelty}

The closest discrete counterpart of the present work is the
monostable setting studied in \cite{Chazottes2016}. In that setting,
the origin is repelling for the limiting deterministic dynamics,
whereas there is a unique positive asymptotically stable equilibrium.
Consequently, every positive macroscopic initial condition is
deterministically driven toward the same positive equilibrium once
the population moves away from the absorbing boundary. There is
therefore essentially a single pre-extinction scenario: after an
initial relaxation, the population enters the positive metastable
region, remains there for a long time, and eventually becomes extinct
through a rare fluctuation.

The bistable regime considered here has a fundamentally different
structure. The extinction state $0$ and the positive equilibrium
$x_2$ are both locally asymptotically stable and are separated by
the unstable Allee threshold $x_1$. Consequently, two competing
transient outcomes coexist: the process may be absorbed before
reaching the positive stable region, or it may cross the threshold
and enter a long-lived metastable regime near $x_2$. In contrast to the
monostable setting, the initial state therefore plays a decisive role in determining the probabilities of these two outcomes.

This probabilistic basin-selection mechanism is encoded in the global profile of the principal eigenvector. Theorem~\ref{theorem of TV}
shows that the transient law is approximated by an
initial-state-dependent mixture of the quasi-stationary distribution
and the Dirac mass at extinction, whose mixing coefficient is precisely the probability of reaching the positive stable region before absorption.
Moreover, Corollary~\ref{prop:transition-profile} identifies a
Gaussian transition of this coefficient across an $O(\sqrt K)$
neighborhood of the Allee threshold. Thus, the deterministic
separatrix gives rise to a probabilistic transition layer between early extinction and positive metastability, a phenomenon that does not occur in the monostable framework considered in \cite{Chazottes2016}.

The geometry governing the principal eigenvalue is also different.
In the monostable regime, extinction requires a rare fluctuation from
the positive stable equilibrium directly toward the absorbing
boundary, and the relevant potential barrier is anchored at $0$.
In the present bistable regime, extinction from the positive stable
region proceeds through a two-stage mechanism: a rare fluctuation
first carries the population across the interior unstable threshold
$x_1$, after which the deterministic drift rapidly drives it toward
$0$. Since $H(x_2)=0$, the exponentially slow part of extinction
is governed by the interior barrier
\[
    H(x_1)-H(x_2)=H(x_1).
\]
The sharp asymptotic formula for $\rho_1(K)$ obtained above shows
that its leading prefactor is determined by the local geometry at
both the unstable threshold $x_1$ and the positive stable
equilibrium $x_2$. This interior-barrier geometry differs from the
boundary-barrier geometry in \cite{Chazottes2016}: in the present
scaling, it produces an order-one prefactor rather than the
boundary-induced $K$-dependent prefactor arising in the monostable
case. Thus, bistability changes not only the action governing
extinction, but also the sharp pre-exponential factor.

The continuous bistable model studied in
\cite{Yan2026quasistationary} and the present birth--death model share
a bistable deterministic phase portrait. A comparison of their low-lying
spectra must also account for the stochastic dynamics near the absorbing
boundary. For the present process, freezing the per-capita rates near
$0$ produces a subcritical linear birth--death operator, whose boundary
relaxation rate is
\[
\widetilde{\mu}(0)-\widetilde{\lambda}(0).
\]
More precisely, the spectral-gap asymptotics established in \cite[Theorem A]{zhou26} yield
\[
    \lim_{K\to\infty}
    \bigl(\rho_2(K)-\rho_1(K)\bigr)
    =
    \min\left\{
        \widetilde{\mu}(0)-\widetilde{\lambda}(0),
        V'(x_1),
        -V'(x_2)
    \right\},
\]
where $V(x)=x(\widetilde{\lambda}(x)-\widetilde{\mu}(x))$  is the deterministic drift. Thus, the asymptotic relaxation rate is selected by the slowest among
three distinct local mechanisms: the boundary relaxation near $0$,
the instability at the Allee threshold $x_1$, and the local
relaxation near the positive stable equilibrium $x_2$. We invoke
this result only to interpret the relaxation terms appearing in
Theorem~\ref{theorem of TV}; its proof lies outside the scope of the
present paper. The existence of the metastable time window used here
follows independently from the coarser bound
$\rho_2(K)\ge c_{\mathrm{sp}}/K$ proved in
Lemma~\ref{lem:coarse-spectral-separation}. The boundary contribution
depends on the local stochastic
scaling; its presence or absence is not determined solely by whether
the state space is discrete or continuous.

These structural differences lead to new analytical difficulties.
First, because the absorbing state is stable, survival conditioning
acts against the deterministic drift throughout the entire
low-population basin $(0,x_1)$, rather than only near the absorbing
boundary. Consequently, the principal eigenvector cannot be
determined solely from its localization near the positive stable
equilibrium. Its global profile must simultaneously encode the
probability of escaping the extinction basin, crossing the unstable
threshold, and entering the positive metastable region. Resolving
the resulting basin-selection mechanism requires uniform control of
the eigenvector across the $O(\sqrt K)$ transition layer around
$n_1(K)$, rather than only pointwise information at the threshold.

Second, the eigenvalue equation for the killed birth--death
generator is a nonconstant-coefficient second-order recurrence whose
dominant behavior changes across four regions: the low-population
extinction basin, the unstable threshold, the positive stable
equilibrium, and the large-population tail. Local constructions in
these regions must be matched while retaining estimates that are
uniform over the entire state space. This global matching argument
is the main analytical ingredient behind both the sharp
principal-eigenvalue asymptotics and the uniform approximation of
the associated eigenvector.

\subsection{Proof strategy and organization}

We begin by realizing the killed generator as a self-adjoint
Jacobi-type operator on $\ell^2(\pi)$. To identify its principal
eigenpair, we construct solutions of the associated second-order
recurrence in two overlapping regions. On the finite interval
extending from the absorbing boundary to the positive stable region,
we construct a perturbation of the normalized zero-energy solution.
Beyond the positive equilibrium, a bounded integral operator on
sequences and its Neumann series produce a solution with the required
behavior at infinity. Matching
these two solutions near $n_2(K)$ yields a scalar characteristic
equation whose unique small zero is the principal eigenvalue.
Discrete Laplace estimates for the reversible weights then give the
sharp prefactor in \eqref{rho1} and uniform control of the associated
positive eigenvector over the whole state space. The same kernel
estimates also yield a coarse lower bound of order $1/K$ for
$\rho_2(K)$, which suffices to separate a polynomial relaxation
scale from the exponentially long extinction scale.

The global eigenvector estimates also identify the
initial-state-dependent basin-selection probability. A local
quadratic expansion of $H$ around the unstable threshold $x_1$,
combined with estimates showing that the contributions away from
$x_1$ are negligible, yields the Gaussian transition profile across
the $O(\sqrt K)$ threshold layer. The probabilistic approximation of
the transient law then follows from the spectral decomposition of the
killed semigroup, together with hitting-time estimates and uniform
bounds on the weighted eigenvectors. These estimates separate the
rapid approach to either extinction or the positive metastable region
from the exponentially slow escape from that region. Finally, a
local expansion of $H$ around the positive stable equilibrium
$x_2$, combined with tail estimates away from $x_2$, yields the
Gaussian approximation of the quasi-stationary distribution.

The remainder of the paper is organized as follows.
Section~\ref{sec:preliminaries} collects the deterministic, spectral,
and analytic preliminaries needed in the subsequent analysis.
In Section~\ref{section of principal eigenvalue}, we construct and
match solutions of the eigenvalue recurrence in different population
regions, derive the sharp asymptotics of the principal eigenvalue,
and establish uniform estimates for the associated positive
eigenvector. The Gaussian transition profile of the basin-selection
probability near the Allee threshold and the coarse spectral
separation bound are also obtained there.
In Section~\ref{QSD and Gauss}, we establish the existence and
uniqueness of the quasi-stationary distribution and derive pointwise
and uniform convergence estimates for the conditioned and
unconditioned laws in total variation.
Section~\ref{section of Gaussian approximation} is devoted to the
Gaussian approximation of the quasi-stationary distribution near the
positive stable equilibrium.
In Section~\ref{section of Application}, we apply the general results
to a logistic birth--death process with a strong Allee effect.
The appendix collects the auxiliary estimates used throughout the
paper.

\section{Spectral and analytic preliminaries}
\label{sec:preliminaries}

\subsection{Deterministic structure and the discrete potential}
\label{subsec:deterministic-structure}

We first explain the deterministic structure encoded in the assumptions.
If $X_0^K/K\to x_0$ in probability, the density-dependent Markov-chain
limit theorem of \cite{Kurtz} implies that $X_t^K/K$ converges in
probability, uniformly on every fixed finite time interval, to the
solution with initial condition $x(0)=x_0$ of
\[
\frac{dx}{dt}
=
V(x),
\quad
V(x)
:=
x\bigl(
\widetilde{\lambda}(x)
-
\widetilde{\mu}(x)
\bigr).
\]
The equilibria are $0,x_1$, and $x_2$. In the bistable regime,
$0$ and $x_2$ are locally asymptotically stable, whereas $x_1$ is
unstable. Thus, $x_1$ represents the Allee threshold and $x_2$ represents
the positive stable population level (see Figure \ref{fig:lambdamu}).

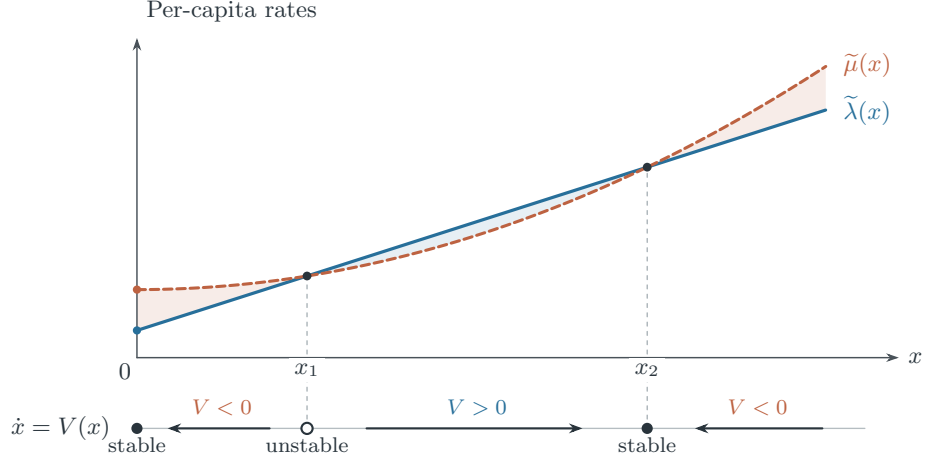
\begin{figure}[htbp]
\centering
\begin{tikzpicture}[x=2.25cm,y=0.18cm,font=\small,
  line cap=round,line join=round]
% The illustrative rates are lambda(x)=4x+2 and mu(x)=x^2+5.
\path[fill=FigDeath!12]
  plot[domain=0:1,samples=35] (\x,{\x*\x+5})
  -- plot[domain=1:0,samples=35] (\x,{4*\x+2}) -- cycle;
\path[fill=FigBirth!11]
  plot[domain=1:3,samples=60] (\x,{4*\x+2})
  -- plot[domain=3:1,samples=60] (\x,{\x*\x+5}) -- cycle;
\path[fill=FigDeath!12]
  plot[domain=3:4.05,samples=40] (\x,{\x*\x+5})
  -- plot[domain=4.05:3,samples=40] (\x,{4*\x+2}) -- cycle;

\draw[paper guide] (1,-5.2)--(1,6);
\draw[paper guide] (3,-5.2)--(3,14);
\draw[paper axis] (0,0)--(4.48,0) node[right,text=FigInk] {$x$};
\draw[paper axis] (0,0)--(0,23.5);
\node[anchor=south west,text=FigInk] at (0,24) {Per-capita rates};

\draw[FigBirth,line width=1.2pt,domain=0:4.05,samples=80]
  plot (\x,{4*\x+2});
\draw[FigDeath,line width=1.15pt,dash pattern=on 3.2pt off 1.8pt,
  domain=0:4.05,samples=100] plot (\x,{\x*\x+5});
\node[anchor=west,text=FigBirth] at (4.10,18.2) {$\widetilde{\lambda}(x)$};
\node[anchor=west,text=FigDeath] at (4.10,21.5) {$\widetilde{\mu}(x)$};
\fill[FigBirth] (0,2) circle[radius=1.55pt];
\fill[FigDeath] (0,5) circle[radius=1.55pt];
\fill[FigInk] (1,6) circle[radius=1.65pt];
\fill[FigInk] (3,14) circle[radius=1.65pt];
\node[below left,inner sep=2pt,text=FigInk] at (0,0) {$0$};
\node[below,fill=white,inner sep=2pt,text=FigInk] at (1,0) {$x_1$};
\node[below,fill=white,inner sep=2pt,text=FigInk] at (3,0) {$x_2$};

% Aligned deterministic phase line.
\draw[FigGuide!75,line width=0.5pt] (0,-5.2)--(4.28,-5.2);
\node[anchor=east,text=FigInk] at (-0.10,-5.2) {$\dot x=V(x)$};
\draw[paper flow] (0.78,-5.2)--(0.18,-5.2);
\draw[paper flow] (1.35,-5.2)--(2.62,-5.2);
\draw[paper flow] (4.03,-5.2)--(3.28,-5.2);
\node[text=FigDeath,font=\footnotesize] at (0.5,-3.65) {$V<0$};
\node[text=FigBirth,font=\footnotesize] at (2,-3.65) {$V>0$};
\node[text=FigDeath,font=\footnotesize] at (3.65,-3.65) {$V<0$};
\fill[FigInk] (0,-5.2) circle[radius=2.1pt];
\draw[FigInk,fill=white,line width=0.8pt] (1,-5.2) circle[radius=2.1pt];
\fill[FigInk] (3,-5.2) circle[radius=2.1pt];
\node[below,text=FigInk,font=\footnotesize,inner sep=3pt] at (0,-5.2) {stable};
\node[below,text=FigInk,font=\footnotesize,inner sep=3pt] at (1,-5.2) {unstable};
\node[below,text=FigInk,font=\footnotesize,inner sep=3pt] at (3,-5.2) {stable};
\end{tikzpicture}
\caption{Per-capita rates and the deterministic phase line in the strong
Allee regime. The birth rate (solid) and death rate (dashed) intersect
at $x_1$ and $x_2$. On the phase line, arrows indicate the drift
direction; filled circles mark the stable equilibria $0,x_2$, and
the open circle marks the unstable threshold $x_1$.}
\label{fig:lambdamu}
\end{figure}

The function $H$ defined in \eqref{Hx} satisfies
\[
H'(x)
=
\ln\frac{\widetilde{\mu}(x)}
          {\widetilde{\lambda}(x)}.
\]
In particular,
\[
H'(x_1)=H'(x_2)=0.
\]
Moreover, in the bistable regime,
\[
H''(x_1)<0,
\quad
H''(x_2)>0.
\]
Hence $x_1$ is a local maximum of $H$, whereas $x_2$ is a local
minimum. Since $H(x_2)=0$, the quantity $H(x_1)$ represents the
potential barrier separating the positive stable region from the extinction
basin.

The sign of $H(0)$ (see Figure \ref{fig:functionh}) is determined by the competition between
\[
\int_0^{x_1}
\ln\frac{\widetilde{\mu}(s)}
          {\widetilde{\lambda}(s)}\,ds
\]
and
\[
-\int_{x_1}^{x_2}
\ln\frac{\widetilde{\mu}(s)}
          {\widetilde{\lambda}(s)}\,ds.
\]

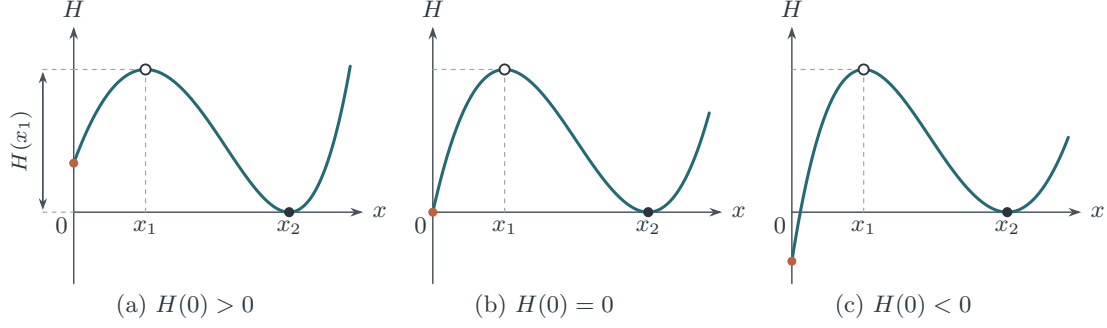
\begin{figure}[htbp]
\centering
\begin{tikzpicture}[x=0.95cm,y=1.18cm,font=\small,
  line cap=round,line join=round]
\begin{scope}[xshift=0.72cm]
\node[text=FigInk] at (1.55,-1.05) {(a) $H(0)>0$};
\draw[paper axis] (0,0)--(4.03,0) node[right,text=FigInk] {$x$};
\draw[paper axis] (0,-0.80)--(0,2.08) node[above,text=FigInk] {$H$};
\draw[paper guide] (1,0)--(1,1.6);
\draw[paper guide] (0,1.6)--(1,1.6);
\node[below left,inner sep=2pt,text=FigInk] at (0,0) {$0$};
\node[below,inner sep=3pt,text=FigInk] at (1,0) {$x_1$};
\node[below,inner sep=3pt,text=FigInk] at (3,0) {$x_2$};
\draw[FigPotential,line width=1.15pt] plot[smooth] coordinates {
(0.00000,0.55000) (0.02278,0.59913) (0.04556,0.64703) (0.06834,0.69370) (0.09112,0.73914)
(0.11391,0.78337) (0.13669,0.82637) (0.15947,0.86815) (0.18225,0.90873) (0.20503,0.94809)
(0.22781,0.98624) (0.25059,1.02318) (0.27337,1.05893) (0.29615,1.09347) (0.31893,1.12683)
(0.34172,1.15899) (0.36450,1.18997) (0.38728,1.21976) (0.41006,1.24838) (0.43284,1.27583)
(0.45562,1.30212) (0.47840,1.32724) (0.50118,1.35120) (0.52396,1.37402) (0.54675,1.39569)
(0.56953,1.41623) (0.59231,1.43563) (0.61509,1.45391) (0.63787,1.47107) (0.66065,1.48712)
(0.68343,1.50206) (0.70621,1.51591) (0.72899,1.52867) (0.75178,1.54035) (0.77456,1.55095)
(0.79734,1.56050) (0.82012,1.56898) (0.84290,1.57642) (0.86568,1.58282) (0.88846,1.58820)
(0.91124,1.59255) (0.93402,1.59590) (0.95680,1.59825) (0.97959,1.59961) (1.00000,1.60000)
(1.00237,1.59999) (1.02515,1.59941) (1.04793,1.59788) (1.07071,1.59540) (1.09349,1.59199)
(1.11627,1.58766) (1.13905,1.58242) (1.16183,1.57629) (1.18462,1.56928) (1.20740,1.56141)
(1.23018,1.55268) (1.25296,1.54311) (1.27574,1.53271) (1.29852,1.52151) (1.32130,1.50951)
(1.34408,1.49672) (1.36686,1.48318) (1.38964,1.46888) (1.41243,1.45386) (1.43521,1.43811)
(1.45799,1.42167) (1.48077,1.40455) (1.50355,1.38676) (1.52633,1.36833) (1.54911,1.34927)
(1.57189,1.32961) (1.59467,1.30935) (1.61746,1.28852) (1.64024,1.26715) (1.66302,1.24524)
(1.68580,1.22283) (1.70858,1.19993) (1.73136,1.17656) (1.75414,1.15275) (1.77692,1.12852)
(1.79970,1.10389) (1.82249,1.07888) (1.84527,1.05353) (1.86805,1.02784) (1.89083,1.00185)
(1.91361,0.97558) (1.93639,0.94906) (1.95917,0.92231) (1.98195,0.89536) (2.00473,0.86824)
(2.02751,0.84097) (2.05030,0.81357) (2.07308,0.78609) (2.09586,0.75855) (2.11864,0.73097)
(2.14142,0.70339) (2.16420,0.67583) (2.18698,0.64833) (2.20976,0.62092) (2.23254,0.59363)
(2.25533,0.56649) (2.27811,0.53953) (2.30089,0.51280) (2.32367,0.48631) (2.34645,0.46011)
(2.36923,0.43423) (2.39201,0.40871) (2.41479,0.38358) (2.43757,0.35887) (2.46036,0.33464)
(2.48314,0.31091) (2.50592,0.28771) (2.52870,0.26510) (2.55148,0.24311) (2.57426,0.22178)
(2.59704,0.20116) (2.61982,0.18127) (2.64260,0.16217) (2.66538,0.14390) (2.68817,0.12649)
(2.71095,0.11000) (2.73373,0.09448) (2.75651,0.07995) (2.77929,0.06647) (2.80207,0.05410)
(2.82485,0.04286) (2.84763,0.03282) (2.87041,0.02402) (2.89320,0.01651) (2.91598,0.01033)
(2.93876,0.00555) (2.96154,0.00222) (2.98432,0.00037) (3.00000,0.00000) (3.00710,0.00008)
(3.02988,0.00138) (3.05266,0.00435) (3.07544,0.00902) (3.09822,0.01546) (3.12101,0.02373)
(3.14379,0.03388) (3.16657,0.04597) (3.18935,0.06006) (3.21213,0.07621) (3.23491,0.09448)
(3.25769,0.11493) (3.28047,0.13763) (3.30325,0.16264) (3.32604,0.19002) (3.34882,0.21984)
(3.37160,0.25216) (3.39438,0.28706) (3.41716,0.32459) (3.43994,0.36484) (3.46272,0.40786)
(3.48550,0.45373) (3.50828,0.50252) (3.53107,0.55431) (3.55385,0.60916) (3.57663,0.66715)
(3.59941,0.72836) (3.62219,0.79286) (3.64497,0.86073) (3.66775,0.93205) (3.69053,1.00689)
(3.71331,1.08535) (3.73609,1.16749) (3.75888,1.25341) (3.78166,1.34319) (3.80444,1.43690)
(3.82722,1.53464) (3.85000,1.63650)
};
\fill[FigDeath] (0,0.55000) circle[radius=1.7pt];
\draw[FigInk,fill=white,line width=0.7pt] (1,1.6) circle[radius=1.9pt];
\fill[FigInk] (3,0) circle[radius=1.9pt];
\draw[paper guide] (-0.43,0)--(0,0);
\draw[paper guide] (-0.43,1.6)--(0,1.6);
\draw[FigInk!85,line width=0.65pt,
  {Stealth[length=1.55mm,width=1.1mm]}-{Stealth[length=1.55mm,width=1.1mm]}]
  (-0.43,0.025)--(-0.43,1.575);
\node[rotate=90,text=FigInk,font=\footnotesize] at (-0.73,0.8) {$H(x_1)$};
\end{scope}
\begin{scope}[xshift=5.47cm]
\node[text=FigInk] at (1.55,-1.05) {(b) $H(0)=0$};
\draw[paper axis] (0,0)--(4.03,0) node[right,text=FigInk] {$x$};
\draw[paper axis] (0,-0.80)--(0,2.08) node[above,text=FigInk] {$H$};
\draw[paper guide] (1,0)--(1,1.6);
\draw[paper guide] (0,1.6)--(1,1.6);
\node[below left,inner sep=2pt,text=FigInk] at (0,0) {$0$};
\node[below,inner sep=3pt,text=FigInk] at (1,0) {$x_1$};
\node[below,inner sep=3pt,text=FigInk] at (3,0) {$x_2$};
\draw[FigPotential,line width=1.15pt] plot[smooth] coordinates {
(0.00000,-0.00000) (0.02278,0.08077) (0.04556,0.15908) (0.06834,0.23495) (0.09112,0.30842)
(0.11391,0.37951) (0.13669,0.44825) (0.15947,0.51467) (0.18225,0.57880) (0.20503,0.64066)
(0.22781,0.70029) (0.25059,0.75771) (0.27337,0.81296) (0.29615,0.86605) (0.31893,0.91702)
(0.34172,0.96589) (0.36450,1.01270) (0.38728,1.05747) (0.41006,1.10024) (0.43284,1.14102)
(0.45562,1.17985) (0.47840,1.21676) (0.50118,1.25177) (0.52396,1.28492) (0.54675,1.31623)
(0.56953,1.34572) (0.59231,1.37344) (0.61509,1.39940) (0.63787,1.42364) (0.66065,1.44618)
(0.68343,1.46705) (0.70621,1.48628) (0.72899,1.50391) (0.75178,1.51994) (0.77456,1.53443)
(0.79734,1.54738) (0.82012,1.55884) (0.84290,1.56883) (0.86568,1.57738) (0.88846,1.58452)
(0.91124,1.59027) (0.93402,1.59466) (0.95680,1.59773) (0.97959,1.59950) (1.00000,1.60000)
(1.00237,1.59999) (1.02515,1.59925) (1.04793,1.59729) (1.07071,1.59414) (1.09349,1.58984)
(1.11627,1.58441) (1.13905,1.57787) (1.16183,1.57027) (1.18462,1.56162) (1.20740,1.55195)
(1.23018,1.54130) (1.25296,1.52969) (1.27574,1.51715) (1.29852,1.50370) (1.32130,1.48939)
(1.34408,1.47422) (1.36686,1.45824) (1.38964,1.44148) (1.41243,1.42395) (1.43521,1.40569)
(1.45799,1.38672) (1.48077,1.36708) (1.50355,1.34680) (1.52633,1.32589) (1.54911,1.30440)
(1.57189,1.28234) (1.59467,1.25975) (1.61746,1.23666) (1.64024,1.21309) (1.66302,1.18907)
(1.68580,1.16463) (1.70858,1.13980) (1.73136,1.11461) (1.75414,1.08909) (1.77692,1.06325)
(1.79970,1.03714) (1.82249,1.01078) (1.84527,0.98420) (1.86805,0.95742) (1.89083,0.93049)
(1.91361,0.90341) (1.93639,0.87623) (1.95917,0.84897) (1.98195,0.82165) (2.00473,0.79432)
(2.02751,0.76699) (2.05030,0.73970) (2.07308,0.71246) (2.09586,0.68532) (2.11864,0.65830)
(2.14142,0.63143) (2.16420,0.60473) (2.18698,0.57824) (2.20976,0.55198) (2.23254,0.52598)
(2.25533,0.50027) (2.27811,0.47488) (2.30089,0.44983) (2.32367,0.42516) (2.34645,0.40089)
(2.36923,0.37706) (2.39201,0.35368) (2.41479,0.33080) (2.43757,0.30842) (2.46036,0.28660)
(2.48314,0.26535) (2.50592,0.24470) (2.52870,0.22468) (2.55148,0.20531) (2.57426,0.18664)
(2.59704,0.16868) (2.61982,0.15146) (2.64260,0.13502) (2.66538,0.11937) (2.68817,0.10456)
(2.71095,0.09060) (2.73373,0.07753) (2.75651,0.06537) (2.77929,0.05415) (2.80207,0.04391)
(2.82485,0.03466) (2.84763,0.02644) (2.87041,0.01928) (2.89320,0.01320) (2.91598,0.00823)
(2.93876,0.00441) (2.96154,0.00175) (2.98432,0.00029) (3.00000,0.00000) (3.00710,0.00006)
(3.02988,0.00108) (3.05266,0.00339) (3.07544,0.00700) (3.09822,0.01196) (3.12101,0.01828)
(3.14379,0.02600) (3.16657,0.03514) (3.18935,0.04574) (3.21213,0.05782) (3.23491,0.07141)
(3.25769,0.08653) (3.28047,0.10322) (3.30325,0.12151) (3.32604,0.14142) (3.34882,0.16298)
(3.37160,0.18623) (3.39438,0.21118) (3.41716,0.23786) (3.43994,0.26632) (3.46272,0.29656)
(3.48550,0.32863) (3.50828,0.36255) (3.53107,0.39835) (3.55385,0.43605) (3.57663,0.47569)
(3.59941,0.51729) (3.62219,0.56089) (3.64497,0.60650) (3.66775,0.65417) (3.69053,0.70391)
(3.71331,0.75576) (3.73609,0.80974) (3.75888,0.86588) (3.78166,0.92422) (3.80444,0.98477)
(3.82722,1.04757) (3.85000,1.11265)
};
\fill[FigDeath] (0,0.00000) circle[radius=1.7pt];
\draw[FigInk,fill=white,line width=0.7pt] (1,1.6) circle[radius=1.9pt];
\fill[FigInk] (3,0) circle[radius=1.9pt];
\end{scope}
\begin{scope}[xshift=10.22cm]
\node[text=FigInk] at (1.55,-1.05) {(c) $H(0)<0$};
\draw[paper axis] (0,0)--(4.03,0) node[right,text=FigInk] {$x$};
\draw[paper axis] (0,-0.80)--(0,2.08) node[above,text=FigInk] {$H$};
\draw[paper guide] (1,0)--(1,1.6);
\draw[paper guide] (0,1.6)--(1,1.6);
\node[below left,inner sep=2pt,text=FigInk] at (0,0) {$0$};
\node[below,inner sep=3pt,text=FigInk] at (1,0) {$x_1$};
\node[below,inner sep=3pt,text=FigInk] at (3,0) {$x_2$};
\draw[FigPotential,line width=1.15pt] plot[smooth] coordinates {
(0.00000,-0.55000) (0.02278,-0.43563) (0.04556,-0.32520) (0.06834,-0.21862) (0.09112,-0.11584)
(0.11391,-0.01679) (0.13669,0.07861) (0.15947,0.17042) (0.18225,0.25870) (0.20503,0.34352)
(0.22781,0.42495) (0.25059,0.50305) (0.27337,0.57789) (0.29615,0.64953) (0.31893,0.71802)
(0.34172,0.78344) (0.36450,0.84584) (0.38728,0.90529) (0.41006,0.96184) (0.43284,1.01556)
(0.45562,1.06650) (0.47840,1.11472) (0.50118,1.16028) (0.52396,1.20324) (0.54675,1.24366)
(0.56953,1.28158) (0.59231,1.31708) (0.61509,1.35019) (0.63787,1.38098) (0.66065,1.40950)
(0.68343,1.43580) (0.70621,1.45994) (0.72899,1.48197) (0.75178,1.50194) (0.77456,1.51990)
(0.79734,1.53590) (0.82012,1.55000) (0.84290,1.56224) (0.86568,1.57267) (0.88846,1.58134)
(0.91124,1.58830) (0.93402,1.59360) (0.95680,1.59729) (0.97959,1.59940) (1.00000,1.60000)
(1.00237,1.59999) (1.02515,1.59911) (1.04793,1.59679) (1.07071,1.59309) (1.09349,1.58805)
(1.11627,1.58170) (1.13905,1.57411) (1.16183,1.56530) (1.18462,1.55532) (1.20740,1.54422)
(1.23018,1.53203) (1.25296,1.51880) (1.27574,1.50456) (1.29852,1.48937) (1.32130,1.47325)
(1.34408,1.45624) (1.36686,1.43839) (1.38964,1.41974) (1.41243,1.40032) (1.43521,1.38016)
(1.45799,1.35932) (1.48077,1.33782) (1.50355,1.31570) (1.52633,1.29300) (1.54911,1.26975)
(1.57189,1.24599) (1.59467,1.22176) (1.61746,1.19708) (1.64024,1.17199) (1.66302,1.14653)
(1.68580,1.12073) (1.70858,1.09462) (1.73136,1.06823) (1.75414,1.04160) (1.77692,1.01476)
(1.79970,0.98774) (1.82249,0.96057) (1.84527,0.93328) (1.86805,0.90591) (1.89083,0.87847)
(1.91361,0.85101) (1.93639,0.82356) (1.95917,0.79613) (1.98195,0.76876) (2.00473,0.74148)
(2.02751,0.71431) (2.05030,0.68729) (2.07308,0.66043) (2.09586,0.63378) (2.11864,0.60734)
(2.14142,0.58116) (2.16420,0.55525) (2.18698,0.52965) (2.20976,0.50437) (2.23254,0.47944)
(2.25533,0.45489) (2.27811,0.43074) (2.30089,0.40702) (2.32367,0.38375) (2.34645,0.36094)
(2.36923,0.33863) (2.39201,0.31685) (2.41479,0.29560) (2.43757,0.27491) (2.46036,0.25481)
(2.48314,0.23532) (2.50592,0.21645) (2.52870,0.19823) (2.55148,0.18068) (2.57426,0.16383)
(2.59704,0.14768) (2.61982,0.13227) (2.64260,0.11760) (2.66538,0.10370) (2.68817,0.09060)
(2.71095,0.07830) (2.73373,0.06683) (2.75651,0.05620) (2.77929,0.04643) (2.80207,0.03755)
(2.82485,0.02956) (2.84763,0.02249) (2.87041,0.01636) (2.89320,0.01117) (2.91598,0.00695)
(2.93876,0.00371) (2.96154,0.00147) (2.98432,0.00025) (3.00000,0.00000) (3.00710,0.00005)
(3.02988,0.00090) (3.05266,0.00281) (3.07544,0.00580) (3.09822,0.00988) (3.12101,0.01506)
(3.14379,0.02136) (3.16657,0.02879) (3.18935,0.03737) (3.21213,0.04712) (3.23491,0.05803)
(3.25769,0.07013) (3.28047,0.08343) (3.30325,0.09795) (3.32604,0.11369) (3.34882,0.13066)
(3.37160,0.14889) (3.39438,0.16838) (3.41716,0.18913) (3.43994,0.21118) (3.46272,0.23452)
(3.48550,0.25916) (3.50828,0.28513) (3.53107,0.31242) (3.55385,0.34105) (3.57663,0.37103)
(3.59941,0.40237) (3.62219,0.43507) (3.64497,0.46916) (3.66775,0.50463) (3.69053,0.54150)
(3.71331,0.57978) (3.73609,0.61947) (3.75888,0.66059) (3.78166,0.70314) (3.80444,0.74713)
(3.82722,0.79258) (3.85000,0.83948)
};
\fill[FigDeath] (0,-0.55000) circle[radius=1.7pt];
\draw[FigInk,fill=white,line width=0.7pt] (1,1.6) circle[radius=1.9pt];
\fill[FigInk] (3,0) circle[radius=1.9pt];
\end{scope}
\end{tikzpicture}
\caption{Schematic potentials for the three signs of $H(0)$, drawn on
common horizontal and vertical scales with $H(x_2)=0$. Each profile
has a local maximum at $x_1$ and a local minimum at $x_2$. The
vertical arrow in panel (a) marks the barrier
$H(x_1)-H(x_2)=H(x_1)$.}
\label{fig:functionh}
\end{figure}

\subsection{Boundary behavior and reversible weights}
\label{subsec:reversible-weights}

By the definition of $\pi_n$, we have the detailed-balance identity
\begin{equation}\label{LambdaPiMu}
\lambda_n\pi_n
=
\mu_{n+1}\pi_{n+1},
\quad n\geq1.
\end{equation}

Assumptions \textbf{(A1)}--\textbf{(A3)} imply that
\begin{equation}\label{TwoSums}
\sum_{n\geq1}\frac1{\lambda_n\pi_n}
=
\infty,
\quad
\sum_{n\geq1}\pi_n
<
\infty.
\end{equation}
The first condition implies that the process is absorbed at $0$ almost
surely, while the second one ensures that
\[
\mathbb{E}_m[T_0]<\infty,
\quad m\in\mathbb{N}^*,
\]
where $T_0$ is the absorption time.

Consider a birth--death process $(X_t^K)_{t\geq0}$ with birth rates
$\lambda=(\lambda_n)$ and death rates $\mu=(\mu_n)$, where
$\lambda_n>0$ and $\mu_n>0$ hold for each $n\ge1$. Define
\[
T_n:=\inf\{t\ge0:X_t^K=n\}.
\]
Then the following lemma gives the necessary and sufficient conditions
for $(X_t^K)_{t\geq0}$ to come down from infinity, which also provides a criterion for the existence and uniqueness of
the quasi-stationary distribution.

\begin{lemma}(\cite{Bansaye2013how}, Proposition 2.3)\label{Proposition 2.3}
Assume that
\[
\sum_{n=1}^{\infty}\frac{1}{\lambda_n\pi_n}=+\infty.
\]
Then the following assertions are equivalent:
\begin{itemize}
\item[{\rm(i)}]
The process $(X_t^K)_{t\geq0}$ comes down from infinity;

\item[{\rm(ii)}]
\[
\sum_{n=1}^{\infty}
\frac{1}{\lambda_n\pi_n}
\sum_{i\ge n+1}\pi_i
<+\infty;
\]

\item[{\rm(iii)}]
\[
\sup_{k\ge0}\mathbb{E}_k[T_0]<+\infty;
\]

\item[{\rm(iv)}]
For all $a>0$, there exists $k_a\in\mathbb{N}$ such that
\[
\sup_{k\ge k_a}
\mathbb{E}_k[\exp(aT_{k_a})]
<+\infty.
\]
\end{itemize}
\end{lemma}

Moreover, Lemma~\ref{sum le M sum le CK} yields
\begin{equation}\label{QSDunique}
\sum_{n\geq1}
\frac1{\lambda_n\pi_n}
\left(
\sum_{i\geq n+1}\pi_i
\right)
<
\infty.
\end{equation}
Together with Lemma~\ref{Proposition 2.3} and
\cite[Theorem~3.2(ii)]{van1991quasi}, this condition provides the criterion
used later to establish the existence and uniqueness of the
quasi-stationary distribution.

Finally, assumption \textbf{(A3)} and the mean value theorem imply that
\begin{equation}\label{sup frac mu n+1 mu n}
\sup_{n\geq1}
\frac{\mu_{n+1}}{\mu_n}<\infty.
\end{equation}
It follows from \textbf{(A3)} that 
\begin{equation}\label{mun2pin}
\lim_{n\to\infty}\mu_n^2\pi_n=0.
\end{equation}

\subsection{Spectral framework and the zero-energy solution}
\label{subsec:spectral-zero-energy}

Let $\mathfrak{D}$ denote the space of finitely supported sequences
on $\mathbb{N}^*$, and define
$\widetilde{L}:\mathfrak{D}\to\mathfrak{D}$ by
\begin{equation}\label{tildeL}
    (\widetilde{L}u)_n
    =
    \lambda_nu_{n+1}
    +
    \mu_nu_{n-1}\mathbbm{1}_{\{n\geq2\}}
    -
    (\lambda_n+\mu_n)u_n.
\end{equation}
Consider the weighted Hilbert space
\begin{equation}\label{l2pi}
    \ell^2(\pi)
    :=
    \left\{
        u=(u_n)_{n\in\mathbb{N}^*}:
        \sum_{n=1}^{\infty}\pi_n|u_n|^2<\infty
    \right\},
\end{equation}
equipped with the inner product
\[
    \langle u,v\rangle_\pi
    :=
    \sum_{n=1}^{\infty}
    \pi_n\overline{u_n}v_n
\]
and the associated norm
\[
    \|u\|_\pi
    :=
    \langle u,u\rangle_\pi^{\frac12}.
\]

Denote by $L$ the closure of $\widetilde{L}$ in $\ell^2(\pi)$.
The next lemma gives the spectral realization used throughout the paper.
For completeness, we include a proof under the standing assumptions;
see also \cite[Theorem~3.1]{Chazottes2016}.

\begin{lemma}\label{operatorL}
The following assertions hold.
\begin{itemize}
    \item[{\rm(i)}]
    The operator $\widetilde{L}$ is symmetric on $\mathfrak{D}$
    and closable in $\ell^2(\pi)$;

    \item[{\rm(ii)}]
    Its closure $L$, with domain $\mathcal{D}(L)$, generates a
    positive contraction semigroup on $\ell^2(\pi)$; this semigroup
    agrees with the killed transition semigroup:
    \begin{equation}\label{eq:killed-semigroup-identification}
    (e^{tL}f)(n)=P_tf(n)
    =\mathbb{E}_n\!\left[f(X_t^K)\mathbbm{1}_{\{T_0>t\}}\right],
    \quad f\in\ell^2(\pi),\quad n\ge1;
    \end{equation}

    \item[{\rm(iii)}]
    The operator $L$ is self-adjoint and dissipative, with compact
    resolvent; consequently, its spectrum is discrete and can be
    ordered as
    \[
        -\rho_1>-\rho_2>-\rho_3>\cdots\to-\infty.
    \]
    In particular, the principal eigenvalue $-\rho_1$ is simple and
    strictly negative, and the corresponding eigenvector can be
    chosen strictly positive.
\end{itemize}
\end{lemma}

\begin{proof}
Fix $K\ge K_0$. The unitary map
$U:\ell^2(\pi)\to\ell^2(\mathbb{N}^*)$, $(Uu)_n=\sqrt{\pi_n}u_n$,
conjugates $-\widetilde{L}$ on finitely supported sequences to $D-A$,
where
\[
(Dv)_n=b_nv_n,\quad (Av)_n=a_nv_{n+1}+a_{n-1}v_{n-1},
\]
with $b_n=\lambda_n+\mu_n$, $a_n=\sqrt{\lambda_n\mu_{n+1}}$,
and $a_0=v_0=0$. Assumptions \textbf{(A2)}--\textbf{(A3)} give
$b_n\to\infty$, $\lambda_n/\mu_n\to0$, and
$\sup\limits_n\mu_{n+1}/\mu_n<\infty$. This yields
\[
\frac{a_n}{b_n}\to0,
\quad \frac{a_n}{b_{n+1}}\to0.
\]
Consequently, for every $\varepsilon>0$, there exists
$C_\varepsilon>0$ such that
\[
\|Av\|_2\leq \varepsilon\|Dv\|_2+C_\varepsilon\|v\|_2,
\quad v\in\mathcal D(D).
\]
The operator $A$ is relatively bounded with respect to $D$ with relative
bound zero. Moreover,
\[
\bigl(A(D+i)^{-1}v\bigr)_n
=
\frac{a_n}{b_{n+1}+i}v_{n+1}
+
\frac{a_{n-1}}{b_{n-1}+i}v_{n-1},
\]
which implies that $A(D+i)^{-1}$ is the sum of two weighted shifts.
Since
\[
\frac{a_n}{b_{n+1}+i}\to0,
\quad
\frac{a_{n-1}}{b_{n-1}+i}\to0,
\]
both weighted shifts are compact on $\ell^2(\mathbb N^*)$. Hence,  $A(D+i)^{-1}$ is compact. The Kato--Rellich theorem
(see \cite{Kato1966perturbation}) implies that $D-A$ is self-adjoint on
$\mathcal{D}(D)$ and that finitely supported sequences form a core.
Since $(D+i)^{-1}$ is compact, the resolvent identity gives compact
resolvent for $D-A$. These conclusions transfer to $-L$ under $U$.
In particular,
\begin{equation}\label{eq:generator-domain}
\mathcal{D}(L)=\left\{u\in\ell^2(\pi):
\sum_{n\ge1}\pi_n(\lambda_n+\mu_n)^2|u_n|^2<\infty\right\}.
\end{equation}

For finitely supported $u$, detailed balance gives
\[
\begin{aligned}
	\langle u,-\widetilde L u\rangle_\pi
	&=
	\mu_1\pi_1|u_1|^2
	+\sum_{n\ge1}\lambda_n\pi_n
	|u_{n+1}-u_n|^2=
	|u_1|^2
	+\sum_{n\ge1}\lambda_n\pi_n
	|u_{n+1}-u_n|^2,
\end{aligned}
\]
where we used $\mu_1\pi_1=1$.
The closure of this nonnegative quadratic form is a Dirichlet form:
normal contractions decrease each term. Thus $L$ generates a positive
contraction semigroup. A vector of zero energy must vanish at $1$
and have all successive differences equal to zero; it is therefore
the zero vector. Compact resolvent then implies $\rho_1>0$.
A ground-state eigenvector may be chosen nonnegative, since replacing
it by its absolute value does not increase the energy. The eigenvalue
recurrence and the positivity of all rates show that this eigenvector
is strictly positive. Finally, the boundary equation at $1$ and the
second-order recurrence determine any eigenvector from its first
coordinate. If that coordinate is zero, the entire vector vanishes.
Every eigenspace is consequently one-dimensional, which proves the
strict ordering of all eigenvalues.

It remains to identify the semigroup with the killed process.
First, the process is nonexplosive. Indeed, for fixed $K$ the drift
of $W(n)=n$, with $W(0)=0$, satisfies
\[
\lambda_n-\mu_n\le C_K,\quad n\ge1,
\]
for a finite $C_K\ge0$, because it is negative above $Kx_2$.
For an integer $R>n$, Dynkin's formula stopped at $t\wedge T_R$
therefore gives
\[
R\,\mathbb{P}_n(T_R\le t)
\le\mathbb{E}_n X_{t\wedge T_R}^K\le n+C_Kt.
\]
An explosion would require leaving every finite set, since all rates
are finite on a finite set. Letting $R\to\infty$ excludes explosion.

Kill the process also upon reaching $R+1$, and denote the resulting
finite transition kernel on $\{1,\ldots,R\}$ by $P_t^{(R)}$.
Its generator satisfies detailed balance with respect to $\pi$,
so $\pi_nP_t^{(R)}(n,\{j\})=\pi_jP_t^{(R)}(j,\{n\})$.
Nonexplosion and the pathwise truncation imply convergence to
$P_t(n,\{j\})$ as $R\to\infty$. Thus $P_t$ is reversible and
sub-Markovian. Jensen's inequality and reversibility give
\[
\|P_tf\|_\pi^2
\le\sum_{n,j\ge1}\pi_nP_t(n,\{j\})|f_j|^2
\le\|f\|_\pi^2.
\]
For finitely supported $f$, coordinatewise continuity at $t=0$ gives
$\langle P_tf,f\rangle_\pi\to\|f\|_\pi^2$. The contraction bound
then implies $\|P_tf-f\|_\pi\to0$. Density extends this conclusion
to every $f\in\ell^2(\pi)$, so $P_t$ is a strongly continuous
contraction semigroup.

For $f\in\mathfrak{D}$, both $f$ and $\widetilde{L}f$ are bounded
and finitely supported. Dynkin's formula, with $f(0)=0$, yields
\[
P_tf-f=\int_0^t P_s\widetilde{L}f\,ds.
\]
The integral also converges in $\ell^2(\pi)$ by strong continuity.
Consequently, the generator of $P_t$ extends $\widetilde{L}$ and
hence its closure $L$. Since $L$ is maximal dissipative, this
extension equals $L$, proving
\eqref{eq:killed-semigroup-identification}. The expectation formula
for arbitrary $f\in\ell^2(\pi)$ follows by truncation; its absolute
integrability follows from the preceding Jensen bound.
\end{proof}

Recall that the sequence $u^0=(u_n^0)_{n\in\mathbb{N}^*}$ is defined in
\eqref{u_n^0}. Since $u_{n_1(K)}^0=1$, for every
$n\geq n_1(K)$,
\[
    u_n^0
    =
    1
    +
    u_1^0
    \sum_{j=n_1(K)}^{n-1}
    \frac{1}{\lambda_j\pi_j}.
\]
Using the detailed-balance identity \eqref{LambdaPiMu}, one readily
verifies that $u^0$ is a normalized zero-energy solution of the
recurrence
\begin{equation}\label{homogeneous equation}
    \lambda_nu_{n+1}^0
    +
    \mu_nu_{n-1}^0\mathbbm{1}_{\{n\geq2\}}
    -
    (\lambda_n+\mu_n)u_n^0
    =
    0,
    \quad n\in\mathbb{N}^*.
\end{equation}
The empty-sum convention is as specified in Section~\ref{subsec:notation}.

The zero-energy solution also identifies the mixing coefficient in
Theorem~\ref{theorem of TV}. Set $N=n_2(K)$ and
$q_n=u_n^0/u_N^0$ for $0\le n\le N$. By
\eqref{homogeneous equation},
\[
Lq(n)=0\quad(1\le n<N),\quad q_0=0,\quad q_N=1,
\]
where $Lq$ denotes the birth--death expression on this finite interval.
The function $n\mapsto\mathbb{P}_n(T_N<T_0)$ solves the same finite
Dirichlet problem by first-step conditioning. The maximum principle
gives uniqueness: a harmonic function with zero boundary values
cannot have a nonzero interior maximum or minimum. Hence
\begin{equation}\label{eq:alpha-hitting}
\alpha_n(K)=\frac{u_n^0}{u_N^0}
=\mathbb{P}_n(T_N<T_0),\quad 1\le n\le N.
\end{equation}
For $n>N$, almost-sure absorption and nearest-neighbour jumps imply
$\mathbb{P}_n(T_N<T_0)=1$, consistently with the definition
$\alpha_n(K)=1$ above $N$.

\section{Principal eigenvalue and the associated eigenvector}\label{section of principal eigenvalue}

In this section, we prove Theorem~\ref{theorem of principal eigenvalue and eigenfunction} by establishing sharp asymptotic estimates for the principal eigenvalue and the corresponding positive eigenvector of the killed generator $L$. Our strategy is to analyze the associated second-order recurrence in different population regions and then match the resulting solutions near the positive stable equilibrium $n_2(K)$. More precisely, we first construct suitable perturbations of the zero-energy solution $u^0$ on the intervals below $n_2(K)$, and then obtain a solution with the prescribed behavior at infinity. Matching these solutions near $n_2(K)$ determines the principal eigenvalue and, at the same time, yields uniform estimates for the associated eigenvector.

We begin with a basic property of the zero-energy solution $u^0$.
\begin{lemma}\label{u 0 notin l2}
	Let $u^0=(u^0_n)$ be as defined in \eqref{u_n^0}. Then $u^0\notin\ell^2(\pi)$.
\end{lemma}
	\begin{proof}
		In terms of \eqref{LambdaPiMu}, it follows that
		\[
		u_n^0\ge\frac{u^0_1}{\mu_n\pi_n},\quad n\ge1,
		\]
		and hence,
		\[
		\sum_{n=1}^{N}(u_n^0)^2\pi_n\ge\sum_{n=1}^{N}\frac{(u_1^0)^2}{\mu^2_{n}\pi_{n}}.
		\]
By equation \eqref{mun2pin},
\[
\frac{1}{\mu_n^2\pi_n}\to\infty.
\]
Consequently,
\[
\|u^0\|_\pi^2
=
\sum_{n\ge1}(u_n^0)^2\pi_n
\geq
(u_1^0)^2
\sum_{n\ge1}\frac{1}{\mu_n^2\pi_n}
=\infty.
\]
Thus $u^0\notin\ell^2(\pi)$.
	\end{proof}

For sufficiently small $\rho$, we seek a sequence
$u=(u_n)_{n\in\mathbb{N}^*}$, depending on $\rho$, that solves
\begin{equation}\label{nonhomogeneous equation}
\lambda_n u_{n+1}+\mu_n u_{n-1}\mathbbm{1}_{\{n\ge2\}}-(\lambda_n+\mu_n)u_n=-\rho u_n.
\end{equation}
At $\rho=0$, the sequence $u^0$ solves this recurrence at every
positive state, while the constant sequence solves it for $n\ge2$.
For small nonzero $\rho$, we construct a single perturbation of $u^0$
on $1\le n\le n_2(K)$, normalized at $n_1(K)$, and a nearly constant
solution on $n\ge n_2(K)-1$. Both constructions use bounded operators
and convergent Neumann series. Matching the two solutions at
$n_2(K)-1$ and $n_2(K)$ gives a global solution and determines the
principal eigenvalue. We then verify membership in $\mathcal{D}(L)$.

\subsection{Piecewise construction of the sequence}
Define $m_K:=\lfloor\sqrt{K}\ln K\rfloor$. Next we derive a uniform estimate for $u^0_{n}$ for $n\in[n_1(K)+m_K+1,n_2(K)]$, where $u^0_n$ is defined in \eqref{u_n^0}. 

\begin{proposition}\label{un2K=2+O}
For all sufficiently large $K$ and $n\in[n_1(K)+m_K+1,n_2(K)]$,
\begin{equation}\label{u_n2K=2-O}
	u^0_{n}=2+O\left(\frac{(\ln K)^3}{\sqrt{K}}\right),
\end{equation}
where $u^0_n$ is defined in \eqref{u_n^0}. In particular,
\[
u^0_{n_2(K)}
=
2+O\left(\frac{(\ln K)^3}{\sqrt{K}}\right).
\]
\end{proposition}

\begin{proof}
Recall the reciprocal conductances $w_n$ from Section~\ref{subsec:notation}.
Their product representation is
\begin{equation}\label{w_n=}
	w_n:=\frac{1}{\lambda_n\pi_n}
	=\prod_{i=1}^{n}\frac{\mu_i}{\lambda_i},
\end{equation}
We use the rounding errors $\epsilon_i(K)$ specified in
Section~\ref{subsec:notation}. The proof consists of the following
two estimates.

\textbf{Claim 1:}
For all sufficiently large $K$ and
$n\in[n_1(K)+m_K,n_2(K)-1]$,
\begin{equation}\label{un sum n1K n_2K}
	\sum_{j=n_1(K)}^{n}w_j
	=
	w_{n_1(K)}
	\left(
	\frac{\sqrt{\pi K}}{\sqrt{-2h'(x_1)}}
	+O\left((\ln K)^3\right)
	\right).
\end{equation}
\textbf{Proof of Claim 1:}
Recall that $h(x_1)=0$, $h'(x_1)<0$.
Then for each $j\in\{0,1,\ldots,m_K\}$,
\[
\begin{split}
\ln\frac{w_{n_1(K)+j}}{w_{n_1(K)}}
&=
\ln\prod_{i=n_1(K)+1}^{n_1(K)+j}
\frac{\widetilde{\mu}(\frac{i}{K})}
     {\widetilde{\lambda}(\frac{i}{K})}=
\sum_{i=1}^{j}
h\left(x_1+\frac{i-\epsilon_1(K)}{K}\right)                                      \\
&=
h'(x_1)\sum_{i=1}^{j}\frac{i-\epsilon_1(K)}{K}
+\frac{h''(x_1)}{2}
 \sum_{i=1}^{j}\frac{(i-\epsilon_1(K))^2}{K^2}+o\left(
\sum_{i=1}^{j}\frac{(i-\epsilon_1(K))^2}{K^2}
\right)=:I_1+I_2+I_3,
\end{split}
\]
where the last estimate is uniform for $0\leq j\leq m_K$.
Indeed, for $0\leq r<x_1$, let
\[
\omega_{h''}(r):=
\sup_{|x-x_1|\leq r}|h''(x)-h''(x_1)|,
\]
the continuity of $h''$ implies $\omega_{h''}(r)\to0$ as $r\downarrow0$.
Taylor's formula therefore gives, for sufficiently large $K$,
\[
|I_3|
\leq \frac12\omega_{h''}\left(\frac{m_K+1}{K}\right)
       \sum_{i=1}^{j}\frac{(i-\epsilon_1(K))^2}{K^2}
\leq O(1)\omega_{h''}\left(\frac{m_K+1}{K}\right)\frac{j^3}{K^2}.
\]
The same argument controls the expansion to the left of $n_1(K)$
used in \textbf{Claim 2} below.
Next, we estimate the three terms respectively. Note that
\[
\begin{split}
I_1
&=
h'(x_1)\sum_{i=1}^{j}\frac{i-\epsilon_1(K)}{K}=
\frac{h'(x_1)j(j+1)}{2K}
-\frac{h'(x_1)j}{K}\epsilon_1(K)                                                  \\
&=
\frac{h'(x_1)}{2K}j^2
+\frac{h'(x_1)}{2K}j
-\frac{h'(x_1)\epsilon_1(K)}{K}j =
\frac{h'(x_1)}{2K}j^2
+\frac{h'(x_1)j}{2K}
 \left(1-2\epsilon_1(K)\right),
\end{split}
\]
and
\[
\begin{split}
I_2
&=
h''(x_1)\sum_{i=1}^{j}
\frac{(i-\epsilon_1(K))^2}{2K^2}=
\frac{h''(x_1)}{2K^2}\sum_{i=1}^{j}i^2
-\frac{h''(x_1)\epsilon_1(K)}{K^2}\sum_{i=1}^{j}i
+\frac{h''(x_1)}{2K^2}
 \sum_{i=1}^{j}(\epsilon_1(K))^2                                                  \\
&=
\frac{h''(x_1)}{12K^2}j(j+1)(2j+1)
-\frac{h''(x_1)\epsilon_1(K)}{2K^2}j(j+1)
+\frac{h''(x_1)(\epsilon_1(K))^2}{2K^2}j                                         \\
&=
\frac{h''(x_1)j^3}{6K^2}
+\frac{h''(x_1)}{K^2}
\left[
\left(\frac14-\frac{\epsilon_1(K)}2\right)j^2
+
\left(
\frac1{12}-\frac{\epsilon_1(K)}2
+\frac{(\epsilon_1(K))^2}{2}
\right)j
\right].
\end{split}
\]
Moreover,
\[
I_3=o\left(\frac{j^3}{K^2}\right)
\]
uniformly for $0\leq j\leq m_K$. Then
\[
\begin{split}
\ln\frac{w_{n_1(K)+j}}{w_{n_1(K)}}
={}&
\frac{h'(x_1)}{2K}j^2
+\frac{h'(x_1)j}{2K}
 \left(1-2\epsilon_1(K)\right)
+\frac{h''(x_1)j^3}{6K^2}                                                        \\
&+
\frac{h''(x_1)}{K^2}
\left[
\left(\frac14-\frac{\epsilon_1(K)}2\right)j^2
+
\left(
\frac1{12}-\frac{\epsilon_1(K)}2
+\frac{(\epsilon_1(K))^2}{2}
\right)j
\right]
+o\left(\frac{j^3}{K^2}\right).
\end{split}
\]
Recall that $m_K=\lfloor\sqrt{K}\ln K\rfloor$, thus
\[
w_{n_1(K)+m_K}
=
w_{n_1(K)}
\exp\left(
\frac{h'(x_1)}{2}(\ln K)^2
+O\left(\frac{(\ln K)^3}{\sqrt{K}}\right)
\right).
\]
Moreover, by \textbf{(A2)},
\[
\frac{w_{n+1}}{w_n}
=
\frac{\mu_{n+1}}{\lambda_{n+1}}
=
\frac{\widetilde\mu((n+1)/K)}
{\widetilde\lambda((n+1)/K)}
<1,\quad
x_1<\frac{n+1}{K}<x_2.
\]
Hence $(w_n)$ is decreasing between $n_1(K)$ and $n_2(K)$, with at most an off-by-one error at the endpoints due to integer rounding. For all sufficiently large $K$ and $m_K<j\le n_2(K)-n_1(K)$,
\[
w_{n_1(K)+j}\leq w_{n_1(K)+m_K}.
\]
Then, for
$n\in[m_K+1,n_2(K)-n_1(K)-1]$ and constant $0<c'<-h'(x_1)$,
\begin{equation}\label{sum n1K+m+1 n2K-1 wj}
\begin{split}
\sum_{j=m_K+1}^{n}w_{n_1(K)+j}
&\leq
\left(n_2(K)-n_1(K)\right)w_{n_1(K)+m_K}\leq
\left(n_2(K)-n_1(K)\right)w_{n_1(K)}
\exp\left(-\frac{c'(\ln K)^2}{2}\right).
\end{split}
\end{equation}
For $0\leq j\leq\lfloor\sqrt K\rfloor$,
\[
\begin{split}
w_{n_1(K)+j}
&=
w_{n_1(K)}
\exp\left(
\frac{h'(x_1)}{2K}j^2
+O\left(\frac{j}{K}\right)
\right)                                                                          
=
w_{n_1(K)}
\exp\left(\frac{h'(x_1)}{2K}j^2\right)
\left(1+O\left(\frac{j}{K}\right)\right),
\end{split}
\]
whereas, for $\lfloor\sqrt K\rfloor\leq j\leq m_K$,
\[
\begin{split}
w_{n_1(K)+j}
&=
w_{n_1(K)}
\exp\left(
\frac{h'(x_1)}{2K}j^2
+O\left(\frac{j^3}{K^2}\right)
\right)                                                                          
=
w_{n_1(K)}
\exp\left(\frac{h'(x_1)}{2K}j^2\right)
\left(1+O\left(\frac{j^3}{K^2}\right)\right).
\end{split}
\]

We next estimate $\sum\limits_{j=0}^{m_K}w_{n_1(K)+j}$. The preceding identities reduce the problem to estimating
\[
\sum_{j=0}^{m_K}
\exp\left(\frac{h'(x_1)}{2K}j^2\right).
\]
Note that
\[
\begin{split}
\sqrt{K}\int_{0}^{\frac{m_K-1}{\sqrt{K}}}
\exp\left(\frac{h'(x_1)}{2}x^2\right)\,dx
&\leq
\sqrt{K}\sum_{j=0}^{m_K}\frac{1}{\sqrt{K}}
\exp\left(\frac{h'(x_1)j^2}{2K}\right)                                           \\
&\leq
1+\sqrt{K}\int_0^{\frac{m_K}{\sqrt{K}}}
\exp\left(\frac{h'(x_1)}{2}x^2\right)\,dx .
\end{split}
\]
For the left side of the inequality,
\[
\begin{split}
\int_{0}^{\frac{m_K-1}{\sqrt{K}}}
\exp\left(\frac{h'(x_1)}{2}x^2\right)\,dx
&=
\int_0^{\infty}
\exp\left(\frac{h'(x_1)}{2}x^2\right)\,dx-
\int^{\infty}_{\frac{m_K-1}{\sqrt{K}}}
\exp\left(\frac{h'(x_1)}{2}x^2\right)\,dx                                       \\
&=
\frac{\sqrt{\pi}}{\sqrt{-2h'(x_1)}}
-o\left(\frac{1}{\sqrt{K}}\right).
\end{split}
\]
For the right side of the inequality,
\[
\begin{split}
\int_0^{\frac{m_K}{\sqrt{K}}}
\exp\left(\frac{h'(x_1)}{2}x^2\right)\,dx
&=
\int_0^{\infty}
\exp\left(\frac{h'(x_1)}{2}x^2\right)\,dx-
\int^{\infty}_{\frac{m_K}{\sqrt{K}}}
\exp\left(\frac{h'(x_1)}{2}x^2\right)\,dx                                       \\
&=
\frac{\sqrt{\pi}}{\sqrt{-2h'(x_1)}}
-o\left(\frac{1}{\sqrt{K}}\right).
\end{split}
\]
Then
\[
\begin{split}
\sum_{j=0}^{m_K}
\exp\left(\frac{h'(x_1)}{2K}j^2\right)
&=
\sqrt{K}\sum_{j=0}^{m_K}\frac{1}{\sqrt{K}}
\exp\left(\frac{h'(x_1)j^2}{2K}\right)=
\frac{\sqrt{\pi K}}{\sqrt{-2h'(x_1)}}
+O(1).
\end{split}
\]
Thus,
\[
\begin{split}
\sum_{j=0}^{m_K}w_{n_1(K)+j}
&=
\sum_{j=0}^{\lfloor\sqrt{K}\rfloor}w_{n_1(K)+j}
+\sum_{j=\lfloor\sqrt{K}\rfloor+1}^{m_K}w_{n_1(K)+j}                             \\
&=
w_{n_1(K)}
\sum_{j=0}^{\lfloor\sqrt{K}\rfloor}
\exp\left(\frac{h'(x_1)}{2K}j^2\right)
\left(1+O\left(\frac{j}{K}\right)\right)                              \\
&\quad+
w_{n_1(K)}
\sum_{j=\lfloor\sqrt{K}\rfloor+1}^{m_K}
\exp\left(\frac{h'(x_1)}{2K}j^2\right)
\left(1+O\left(\frac{j^3}{K^2}\right)\right)                          \\
&=
w_{n_1(K)}
\sum_{j=0}^{m_K}
\exp\left(\frac{h'(x_1)}{2K}j^2\right)                                          \\
&\quad+
w_{n_1(K)}
\sum_{j=0}^{\lfloor\sqrt{K}\rfloor}
\exp\left(\frac{h'(x_1)}{2K}j^2\right)
O\left(\frac{j}{K}\right)                                             \\
&\quad+
w_{n_1(K)}
\sum_{j=\lfloor\sqrt{K}\rfloor+1}^{m_K}
\exp\left(\frac{h'(x_1)}{2K}j^2\right)
O\left(\frac{j^3}{K^2}\right).
\end{split}
\]
Note that the order of
\[
\max\left\{
\frac{j}{K},\frac{j^3}{K^2}
\right\}
\]
is at most $\frac{(\ln K)^3}{\sqrt{K}}$ for
$j\in\{0,1,\ldots,m_K\}$. Then
\begin{equation}\label{sum n1K n1K+m}
\begin{split}
\sum_{j=0}^{m_K}w_{n_1(K)+j}
&=
w_{n_1(K)}
\left(
\frac{\sqrt{\pi K}}{\sqrt{-2h'(x_1)}}
+O(1)
\right)
\left(
1+O\left(\frac{(\ln K)^3}{\sqrt{K}}\right)
\right)                                                                          \\
&=
w_{n_1(K)}
\left(
\frac{\sqrt{\pi K}}{\sqrt{-2h'(x_1)}}
+O\left((\ln K)^3\right)
\right).
\end{split}
\end{equation}
It follows from \eqref{sum n1K+m+1 n2K-1 wj} and
\eqref{sum n1K n1K+m} that, for all sufficiently large $K$ and
$n\in[n_1(K)+m_K,n_2(K)-1]$,
\[
\begin{split}
\sum_{j=n_1(K)}^{n}w_j
&=
\sum_{j=0}^{m_K}w_{n_1(K)+j}
+\sum_{j=m_K+1}^{n-n_1(K)}w_{n_1(K)+j} =
w_{n_1(K)}
\left(
\frac{\sqrt{\pi K}}{\sqrt{-2h'(x_1)}}
+O\left((\ln K)^3\right)
\right).
\end{split}
\]

\textbf{Claim 2:}
For all sufficiently large $K$, one has
\begin{equation}\label{un sum 1 n1K}
	\sum_{j=1}^{n_1(K)-1}w_j
	=
	w_{n_1(K)}
	\left(
	\frac{\sqrt{\pi K}}{\sqrt{-2h'(x_1)}}
	+O\left((\ln K)^3\right)
	\right).
\end{equation}
\textbf{Proof of Claim 2:}
By \textbf{(A2)}, the sequence $(w_n)$ is
increasing below the Allee threshold, since
\[
\frac{w_{n+1}}{w_n}
=
\frac{\widetilde\mu((n+1)/K)}
{\widetilde\lambda((n+1)/K)}
>1,\quad \quad
\frac{n+1}{K}<x_1.
\]
Together with the local expansion used in \textbf{Claim 1}, this gives,
for $m_K<j\le n_1(K)-1$,
\[
w_{n_1(K)-j}\le w_{n_1(K)-m_K}.
\]
Then there exists 
$0<c'<\frac{-h'(x_1)}{4}$ such that
\begin{equation}\label{sum 1 m wj}
\begin{split}
\sum_{j=m_K+1}^{n_1(K)-1}w_{n_1(K)-j}
&\leq
n_1(K)w_{n_1(K)}
\exp\left(
\frac{h'(x_1)}{2}(\ln K)^2
+O\left(\frac{(\ln K)^3}{\sqrt{K}}\right)
\right)                                                                          \\
&\leq
n_1(K)w_{n_1(K)}
\exp\left(-\frac{c'(\ln K)^2}{2}\right).
\end{split}
\end{equation}
For $0\leq j\leq\lfloor\sqrt K\rfloor$,
\[
\begin{split}
w_{n_1(K)-j}
&=
w_{n_1(K)}
\exp\left(
\frac{h'(x_1)}{2K}j^2
+O\left(\frac{j}{K}\right)
\right)                                                                          
=
w_{n_1(K)}
\exp\left(\frac{h'(x_1)}{2K}j^2\right)
\left(1+O\left(\frac{j}{K}\right)\right),
\end{split}
\]
whereas, for $\lfloor\sqrt K\rfloor\leq j\leq m_K$,
\[
\begin{split}
w_{n_1(K)-j}
&=
w_{n_1(K)}
\exp\left(
\frac{h'(x_1)}{2K}j^2
+O\left(\frac{j^3}{K^2}\right)
\right)                                                                          
=
w_{n_1(K)}
\exp\left(\frac{h'(x_1)}{2K}j^2\right)
\left(1+O\left(\frac{j^3}{K^2}\right)\right).
\end{split}
\]
In order to estimate
$\sum\limits_{j=1}^{m_K}w_{n_1(K)-j}$,
it remains to evaluate
\[
\sum_{j=1}^{m_K}
\exp\left(\frac{h'(x_1)}{2K}j^2\right).
\]
By the same argument as in the proof of \textbf{Claim 1},
\[
\begin{split}
\sum_{j=1}^{m_K}
\exp\left(\frac{h'(x_1)}{2K}j^2\right)
&=
\sqrt{K}
\sum_{j=1}^{m_K}\frac{1}{\sqrt{K}}
\exp\left(\frac{h'(x_1)j^2}{2K}\right)=
\frac{\sqrt{\pi K}}{\sqrt{-2h'(x_1)}}
+O(1).
\end{split}
\]
It follows that
\[
\begin{split}
\sum_{j=1}^{m_K}w_{n_1(K)-j}
&=
\sum_{j=1}^{\lfloor\sqrt{K}\rfloor}w_{n_1(K)-j}
+\sum_{j=\lfloor\sqrt{K}\rfloor+1}^{m_K}w_{n_1(K)-j}                             \\
&=
w_{n_1(K)}
\sum_{j=1}^{\lfloor\sqrt{K}\rfloor}
\exp\left(\frac{h'(x_1)}{2K}j^2\right)
\left(1+O\left(\frac{j}{K}\right)\right)                              \\
&\quad+
w_{n_1(K)}
\sum_{j=\lfloor\sqrt{K}\rfloor+1}^{m_K}
\exp\left(\frac{h'(x_1)}{2K}j^2\right)
\left(1+O\left(\frac{j^3}{K^2}\right)\right)                          \\
&=
w_{n_1(K)}
\sum_{j=1}^{m_K}
\exp\left(\frac{h'(x_1)}{2K}j^2\right)                                          \\
&\quad+
w_{n_1(K)}
\sum_{j=1}^{\lfloor\sqrt{K}\rfloor}
\exp\left(\frac{h'(x_1)}{2K}j^2\right)
O\left(\frac{j}{K}\right)                                             \\
&\quad+
w_{n_1(K)}
\sum_{j=\lfloor\sqrt{K}\rfloor+1}^{m_K}
\exp\left(\frac{h'(x_1)}{2K}j^2\right)
O\left(\frac{j^3}{K^2}\right).
\end{split}
\]
Recall that the order of
\[
\max\left\{
\frac{j}{K},\frac{j^3}{K^2}
\right\}
\]
is at most $\frac{(\ln K)^3}{\sqrt{K}}$ for
$j\in\{1,\ldots,m_K\}$. Then
\begin{equation}\label{sum 1 m w n1K-j}
\begin{split}
\sum_{j=1}^{m_K}w_{n_1(K)-j}
&=
w_{n_1(K)}
\left(
\frac{\sqrt{\pi K}}{\sqrt{-2h'(x_1)}}
+O(1)
\right)
\left(
1+O\left(\frac{(\ln K)^3}{\sqrt{K}}\right)
\right)                                                                          \\
&=
w_{n_1(K)}
\left(
\frac{\sqrt{\pi K}}{\sqrt{-2h'(x_1)}}
+O\left((\ln K)^3\right)
\right).
\end{split}
\end{equation}
It follows from \eqref{sum 1 m wj} and
\eqref{sum 1 m w n1K-j} that, for each sufficiently large $K$,
\[
\begin{split}
\sum_{j=1}^{n_1(K)-1}w_j
&=
\sum_{j=1}^{m_K}w_{n_1(K)-j}
+\sum_{j=m_K+1}^{n_1(K)-1}w_{n_1(K)-j} =
w_{n_1(K)}
\left(
\frac{\sqrt{\pi K}}{\sqrt{-2h'(x_1)}}
+O\left((\ln K)^3\right)
\right).
\end{split}
\]

We now complete the proof of Proposition \ref{un2K=2+O}. Combining \eqref{un sum n1K n_2K} and \eqref{un sum 1 n1K}, for each sufficiently large $K$ and $n\in[n_1(K)+m_K+1,n_2(K)]$,
	\[
	\begin{split}
		u^0_{n} =&1+u^0_1\sum_{j=n_1(K)}^{n-1}\frac{1}{\lambda_j\pi_j} =\frac{1+\sum\limits_{j=1}^{n-1}\frac{1}{\lambda_j\pi_j}}{1+\sum\limits_{j=1}^{n_1(K)-1}\frac{1}{\lambda_j\pi_j}} =\frac{1+\sum\limits_{j=1}^{n-1}w_j}{1+\sum\limits_{j=1}^{n_1(K)-1}w_j}\\
		=&\frac{1+2w_{n_1(K)}\left(\frac{\sqrt{\pi K}}{\sqrt{-2h'(x_1)}}+O\left((\ln K)^3\right)\right)}{1+w_{n_1(K)}\left(\frac{\sqrt{\pi K}}{\sqrt{-2h'(x_1)}}+O\left((\ln K)^3\right)\right)}\\
=& 2+\frac{-1+w_{n_1(K)}O((\ln K)^3)}{1+w_{n_1(K)}\left(\frac{\sqrt{\pi K}}{\sqrt{-2h'(x_1)}}+O\left((\ln K)^3\right)\right)}.
	\end{split}
	\]
	Note that
\[
\ln w_{n_1(K)}
=
K\sum_{i=1}^{n_1(K)}
\frac{1}{K}
\ln\frac{\widetilde\mu(i/K)}
{\widetilde\lambda(i/K)}
=
K\int_0^{\frac{n_1(K)}{K}}
\ln\frac{\widetilde\mu(s)}
{\widetilde\lambda(s)}\,ds+O(1)=
K\int_0^{x_1}
\ln\frac{\widetilde\mu(s)}
{\widetilde\lambda(s)}\,ds+O(1).
\]
By \textbf{(A2)}, $\widetilde\mu(s)>\widetilde\lambda(s)$ holds for all $0<s<x_1$, hence,
\[
\int_0^{x_1}
\ln\frac{\widetilde\mu(s)}
{\widetilde\lambda(s)}\,ds>0.
\]
Therefore,
\[
w_{n_1(K)}
=
\exp\left\{
K\int_0^{x_1}
\ln\frac{\widetilde\mu(s)}
{\widetilde\lambda(s)}\,ds+O(1)
\right\}
\]
grows exponentially in $K$. Consequently, the constant terms in
the numerator and denominator above are exponentially negligible, which yields
\[
u_n^0
=
2+O\left(\frac{(\ln K)^3}{\sqrt K}\right)
\]
uniformly for $n\in[n_1(K)+m_K+1,n_2(K)]$.
In particular, $$u^0_{n_2(K)}=2+O\left(\frac{(\ln K)^3}{\sqrt{K}}\right)$$ 
for all sufficiently large $K$.
\end{proof}

\begin{corollary}
\label{prop:transition-profile}
Let $(n_K)_{K\geq1}$ be a sequence of positive integers such that there exists $z\in\mathbb{R}$ with
\[
    z_K:=\frac{n_K-n_1(K)}{\sqrt K}\to z\quad\text{as} \quad K\to\infty.
\]
Then
\[
    \alpha_{n_K}(K)
    \to
    \mathcal{N}\left(\sqrt{-H''(x_1)}\,z\right)\quad\text{as} \quad K\to\infty,
\]
where $\mathcal{N}$ is the standard normal distribution function
defined in Section~\ref{subsec:notation}.
\end{corollary}

\begin{proof}
Since $z_K$ converges, $n_K=n_1(K)+O(\sqrt K)$. Moreover,
$n_2(K)-n_1(K)=(x_2-x_1)K+O(1)$, so that
$1\leq n_K\leq n_2(K)$ for all sufficiently large $K$.
For $1\leq n\leq n_2(K)$, the definition of $\alpha_n(K)$ gives
\[
    \alpha_n(K)
    =
    \frac{u_n^0}{u_{n_2(K)}^0}
    =
    \frac{1+\displaystyle\sum_{j=1}^{n-1}w_j}
         {1+\displaystyle\sum_{j=1}^{n_2(K)-1}w_j},
\]
where \(w_j\) is defined in Section~\ref{subsec:notation}.

We first recall the local estimate obtained in the proof of
Proposition~\ref{un2K=2+O}. Uniformly for
\[
    |\ell|\leq m_K=\lfloor\sqrt K\ln K\rfloor,
\]
we have
\[
\begin{split}
    \ln\frac{w_{n_1(K)+\ell}}{w_{n_1(K)}}
    &=
    \frac{h'(x_1)}{2K}\ell^2
    +
    O\left(
        \frac{|\ell|}{K}
        +
        \frac{|\ell|^3}{K^2}
    \right)=
    \frac{h'(x_1)}{2K}\ell^2
    +
    O\left(\frac{(\ln K)^3}{\sqrt K}\right).
\end{split}
\]
Since $h'(x_1)=H''(x_1)<0$, it follows that
\begin{equation}\label{eq:local-Gaussian-profile}
    \frac{w_{n_1(K)+\ell}}{w_{n_1(K)}}
    =
    \exp\left(\frac{H''(x_1)}{2K}\ell^2\right)
    \left(
        1+
        O\left(\frac{(\ln K)^3}{\sqrt K}\right)
    \right)
\end{equation}
uniformly for $|\ell|\leq m_K$.

The tail estimates established in \textbf{Claims 1 and 2} in Proposition \ref{un2K=2+O} give
\begin{equation}\label{eq:two-tail-negligible}
\begin{split}
    \sum_{j=1}^{n_1(K)-m_K-1}w_j
    &=
    o\bigl(w_{n_1(K)}\sqrt K\bigr),\sum_{j=n_1(K)+m_K+1}^{n_2(K)-1}w_j=
    o\bigl(w_{n_1(K)}\sqrt K\bigr).
\end{split}
\end{equation}
Moreover, $w_{n_1(K)}$ is exponentially large, and hence \(    1=o\bigl(w_{n_1(K)}\sqrt K\bigr).\)

We now estimate the numerator. Since $z_K\to z$, for all
sufficiently large $K$,
\[
    |n_K-n_1(K)|\leq m_K.
\]
Using \eqref{eq:two-tail-negligible}, we obtain
\[
\begin{split}
    1+\sum_{j=1}^{n_K-1}w_j
    &=
    \sum_{\ell=-m_K}^{n_K-n_1(K)-1}
    w_{n_1(K)+\ell}
    +
    o\bigl(w_{n_1(K)}\sqrt K\bigr).
\end{split}
\]
Therefore, by \eqref{eq:local-Gaussian-profile},
\[
\begin{split}
&\frac{1+\displaystyle\sum_{j=1}^{n_K-1}w_j}
      {w_{n_1(K)}\sqrt K}=
\frac{1}{\sqrt K}
\sum_{\ell=-m_K}^{n_K-n_1(K)-1}
\exp\left(
    \frac{H''(x_1)}{2}
    \left(\frac{\ell}{\sqrt K}\right)^2
\right)
+o(1).
\end{split}
\]
Since
\[
    \frac{m_K}{\sqrt K}\to\infty,
    \quad
    \frac{n_K-n_1(K)-1}{\sqrt K}\to z,
\]
the Riemann-sum approximation yields
\begin{equation}\label{eq:numerator-Gaussian-limit}
    \frac{1+\displaystyle\sum_{j=1}^{n_K-1}w_j}
         {w_{n_1(K)}\sqrt K}
    \to
    \int_{-\infty}^{z}
    \exp\left(\frac{H''(x_1)}{2}y^2\right)\,dy .
\end{equation}
Similarly, using both tail estimates in
\eqref{eq:two-tail-negligible}, we have
\[
\begin{split}
    1+\sum_{j=1}^{n_2(K)-1}w_j
    &=
    \sum_{\ell=-m_K}^{m_K}
    w_{n_1(K)+\ell}
    +
    o\bigl(w_{n_1(K)}\sqrt K\bigr).
\end{split}
\]
Consequently,
\[
\begin{split}
&\frac{1+\displaystyle\sum_{j=1}^{n_2(K)-1}w_j}
      {w_{n_1(K)}\sqrt K}=
\frac{1}{\sqrt K}
\sum_{\ell=-m_K}^{m_K}
\exp\left(
    \frac{H''(x_1)}{2}
    \left(\frac{\ell}{\sqrt K}\right)^2
\right)
+o(1),
\end{split}
\]
which implies that
\begin{equation}\label{eq:denominator-Gaussian-limit}
    \frac{1+\displaystyle\sum_{j=1}^{n_2(K)-1}w_j}
         {w_{n_1(K)}\sqrt K}
    \to
    \int_{-\infty}^{\infty}
    \exp\left(\frac{H''(x_1)}{2}y^2\right)\,dy .
\end{equation}

Combining \eqref{eq:numerator-Gaussian-limit} and
\eqref{eq:denominator-Gaussian-limit}, we obtain
\[
\begin{split}
    \alpha_{n_K}(K)
    &\to
    \frac{
        \displaystyle\int_{-\infty}^{z}
        \exp\left(\frac{H''(x_1)}{2}y^2\right)\,dy
    }{
        \displaystyle\int_{-\infty}^{\infty}
        \exp\left(\frac{H''(x_1)}{2}y^2\right)\,dy
    }.
\end{split}
\]
Making the change of variables
\[
    s=\sqrt{-H''(x_1)}\,y,
\]
we conclude that
\[
\begin{split}
    \alpha_{n_K}(K)
    &\to
    \frac{1}{\sqrt{2\pi}}
    \int_{-\infty}^{\sqrt{-H''(x_1)}\,z}
    e^{-\frac{s^2}{2}}\,ds=
    \mathcal{N}\left(\sqrt{-H''(x_1)}\,z\right).
\end{split}
\]
This completes the proof.
\end{proof}

We next construct a single left solution up to $n_2(K)$, with its
normalization imposed at $n_1(K)$. This construction preserves the
recurrence at the threshold, including the contribution from smaller
population states.

\begin{lemma}\label{v n-u 0 n 1lenlen1K}\label{n1K<n<n2K}
For all sufficiently large $K$, write $s=n_1(K)$ and $N=n_2(K)$.
There exists a constant $C_L>0$, independent of $K$, such that, for
$|\rho|\leq(3C_LK)^{-1}$, there is a sequence
\[
v_n(\rho)=u_n^0(1+\zeta_n(\rho)),\quad 1\leq n\leq N,
\]
which satisfies \eqref{nonhomogeneous equation} for $1\leq n<N$ and
$\zeta_s(\rho)=0$. The sequence $\zeta$ is analytic in $\rho$ on
$|\rho|<(3C_LK)^{-1}$ and satisfies
\[
\|\zeta(\rho)\|_\infty
\leq\frac{C_LK|\rho|}{1-C_LK|\rho|}\leq\frac12.
\]
In particular, $v_n(\rho)>0$ for real $\rho$ in the stated interval.
Set $\eta_n=\zeta_n$ for $1\leq n\leq s$ and
$\xi_n=\zeta_n$ for $s\leq n\leq N$. Then
\[
\|\xi'(\rho)-\Xi^0\|_\infty\leq4(C_LK)^2|\rho|,
\]
where $\Xi_s^0=0$ and
\begin{equation}\label{Xi_n^0}
\Xi_n^0=-\sum_{j=s}^{n-1}
\frac{1}{\lambda_j\pi_ju_j^0u_{j+1}^0}
\sum_{p=1}^{j}\pi_p(u_p^0)^2,
\quad s<n\leq N.
\end{equation}
Moreover, $\|\Xi^0\|_\infty\leq C_LK$, and, uniformly on
$|\rho|\leq(3C_LK)^{-1}$,
\[
\|\zeta'(\rho)\|_\infty\leq O(1) K,
\quad
\|\zeta''(\rho)\|_\infty\leq O(1) K^2.
\]
Consequently,
\[
|v_n(\rho)-u_n^0|\leq O(1) K|\rho|\Phi_n,
\quad 1\leq n\leq N.
\]
\end{lemma}

\begin{proof}
Throughout this proof, the supremum norm is taken on $\{1,\ldots,N\}$.
Put
\[
C_j=\lambda_j\pi_ju_j^0u_{j+1}^0,
\quad m_p=\pi_p(u_p^0)^2.
\]
Define a linear operator $\mathcal{T}$ by
\[
(\mathcal{T}z)_n=
\begin{cases}
\displaystyle\sum_{j=n}^{s-1}\frac{1}{C_j}
\sum_{p=1}^{j}m_pz_p,&1\leq n\leq s,\\[2mm]
\displaystyle-\sum_{j=s}^{n-1}\frac{1}{C_j}
\sum_{p=1}^{j}m_pz_p,&s\leq n\leq N.
\end{cases}
\]
Both expressions vanish at $n=s$. We first prove
$\|\mathcal{T}\|\leq C_LK$.
Since $u_{j+1}^0-u_j^0=u_1^0w_j$, telescoping gives
\[
\frac1{C_j}=\frac1{u_1^0}
\left(\frac1{u_j^0}-\frac1{u_{j+1}^0}\right).
\]
For $1\leq p\leq s$, the weights $w_0:=1,w_1,\ldots,w_s$
are nondecreasing, and therefore
\begin{equation}\label{eq:left-weight-bound}
\frac{\pi_pu_p^0}{u_1^0}
=\pi_p\sum_{r=0}^{p-1}w_r
\leq p\pi_pw_p
=\frac{p}{\lambda_p}
\leq\frac1{\widetilde{\lambda}(0)}.
\end{equation}
For $n\leq s$ and $\|z\|_\infty\leq1$, it follows that
\[
\begin{aligned}
|(\mathcal{T}z)_n|
&\leq\sum_{p=1}^{s-1}m_p
\sum_{j=\max\{n,p\}}^{s-1}\frac1{C_j}\leq\sum_{p=1}^{s-1}\frac{\pi_pu_p^0}{u_1^0}
\leq O(1) K.
\end{aligned}
\]
For $n\geq s$, split the inner sum at $p=s$. Since $u_p^0\leq1$
for $p\leq s$, \eqref{eq:left-weight-bound} gives
\begin{equation}\label{eq:left-flux-bound}
\sum_{p=1}^{s}m_p\leq O(1) K u_1^0.
\end{equation}
Also,
\[
\sum_{j=s}^{N-1}\frac1{C_j}
=\frac1{u_1^0}\left(1-\frac1{u_N^0}\right)
\leq\frac1{u_1^0}.
\]
Thus the contribution from $p\leq s$ is at most $O(1)K$.
For $s<p\leq j<N$, monotonicity of $u^0$ and
$\mu_r/\lambda_r\leq1$ for $s<r\leq N$ imply
\[
\frac{m_p}{C_j}
\leq\frac{\pi_p}{\lambda_j\pi_j}
=\frac{w_j}{\lambda_pw_p}
\leq\frac1{\lambda_p}\leq O\left(\frac{1}{K}\right).
\]
Since $N-s=O(K)$,
\[
\sum_{j=s}^{N-1}\sum_{p=s+1}^{j}\frac{m_p}{C_j}
\leq O\left(\frac{1}{K}\right)(N-s)^2\leq O(1)K.
\]
Combining this estimate with the contribution
from $p\le s$ yields
\[
\|T\|_{\infty}\le C_LK.
\]

Define
\[
\zeta=\rho(I-\rho\mathcal{T})^{-1}\mathcal{T}\mathbbm{1}.
\]
The Neumann series proves analyticity and the asserted norm estimate.
The fixed-point identity $\zeta=\rho\mathcal{T}(\mathbbm{1}+\zeta)$
gives, for $1\leq n<N$,
\[
C_n(\zeta_n-\zeta_{n+1})
=\rho\sum_{p=1}^{n}m_p(1+\zeta_p).
\]
Subtracting the identity at $n-1$, with zero flux at $n=0$, yields
\[
C_n(\zeta_n-\zeta_{n+1})
-C_{n-1}(\zeta_{n-1}-\zeta_n)\mathbbm{1}_{\{n\geq2\}}
=\rho m_n(1+\zeta_n).
\]
After division by $\pi_nu_n^0$, this is precisely
\eqref{nonhomogeneous equation} for $v_n=u_n^0(1+\zeta_n)$.
In particular, the identity holds at $n=s$; the sum over $p\leq s$
retains the flux from the left interval.

Differentiating the resolvent gives
\[
\zeta'(\rho)=(I-\rho\mathcal{T})^{-2}\mathcal{T}\mathbbm{1},
\quad
\zeta''(\rho)=2(I-\rho\mathcal{T})^{-3}\mathcal{T}^2\mathbbm{1}.
\]
These formulas imply the derivative bounds. If $t=C_LK|\rho|\leq1/3$,
then the Neumann series also gives
\[
\begin{aligned}
\|\zeta'(\rho)-\mathcal{T}\mathbbm{1}\|_\infty
&\leq C_LK\bigl((1-t)^{-2}-1\bigr)\leq4(C_LK)^2|\rho|.
\end{aligned}
\]
Restricting $\mathcal{T}\mathbbm{1}$ to $\{s,\ldots,N\}$ gives
\eqref{Xi_n^0}. Finally, $\Phi_n=u_n^0$ for $n\leq N$, so the
bound for $v-u^0$ follows from the norm estimate for $\zeta$.
\end{proof}

\begin{lemma}\label{n>n2K}
For all sufficiently large $K$, write $N=n_2(K)$. There is a constant
$C_R>0$, independent of $K$, such that for
$|\rho|\leq(3C_RK)^{-1}$ there is a bounded sequence
\[
v_n(\rho)=1+\theta_n(\rho),\quad n\geq N-1,
\]
which satisfies \eqref{nonhomogeneous equation} for $n\geq N$ and
$\theta_{N-1}=0$. The sequence $\theta$ is analytic in \(\rho\) on $|\rho|<(3C_RK)^{-1}$ and satisfies
\[
\|\theta(\rho)\|_\infty
\leq\frac{C_RK|\rho|}{1-C_RK|\rho|}\leq\frac12.
\]
Thus $1+\theta_n(\rho)>0$ for real $\rho$ in the stated interval.
Furthermore,
\[
\|\theta'(\rho)-\Theta^0\|_\infty\leq4(C_RK)^2|\rho|,
\]
where $\Theta_{N-1}^0=0$ and
\begin{equation}\label{Theta_n^0}
\Theta_n^0=\sum_{j=N-1}^{n-1}\frac1{\lambda_j\pi_j}
\sum_{p=j+1}^{\infty}\pi_p,\quad n\geq N.
\end{equation}
The norms of $\Theta^0$, $\theta'(\rho)$, and $\theta''(\rho)$
are bounded by $O(1)K$, $O(1)K$, and $O(1)K^2$, respectively.
\end{lemma}

\begin{proof}
On $\ell^\infty(\{N-1,N,\ldots\})$, define
\[
(Bz)_n=\sum_{j=N-1}^{n-1}\frac1{\lambda_j\pi_j}
\sum_{p=j+1}^{\infty}\pi_pz_p.
\]
By Lemma~\ref{sum le M sum le CK}, 
\[
\sum_{j=N}^{\infty}
\frac1{\lambda_j\pi_j}
\sum_{p=j+1}^{\infty}\pi_p\le CK ,\quad\text{for}\ \|z\|_\infty\le1.
\]
For the remaining term $j=N-1$, detailed balance gives
\[
\frac1{\lambda_{N-1}\pi_{N-1}}\sum_{p=N}^{\infty}\pi_p
=\frac1{\mu_N}
+\frac{\lambda_N}{\mu_N}\frac1{\lambda_N\pi_N}
\sum_{p=N+1}^{\infty}\pi_p\leq O(1)K,
\]
since $\lambda_N/\mu_N=1+O(1/K)$ and $\mu_N$ is of order $K$.
 Therefore, after increasing the constant if necessary,
\[
\|B\|_{\infty}\le C_RK.
\]
Set
\begin{equation}\label{theta_n}
\theta=\rho(I-\rho B)^{-1}B\mathbbm{1}.
\end{equation}
The fixed-point equation $\theta=\rho B(\mathbbm{1}+\theta)$ implies
\[
\lambda_n\pi_n(\theta_{n+1}-\theta_n)
=\rho\sum_{p=n+1}^{\infty}\pi_p(1+\theta_p).
\]
Subtracting consecutive identities proves
\[
\lambda_n(\theta_{n+1}-\theta_n)
+\mu_n(\theta_{n-1}-\theta_n)=-\rho(1+\theta_n),\quad n\geq N.
\]
The analyticity, positivity for real $\rho$, and derivative estimates
follow from the same resolvent bounds as in the preceding lemma,
with $B\mathbbm{1}=\Theta^0$.
\end{proof}

\subsection{Matching and proof of Theorem~\ref{theorem of principal eigenvalue and eigenfunction}}
We first estimate the quantities needed to match the two solutions.

\begin{lemma}\label{u-u,Theta,Xi,GK,pi}
For all sufficiently large $K$, the following estimates hold:
\begin{itemize}
\item[\rm(i)]
\[
u^0_{n_2(K)}-u^0_{n_2(K)-1}
=\sqrt{\frac{-2H''(x_1)}{\pi K}}e^{-KH(x_1)}
\left(1+O\left(\frac{(\ln K)^3}{\sqrt K}\right)\right);
\]
\item[\rm(ii)]
\[
\Xi^0_{n_2(K)-1}-\Xi^0_{n_2(K)}
=\frac1{x_2\widetilde{\lambda}(x_2)}
\sqrt{\frac{\pi}{2KH''(x_2)}}
\left(1+O\left(\frac{(\ln K)^3}{\sqrt K}\right)\right);
\]
\item[\rm(iii)]
\[
\Theta^0_{n_2(K)}
=\frac1{x_2\widetilde{\lambda}(x_2)}
\sqrt{\frac{\pi}{2KH''(x_2)}}
\left(1+O\left(\frac{(\ln K)^3}{\sqrt K}\right)\right).
\]
\end{itemize}
\end{lemma}

\begin{proof}
Write 
\[
s=n_1(K),\quad N=n_2(K),\quad m=m_K,\quad r_K=\frac{(\ln K)^3}{\sqrt K}.
\]
Let $a=-H''(x_1)>0$ and
$b_2=H''(x_2)>0$.
By \textbf{Claim~2} in Proposition~\ref{un2K=2+O},
\[
u_1^0=\frac1{w_s}\sqrt{\frac{2a}{\pi K}}(1+O(r_K)).
\]
By Lemma~\ref{Gamma n m=sqrt},
\[
\Gamma_{N,s}
=
\sqrt{
	\frac{\mu_s}{\lambda_s}
	\frac{\lambda_N}{\mu_N}}
\exp\left\{
K\left[H\left(\frac{N}{K}\right)-H\left(\frac{s}{K}\right)\right]+O\left(\frac{1}{K}\right)
\right\}.
\]
Since $s/K=x_1+O(1/K)$, $N/K=x_2+O\left(1/K\right)$,
$H'(x_1)=H'(x_2)=0$, and $H(x_2)=0$, we have
\[
H\left(\frac{N}{K}\right)-H\left(\frac{s}{K}\right)
=
-H(x_1)+O\left(\frac{1}{K^2}\right).
\]
Moreover,
\[
\frac{\lambda_s}{\mu_s}=1+O\left(\frac{1}{K}\right),
\quad
\frac{\lambda_N}{\mu_N}=1+O\left(\frac{1}{K}\right).
\]
Combining these estimates,
\[
\frac{w_{N-1}}{w_s}
=
\Gamma_{N,s}\frac{\lambda_s}{\mu_s}
=
e^{-KH(x_1)}\left(1+O\left(\frac{1}{K}\right)\right).
\]
The identity $u_N^0-u_{N-1}^0=u_1^0w_{N-1}$ then yields (i).

For (ii), \eqref{Xi_n^0} gives the exact identity
\[
\Xi^0_{N-1}-\Xi^0_N
=\frac{w_{N-1}}{u_{N-1}^0u_N^0}
\sum_{p=1}^{N-1}\pi_p(u_p^0)^2.
\]
The part with $p\leq s$ is, by \eqref{eq:left-flux-bound}, at most
\[
Ku_1^0w_{N-1}
=O(1)\sqrt K e^{-KH(x_1)}.
\]
For $s<p<N-m$, we use
\[
(u_p^0)^2\le u_{N-1}^0u_N^0.
\]
Moreover, by Lemma~\ref{Gamma n m=sqrt},
\[
\pi_p w_{N-1}
=
\frac{\Gamma_{N,p}}{\mu_p}
\le
\frac{C}{K}
\exp\left\{
-K\left[
H\left(\frac{p}{K}\right)
-H\left(\frac{N}{K}\right)
\right]
\right\}.
\]
Here the prefactor is bounded by $C/K$ uniformly on the
interval under consideration. Since $H$ is decreasing on
$[x_1,x_2]$, for $s<p<N-m$,
\[
H\left(\frac{p}{K}\right)-H\left(\frac{N}{K}\right)
\ge
H\left(\frac{N-m}{K}\right)-H\left(\frac{N}{K}\right).
\]
Using $H'(x_2)=0$, $H''(x_2)=b_2>0$,
$N/K=x_2+O(1/K)$, and
$m=\lfloor\sqrt K\ln K\rfloor$, Taylor's formula gives
\[
K\left[
H\left(\frac{N-m}{K}\right)
-H\left(\frac{N}{K}\right)
\right]
=
\frac{b_2}{2}(\ln K)^2
+O\left(\frac{(\ln K)^3}{\sqrt K}\right).
\]
Consequently, for some fixed $c>0$,
\[
\sum_{p=s+1}^{N-m-1}\pi_p w_{N-1}
\le
Ce^{-c(\ln K)^2}.
\]
For $N-m\leq p<N$, Proposition~\ref{un2K=2+O} yields
\[
\frac{(u_p^0)^2}{u_{N-1}^0u_N^0}=1+O(r_K).
\]
On the same interval, Lemma~\ref{Gamma n m=sqrt}, Taylor expansion
at $x_2$, and $N/K=x_2+O(1/K)$ imply uniformly that
\begin{equation}\label{eq:local-pi-matching}
\pi_pw_{N-1}
=\frac1{Kx_2\widetilde{\lambda}(x_2)}
\exp\left(-\frac{b_2(p-N)^2}{2K}\right)(1+O(r_K)).
\end{equation}
The elementary Gaussian-sum estimate
\[
\sum_{q=1}^{m}\exp\left(-\frac{b_2q^2}{2K}\right)
=\sqrt{\frac{\pi K}{2b_2}}+O(1)
\]
now proves (ii). The omitted terms are smaller than its asserted
error because $H(x_1)>0$.

For (iii), write
\[
\Theta_N^0=w_{N-1}\sum_{p=N}^{\infty}\pi_p.
\]
For $N\leq p\leq N+m$, the estimate
\eqref{eq:local-pi-matching} remains valid, and the sum starting
at $q=0$ has the same leading Gaussian term. For $p>N+m$,
\[
\pi_pw_{N-1}
=\frac1{\mu_p}\prod_{j=N}^{p-1}\frac{\lambda_j}{\mu_j}
\leq\frac1{\mu_p}\prod_{j=N}^{N+m}\frac{\lambda_j}{\mu_j}.
\]
Here the omitted factors are at most one. By
Lemma~\ref{sum le M sum le CK} and a Taylor expansion at $x_2$,
\[
\sum_{p>N+m}\pi_pw_{N-1}
\leq M\prod_{j=N}^{N+m}\frac{\lambda_j}{\mu_j}
\leq O(1)\exp\bigl(-c(\ln K)^2\bigr).
\]
This proves (iii).
\end{proof}

Set $C_*:=\max\{C_L,C_R\}$ and
$I_K=[-(3C_*K)^{-1},(3C_*K)^{-1}]$. The two solutions agree up to
a multiplicative constant at $N-1$ and $N$ precisely when the
following matching function vanishes.

\begin{proposition}\label{matching n_2}
Write $N=n_2(K)$ and define, for $\rho\in I_K$,
\[
f(\rho)=u_{N-1}^0(1+\xi_{N-1}(\rho))(1+\theta_N(\rho))
-u_N^0(1+\xi_N(\rho)).
\]
For all sufficiently large $K$, this function has a positive zero.
Its smallest positive zero $\widetilde{\rho}_1$ satisfies
\[
\widetilde{\rho}_1
=\frac{x_2\widetilde{\lambda}(x_2)
\sqrt{-H''(x_1)H''(x_2)}}{2\pi}e^{-KH(x_1)}
\left(1+O\left(\frac{(\ln K)^3}{\sqrt K}\right)\right).
\]
\end{proposition}

\begin{proof}
Let $\Delta_K=u_N^0-u_{N-1}^0>0$, and define
\[
G_K=u_N^0(\Theta_N^0+\Xi_{N-1}^0-\Xi_N^0),
\quad D_K=f'(0).
\]
Since $\xi(0)=\theta(0)=0$, $\xi'(0)=\Xi^0$, and
$\theta'(0)=\Theta^0$, we have
\[
f(0)=-\Delta_K,
\quad
D_K=G_K-\Delta_K(\Theta_N^0+\Xi_{N-1}^0).
\]
By Lemma~\ref{u-u,Theta,Xi,GK,pi}(ii) and (iii),
\[
\Theta_N^0+\Xi_{N-1}^0-\Xi_N^0
=
\frac{2}{x_2\widetilde\lambda(x_2)}
\sqrt{\frac{\pi}{2KH''(x_2)}}
\left(1+O\left(\frac{(\ln K)^3}{\sqrt K}\right)\right).
\]
Together with Proposition~\ref{un2K=2+O}, which gives
\[
u_N^0=2+O\left(\frac{(\ln K)^3}{\sqrt K}\right),
\]
this yields
\[
G_K=
\frac{4}{x_2\widetilde\lambda(x_2)}
\sqrt{\frac{\pi}{2KH''(x_2)}}
\left(1+O\left(\frac{(\ln K)^3}{\sqrt K}\right)\right).
\]
Moreover,
\[
|D_K-G_K|
=
\Delta_K|\Theta_N^0+\Xi_{N-1}^0|
\le O(1)K\Delta_K.
\]
Thus $D_K>0$ for all sufficiently large $K$, and both
$D_K$ and $G_K$ are of order $K^{-1/2}$. The first and second derivative bounds for $\xi$ and $\theta$,
together with their uniform boundedness, give
\[
\sup_{\rho\in I_K}|f''(\rho)|\leq O(1)K^2.
\]
Taylor's theorem yields, uniformly on $I_K$,
\[
f(\rho)=-\Delta_K+D_K\rho+O(1)K^2\rho^2,
\quad |f'(\rho)-D_K|\leq O(1)K^2|\rho|.
\]
Put
\[
R_K=\frac{2\Delta_K}{D_K}.
\]
By Lemma~\ref{u-u,Theta,Xi,GK,pi}(i) and the preceding estimate for $D_K$,
\[
\Delta_K\asymp K^{-\frac{1}{2}}e^{-KH(x_1)},
\quad
D_K\asymp K^{-\frac{1}{2}}.
\]
Hence
\[
R_K=O(1)e^{-KH(x_1)}.
\]
In particular, $[0,R_K]\subset I_K$ for all sufficiently large
$K$. Moreover, since $H(x_1)>0$,
\[
\frac{K^2R_K^2}{\Delta_K}
+
\frac{K^2R_K}{D_K}
=
O\left(1\right)K^{\frac{5}{2}}e^{-KH(x_1)}
\to0.
\]
Combining these bounds,
\[
f(R_K)=\Delta_K+O(1)K^2R_K^2>0,
\quad
\inf_{0\le\rho\le R_K}f'(\rho)
\ge\frac{D_K}{2}>0
\]
for all sufficiently large $K$.
Since $f(0)<0$, $f(R_K)>0$, and $f'(\rho)>0$ on
$[0,R_K]$, the intermediate value theorem gives a unique zero
$\widetilde\rho_1\in(0,R_K)$. In particular, $f$ has no zero in
$(0,\widetilde\rho_1)$, which implies that $\widetilde\rho_1$ is the
smallest positive zero.
At this zero, the Taylor estimate gives
\[
\left|\widetilde{\rho}_1-\frac{\Delta_K}{D_K}\right|
\leq O(1) K^{\frac{5}{2}}e^{-2KH(x_1)}.
\]
Replacing $D_K$ by $G_K$ changes the quotient by at most
$O(1)K^{\frac{3}{2}}e^{-2KH(x_1)}$. Finally, inserting the asymptotic estimates
for $\Delta_K$ and $G_K$ yields
\[
\frac{\Delta_K}{G_K}
=\frac{x_2\widetilde{\lambda}(x_2)
\sqrt{-H''(x_1)H''(x_2)}}{2\pi}e^{-KH(x_1)}
\left(1+O\left(\frac{(\ln K)^3}{\sqrt K}\right)\right).
\]
Since $H(x_1)>0$, the preceding remainders are exponentially
smaller than $e^{-KH(x_1)}\frac{(\ln K)^3}{\sqrt K}.$
They are absorbed into the stated relative error, which proves the proposition.
\end{proof}

We are now ready to prove Theorem \ref{theorem of principal eigenvalue and eigenfunction}.

\begin{proof}[Proof of Theorem \ref{theorem of principal eigenvalue and eigenfunction}]
Define the sequences $\widetilde{\phi}=(\widetilde{\phi}_n)_{n\ge1}$
and $\phi=(\phi_n)_{n\ge1}$ by 
\begin{equation}\label{phi n}
\widetilde{\phi}_n=
\begin{cases}
  u_n^0(1+\eta_n(\widetilde{\rho}_1)), & \mbox{for } 1\le n\le n_1(K), \\
  u_n^0(1+\xi_n(\widetilde{\rho}_1)), & \mbox{for } n_1(K)\le n\le n_2(K), \\
  b(1+\theta_n(\widetilde{\rho}_1)), & \mbox{for } n\ge n_2(K),
\end{cases}
\quad\text{and}\ \phi_n:=\frac{\widetilde{\phi}_n}{\widetilde{\phi}_{n_1(K)}},
\end{equation}
where $\eta_n(\widetilde{\rho}_1)$, $\xi_n(\widetilde{\rho}_1)$ and $\theta_n(\widetilde{\rho}_1)$ are defined in Lemmas \ref{v n-u 0 n 1lenlen1K} and \ref{n>n2K}, respectively, and
\[
b=\frac{u_{n_2(K)}^0(1+\xi_{n_2(K)}(\widetilde{\rho}_1))}{1+\theta_{n_2(K)}(\widetilde{\rho}_1)}.
\]

The two construction lemmas imply that $\widetilde{\phi}$ is bounded.
Since $\sum\limits_{n\ge1}\pi_n<\infty$, it belongs to $\ell^2(\pi)$.
The matching conditions imply that, for $n\ge1$,
\[
\lambda_n\widetilde{\phi}_{n+1}+\mu_n\widetilde{\phi}_{n-1}\mathbbm{1}_{\{n\ge2\}}-(\lambda_n+\mu_n)\widetilde{\phi}_n=-\widetilde{\rho}_1\widetilde{\phi}_n.
\]
Consider the sequence $(\widetilde{\phi}^{(k)})_{k\ge1}$ of elements in $\ell^2(\pi)$ defined by $\widetilde{\phi}^{(k)}_n=\widetilde{\phi}_n\mathbbm{1}_{\{n\le k\}}$. Note that for all $k\ge1$, $\widetilde{\phi}^{(k)}\in\mathfrak{D}$. A direct computation gives
\[
(L\widetilde{\phi}^{(k)})_n+\widetilde{\rho}_1\widetilde{\phi}^{(k)}_n=
\begin{cases}
  0,& \mbox{for } n<k, \\
  -\lambda_k\widetilde{\phi}_{k+1},& \mbox{for } n=k, \\
  \mu_{k+1}\widetilde{\phi}_{k},& \mbox{for } n=k+1,\\
  0,& \mbox{for } n>k+1.
\end{cases}
\]
Since $\widetilde\phi$ is bounded and
$\sum\limits_{n\ge1}\pi_n<\infty$,
\[
\|\widetilde\phi-\widetilde\phi^{(k)}\|_\pi
\to0.
\]
Moreover,
\[
\|(L+\widetilde\rho_1)\widetilde\phi^{(k)}\|_\pi^2
\le
\|\widetilde\phi\|_\infty^2
\left(
\lambda_k^2\pi_k+\mu_{k+1}^2\pi_{k+1}
\right).
\]
By \eqref{mun2pin} and $\lambda_k/\mu_k\to0$,
\[
\lambda_k^2\pi_k
=
\left(\frac{\lambda_k}{\mu_k}\right)^2
\mu_k^2\pi_k
\to0,
\quad
\mu_{k+1}^2\pi_{k+1}\to0.
\]
Consequently,
\[
\|(L+\widetilde\rho_1)\widetilde\phi^{(k)}\|_\pi
\to0.
\]
Since $L$ is closed, it follows that
$\widetilde\phi\in D(L)$ and
\[
L\widetilde\phi
=
-\widetilde\rho_1\widetilde\phi.
\]

Let $\psi$ be a strictly positive eigenvector associated with the
principal eigenvalue $-\rho_1$, whose existence follows from
Lemma~\ref{operatorL}. By Lemmas
\ref{v n-u 0 n 1lenlen1K}-\ref{n>n2K} and the definition of the matching constant
$b$, all components of $\widetilde\phi$ are strictly positive. This gives
\[
\langle\psi,\widetilde{\phi}\rangle_\pi>0.
\]
Since $L$ is self-adjoint, eigenvectors corresponding to distinct
eigenvalues are orthogonal. It follows that
\[
\widetilde{\rho}_1=\rho_1.
\]

Consequently, the sequence $\phi$ defined in \eqref{phi n} is the
strictly positive principal eigenvector normalized by
$\phi_{n_1(K)}=1$.

By Proposition~\ref{matching n_2}, for all sufficiently large $K$,
\[
\rho_1(K)
=
\frac{
x_2\widetilde{\lambda}(x_2)
\sqrt{-H''(x_1)H''(x_2)}
}{2\pi}
\exp\bigl(-KH(x_1)\bigr)
\left(
1+
O\left(
\frac{(\ln K)^3}{\sqrt K}
\right)
\right).
\]

It remains to establish the estimate for the normalized eigenvector
$\phi$. For $1\le n\le N$, Lemma~\ref{v n-u 0 n 1lenlen1K} gives
\[
\widetilde\phi_n
=
u_n^0\bigl(1+O(1)K\rho_1\bigr)
=
\Phi_n\bigl(1+O(1)K\rho_1\bigr)
\]
uniformly in $n$. For $n\ge N$, the matching constant satisfies
\[
\frac{b}{u_N^0}
=
\frac{1+\xi_N(\widetilde\rho_1)}
{1+\theta_N(\widetilde\rho_1)}
=
1+O(1) K\rho_1.
\]
Meanwhile, Lemma~\ref{n>n2K} gives
\[
1+\theta_n(\widetilde\rho_1)
=
1+O(1)K\rho_1
\]
uniformly in $n\ge N$. Since $\Phi_n=u_N^0$ on this range, we obtain
\[
\widetilde\phi_n
=
\Phi_n\bigl(1+O(1)K\rho_1\bigr)
\]
uniformly for $n\ge N$. Then there exist quantities $\varepsilon_n(K)$ such that
\[
\widetilde\phi_n
=
\Phi_n(K)\bigl(1+\varepsilon_n(K)\bigr),
\quad n\ge1,
\]
with
\[
\sup_{n\ge1}|\varepsilon_n(K)|
\le O(1) \rho_1(K)K
\]
for all sufficiently large $K$.

The construction gives $\eta_{n_1(K)}=0$ and
$u^0_{n_1(K)}=1$, which yields $\widetilde{\phi}_{n_1(K)}=1$.
This implies that $\phi=\widetilde{\phi}$, and
\[
|\phi_n(K)-\Phi_n(K)|\le O(1)\rho_1(K)K\Phi_n(K)
\]
uniformly for $n\ge1$. Finally, since $u_n^0$ is increasing and Proposition~\ref{un2K=2+O} gives
\[
u_N^0
=
2+O\left(\frac{(\ln K)^3}{\sqrt K}\right),
\]
the definition of $\Phi$ implies
\[
0<\Phi_n(K)\le u_N^0\le O(1)
\]
uniformly for $n\ge1$ and all sufficiently large $K$.
Then
\[
\sup_{n\ge1}
|\phi_n(K)-\Phi_n(K)|
\le O(1)\rho_1(K)K.
\]
This completes the proof.
\end{proof}

The same constructions also separate the remaining spectrum from the
exponentially small principal eigenvalue. The following estimate is
sufficient for the time-scale comparisons below.

\begin{lemma}[A lower bound for the remaining spectrum]
\label{lem:coarse-spectral-separation}
There exists a constant $c_{\mathrm{sp}}>0$, independent of $K$, such
that for all sufficiently large $K$,
\begin{equation}\label{eq:coarse-spectral-separation}
\rho_2(K)\geq\frac{c_{\mathrm{sp}}}{K},
\quad
\rho_2(K)-\rho_1(K)\geq\frac{c_{\mathrm{sp}}}{2K}.
\end{equation}
\end{lemma}

\begin{proof}
Write $s=n_1(K)$ and $N=n_2(K)$. Let $C_L$ and $C_R$ be the
constants in Lemmas~\ref{v n-u 0 n 1lenlen1K} and~\ref{n>n2K}.
The estimate for the operator $B$ in the proof of the latter lemma
also gives
\[
S_R(K):=\sum_{j=N-1}^{\infty}\frac1{\lambda_j\pi_j}
\sum_{p=j+1}^{\infty}\pi_p\leq C_RK,
\]
after increasing $C_R$ if necessary. Fix
$C_{\mathrm{gap}}\geq\max\{1,C_L,C_R\}$, independent of $K$.

First consider the Dirichlet operator $\mathcal{H}_R$ on
$\ell^2(\{N,N+1,\ldots\},\pi)$, defined on finitely supported
sequences by
\[
(\mathcal{H}_Rv)_n
=(\lambda_n+\mu_n)v_n-\lambda_nv_{n+1}
-\mu_nv_{n-1}\mathbbm{1}_{\{n>N\}},\quad n\geq N,
\]
and closed in this Hilbert space. Thus the boundary value is
$v_{N-1}=0$. The argument in Lemma~\ref{operatorL}, applied to the
tail interval, shows that $\mathcal{H}_R$ is self-adjoint, that
finitely supported sequences form an operator core, and that
\[
\mathcal{D}(\mathcal{H}_R)
=\left\{v\in\ell^2(\{N,N+1,\ldots\},\pi):
\sum_{n=N}^{\infty}\pi_n(\lambda_n+\mu_n)^2|v_n|^2<\infty\right\}.
\]
Indeed, deleting the first $N-1$ coordinates leaves unchanged the
vanishing relative bound of the off-diagonal part with respect to
the diagonal multiplication operator. For finitely supported $v$,
detailed balance gives
\[
\langle v,\mathcal{H}_Rv\rangle_{\pi,R}
=\mathcal{E}_R(v)
:=\sum_{j=N-1}^{\infty}\lambda_j\pi_j
|v_{j+1}-v_j|^2,
\quad v_{N-1}=0,
\]
where the inner product and norm with subscript $\pi,R$ are taken
over $\{N,N+1,\ldots\}$. By Cauchy--Schwarz,
\[
|v_n|^2
=\left|\sum_{j=N-1}^{n-1}(v_{j+1}-v_j)\right|^2
\leq\left(\sum_{j=N-1}^{n-1}\frac1{\lambda_j\pi_j}\right)
\mathcal{E}_R(v),\quad n\geq N.
\]
Multiplying by $\pi_n$, summing over $n\ge N$, and interchanging
the nonnegative sums, we obtain
\[
\begin{aligned}
	\|v\|_{\pi,R}^2
	&\le
	E_R(v)
	\sum_{n=N}^\infty
	\pi_n\sum_{j=N-1}^{n-1}\frac1{\lambda_j\pi_j}=
	E_R(v)
	\sum_{j=N-1}^\infty
	\frac1{\lambda_j\pi_j}
	\sum_{n=j+1}^\infty\pi_n=
	S_R(K)E_R(v)
	\le C_{\rm gap}K E_R(v).
\end{aligned}
\]
Passing to the operator core closure proves
\begin{equation}\label{eq:tail-dirichlet-coercivity}
\langle v,\mathcal{H}_Rv\rangle_{\pi,R}
\geq\frac1{C_{\mathrm{gap}}K}\|v\|_{\pi,R}^2,
\quad v\in\mathcal{D}(\mathcal{H}_R).
\end{equation}

It remains to verify that every eigenvalue $\rho$ of $-L$ satisfying
\[
0<\rho\leq\frac1{3C_{\mathrm{gap}}K}
\]
must equal $\rho_1(K)$. Choose a nonzero real eigenvector $u$ with
$Lu=-\rho u$, and set $y_n=u_n/u_n^0$ for $1\leq n\leq N$.
Use the zero-energy recurrence for $u^0$ and detailed balance,
and sum the resulting flux identity from $1$ to $n$, this yields
\[
\lambda_n\pi_nu_n^0u_{n+1}^0(y_n-y_{n+1})
=\rho\sum_{p=1}^{n}\pi_p(u_p^0)^2y_p,
\quad 1\leq n<N.
\]
There is no flux term at $0$, because both recurrences have the
absorbing boundary value zero. By the definition of the operator $\mathcal{T}$ in Lemma~\ref{v n-u 0 n 1lenlen1K}, summing this identity from the normalization index $s$ gives
\[
y=y_s\mathbbm{1}+\rho\mathcal{T}y,\quad\text{on }\{1,\ldots,N\},
\]
where $\mathcal{T}$ is precisely the finite-interval operator in the
proof of Lemma~\ref{v n-u 0 n 1lenlen1K}. Since
$\rho\|\mathcal{T}\|\leq1/3$, the operator $I-\rho\mathcal{T}$
is invertible. Hence
\[
y=y_s(I-\rho T)^{-1}\mathbbm{1}.
\]
If $y_s=0$, then $y_n=0$ for $1\le n\le N$. Hence,
$u_n=0$ on this interval. The second-order recurrence, together
with $\lambda_n>0$, then implies successively that $u_n=0$ for
all $n>N$, contradicting the assumption that $u$ is a nonzero
eigenvector. Thus $y_s\ne0$. Change the sign of $u$ if necessary, so that $y_s>0$.
The same lemma now gives
\[
y=y_s(I-\rho\mathcal{T})^{-1}\mathbbm{1}
=y_s(\mathbbm{1}+\zeta(\rho)),
\quad \|\zeta(\rho)\|_\infty\leq\frac12.
\]
Consequently, $u_n>0$ for $1\leq n\leq N$.

Put $b=u_{N-1}>0$. Lemma~\ref{n>n2K} constructs the bounded
positive sequence
\[
g_n=b(1+\theta_n(\rho)),\quad n\geq N-1,
\quad g_{N-1}=b,
\]
which satisfies the eigenvalue recurrence for $n\geq N$.
Its tail belongs to $\ell^2(\{N,N+1,\ldots\},\pi)$ by
\eqref{TwoSums}. We verify its membership in
$\mathcal{D}(\mathcal{H}_R)$ before invoking uniqueness. Let
$g^{(k)}_n=g_n\mathbbm{1}_{\{n\leq k\}}$ for $n\geq N$ and
$k>N$, and let $e_N$ denote the sequence supported at $N$ with
value one there. The recurrence gives
\[
\bigl\|(\mathcal{H}_R-\rho)g^{(k)}-\mu_Nb e_N\bigr\|_{\pi,R}^2
\leq\|g\|_\infty^2
\bigl(\lambda_k^2\pi_k+\mu_{k+1}^2\pi_{k+1}\bigr)
\to0,
\]
where the limit is for fixed $K$, and it follows from \eqref{mun2pin}
and $\lambda_k/\mu_k\to0$. Since $g^{(k)}\to g$ in the tail
Hilbert space and $\mathcal{H}_R$ is closed, its tail lies in
$\mathcal{D}(\mathcal{H}_R)$ and satisfies
\[
(\mathcal{H}_R-\rho)g=\mu_Nb e_N.
\]
By the domain characterization \eqref{eq:generator-domain},
the restriction $u_R=(u_n)_{n\geq N}$ also belongs to
$\mathcal{D}(\mathcal{H}_R)$, and the full recurrence gives
\[
(\mathcal{H}_R-\rho)u_R=\mu_Nb e_N.
\]
Then $u_R-g\in\mathcal{D}(\mathcal{H}_R)$ and
$(\mathcal{H}_R-\rho)(u_R-g)=0$. 
Taking the inner product with $u_R-g$ and using \eqref{eq:tail-dirichlet-coercivity}, we obtain
\[
\begin{aligned}
	0
	&=
	\langle u_R-g,(H_R-\rho)(u_R-g)\rangle_{\pi,R}\\
	&=
	\langle u_R-g,H_R(u_R-g)\rangle_{\pi,R}
	-\rho\|u_R-g\|_{\pi,R}^2\\
	&\ge
	\left(
	\frac1{C_{\rm gap}K}-\rho
	\right)
	\|u_R-g\|_{\pi,R}^2.
\end{aligned}
\]
Since
\[
\rho\le\frac1{3C_{\rm gap}K}
<
\frac1{C_{\rm gap}K},
\]
the coefficient on the right-hand side is strictly positive.
Therefore
\[
u_R=g.
\]
Thus $u$ is strictly positive on the entire state space. Let
$\phi$ denote the strictly positive principal eigenvector.
Then
\[
\langle u,\phi\rangle_\pi>0.
\]
Since $L$ is self-adjoint (see Lemma~\ref{operatorL}), eigenvectors corresponding to distinct eigenvalues are orthogonal. This yields $\rho=\rho_1(K)$.

Consequently, every eigenvalue $\rho$ satisfying
\[
0<\rho\le\frac1{3C_{\rm gap}K}
\]
must coincide with $\rho_1(K)$. Since the principal eigenvalue is
simple,
\[
\rho_2(K)\ge\frac1{3C_{\rm gap}K}.
\]
Finally, Theorem~\ref{theorem of principal eigenvalue and eigenfunction}
and $H(x_1)>0$ imply that
$\rho_1(K)\leq(6C_{\mathrm{gap}}K)^{-1}$ for all sufficiently
large $K$. Taking $c_{\mathrm{sp}}=(3C_{\mathrm{gap}})^{-1}$
proves \eqref{eq:coarse-spectral-separation}.
\end{proof}

\section{Quasi-stationary distribution and multiscale convergence}\label{QSD and Gauss}

We begin by characterizing the unique quasi-stationary
distribution in terms of the positive principal eigenvector. We next
derive a pointwise spectral estimate, control the coefficient of
the principal projection, and use hitting times to obtain uniform
bounds for the killed semigroup. Combining these ingredients proves
Theorem~\ref{theorem of TV}.

\subsection{The quasi-stationary distribution and proof of Theorem \ref{UniqueQSD}}

We use the probability and expectation notation introduced in
Section~\ref{subsec:notation}. Theorem~\ref{UniqueQSD} follows from
the principal eigenvector identity and the coming-down-from-infinity
criterion.

\begin{proof}[Proof of Theorem \ref{UniqueQSD}]
By Theorem~\ref{theorem of principal eigenvalue and eigenfunction}, $\phi$ is strictly positive and uniformly
bounded. Together with \eqref{TwoSums}, this gives
\[
0<\langle\phi,\mathbf 1\rangle_\pi
=\sum_{n\ge1}\pi_n\phi_n<\infty.
\]
Hence
\[
\nu_n=\frac{\pi_n\phi_n}
{\langle\phi,\mathbf 1\rangle_\pi},
\quad n\in\mathbb N^*,
\]
which defines a probability measure on $\mathbb N^*$.
For $A\subseteq\mathbb{N}^*$, self-adjointness and
$e^{tL}\phi=e^{-\rho_1t}\phi$ give
\[
\begin{aligned}
\mathbb{P}_\nu(X_t^K\in A,\ T_0>t)
&=\frac{\langle\phi,e^{tL}\mathbbm{1}_A\rangle_\pi}
{\langle\phi,\mathbbm{1}\rangle_\pi}=e^{-\rho_1t}\frac{\langle\phi,\mathbbm{1}_A\rangle_\pi}
{\langle\phi,\mathbbm{1}\rangle_\pi}
=e^{-\rho_1t}\nu(A).
\end{aligned}
\]
Taking $A=\mathbb{N}^*$ yields
\begin{equation}\label{PT0>t=e -rho1t}
\mathbb{P}_\nu(T_0>t)=e^{-\rho_1t}.
\end{equation}
Dividing the preceding identity by this survival probability proves quasi-stationarity. Finally, \eqref{QSDunique}, together with \cite[Theorem~3.2(ii)]{van1991quasi}, yields uniqueness of the quasi-stationary distribution.
\end{proof}

Having identified the unique quasi-stationary distribution, we next establish the spectral estimates for the killed semigroup that will be used in the proof of the multiscale convergence result.

\subsection{Spectral estimates for the killed semigroup}

Recall the potential normalization introduced in Section~\ref{subsec:notation}: 
\[
F(x)=H(0)-H(x)=-\int_0^x\ln \frac{\widetilde{\mu}(s)}{\widetilde{\lambda}(s)}\,ds.
\]
It follows that $F(0)=0$, $F'(x_1)=F'(x_2)=0$, $F''(x_1)>0$ and $F''(x_2)<0$. By assumption \textbf{(H)}, the function $F$ is three times
differentiable and satisfies
\[
\sup_{x\in\mathbb{R}_+}(1+x^2)|F'''(x)|<\infty.
\]

Recall that $\phi$ is the strictly positive principal eigenvector
associated with $-\rho_1$, normalized by
$\phi_{n_1(K)}=1$. We will use the properties of $\pi=(\pi_{n})$ frequently in this section (see Lemmas \ref{sum pi n F(x 2)>0} and \ref{sum pi n F(x 2)le0} in Appendix). We begin with a pointwise estimate obtained from the spectral decomposition of the killed semigroup.

\begin{proposition}\label{proposition7.2}
For all sufficiently large $K$, every $t\ge0$ and $n\in\mathbb{N}^*$, 
\[
\sup_{A\subseteq\mathbb{N}^*}\Bigg|P_t(n,A)-e^{-\rho_1t}\phi_n\frac{\langle\phi,\mathbbm{1}\rangle_{\pi}}{\|\phi\|_{\pi}^2}\nu(A)\Bigg| 
\le R(t,n,K,\rho_2),
\]
where
\[
R(t,n,K,\rho_2)=
\begin{cases}
\sqrt{\frac{C_1}{\pi_n}}K^{-\frac{1}{4}}e^{-\rho_2t}e^{\frac{c}{2}K},&\quad F(x_2)>0,\\
\sqrt{\frac{C_2}{\pi_n}}e^{-\rho_2t},&\quad F(x_2)\le0,
\end{cases}
\]
$c=F(x_2)$, $C_1>0$ is given in Lemma \ref{sum pi n F(x 2)>0}, and $C_2>0$ is given in Lemma \ref{sum pi n F(x 2)le0}.
\end{proposition}

\begin{proof}
Let $e_n=(\mathbbm{1}_{\{j=n\}})_{j\geq1}$ be the $n$-th
standard basis vector in $\ell^2(\pi)$. Since
\[
\langle e_n,u\rangle_\pi=\pi_nu_n
\]
for every $u\in\ell^2(\pi)$, we have
\[
u_n=\frac{1}{\pi_n}\langle e_n,u\rangle_\pi.
\]
In particular, for every $A\subseteq\mathbb{N}^*$,
\[
P_t(n,A)
=
\bigl(e^{tL}\mathbbm{1}_A\bigr)_n
=
\frac{1}{\pi_n}
\left\langle e_n,e^{tL}\mathbbm{1}_A\right\rangle_\pi.
\]

Let $Q$ be the spectral projection onto the spectral complement of
the principal eigenspace associated with $-\rho_1$. By the spectral
theorem
\cite[Theorem~V.2.10, p.~260]{Kato1966perturbation},
\[
e^{tL}\mathbbm{1}_A
=
e^{-\rho_1t}\phi
\frac{\langle\phi,\mathbbm{1}_A\rangle_\pi}
     {\|\phi\|_\pi^2}
+
e^{tL}Q\mathbbm{1}_A.
\]
Taking the $n$-th coordinate gives
\[
P_t(n,A)
=
e^{-\rho_1t}\phi_n
\frac{\langle\phi,\mathbbm{1}_A\rangle_\pi}
     {\|\phi\|_\pi^2}
+
\frac{1}{\pi_n}
\left\langle e_n,e^{tL}Q\mathbbm{1}_A\right\rangle_\pi.
\]
Since
\[
\nu(A)
=
\frac{\langle\phi,\mathbbm{1}_A\rangle_\pi}
     {\langle\phi,\mathbbm{1}\rangle_\pi},
\]
we obtain
\[
\begin{aligned}
&\left|
P_t(n,A)
-
e^{-\rho_1t}\phi_n
\frac{\langle\phi,\mathbbm{1}\rangle_\pi}
     {\|\phi\|_\pi^2}
\nu(A)
\right|=
\frac{1}{\pi_n}
\left|
\left\langle
e_n,e^{tL}Q\mathbbm{1}_A
\right\rangle_\pi
\right|.
\end{aligned}
\]

The spectrum of $L$ restricted to the range of $Q$ is contained in
$(-\infty,-\rho_2]$. Therefore,
\[
\|e^{tL}Q\|_{\ell^2(\pi)\to\ell^2(\pi)}
\leq e^{-\rho_2t}.
\]
Moreover,
\[
\|\mathbbm{1}_A\|_\pi
\leq
\|\mathbbm{1}\|_\pi
=
\left(
\sum_{j=1}^{\infty}\pi_j
\right)^{\frac{1}{2}},
\quad
\|e_n\|_\pi=\sqrt{\pi_n}.
\]
It follows from the Cauchy--Schwarz inequality that
\[
\begin{aligned}
\frac{1}{\pi_n}
\left|
\left\langle
e_n,e^{tL}Q\mathbbm{1}_A
\right\rangle_\pi
\right|
&\leq
\frac{e^{-\rho_2t}}{\pi_n}
\|e_n\|_\pi
\|\mathbbm{1}_A\|_\pi\leq
e^{-\rho_2t}
\sqrt{
\frac{1}{\pi_n}
\sum_{j=1}^{\infty}\pi_j
}.
\end{aligned}
\]
The right-hand side is independent of $A$. Hence,
\[
\begin{aligned}
\sup_{A\subseteq\mathbb{N}^*}
&\left|
P_t(n,A)
-
e^{-\rho_1t}\phi_n
\frac{\langle\phi,\mathbbm{1}\rangle_\pi}
     {\|\phi\|_\pi^2}
\nu(A)
\right|\leq
e^{-\rho_2t}
\sqrt{
\frac{1}{\pi_n}
\sum_{j=1}^{\infty}\pi_j
}.
\end{aligned}
\]
The conclusion now follows from Lemmas
\ref{sum pi n F(x 2)>0} and
\ref{sum pi n F(x 2)le0}.
\end{proof}

Before deriving a bound uniform in the initial state, we estimate
the coefficient of the principal spectral projection. The following proposition gives an approximation for this quantity, which will be used in the proof
of Theorem \ref{theorem of TV}.

\begin{proposition}\label{le mathcal O (1) rho_1(K)K}
For all sufficiently large $K$,
\[
\left|\frac{\langle\phi,\mathbbm{1}\rangle_\pi}{\|\phi\|_\pi^2}
-\frac{1}{u_{n_2(K)}^0}\right|\le O(1)K\rho_1(K).
\]
\end{proposition}
\begin{proof}
Write $k=n_1(K)$, $m=n_2(K)$, $a_K=u_m^0$, and
\[
A_\Phi=\sum_{n\ge1}\pi_n\Phi_n,\quad
B_\Phi=\sum_{n\ge1}\pi_n\Phi_n^2,\quad
A_\phi=\sum_{n\ge1}\pi_n\phi_n,\quad
B_\phi=\sum_{n\ge1}\pi_n\phi_n^2.
\]
Proposition~\ref{un2K=2+O} gives $a_K\to2$ and $0<\Phi_n\le a_K$.
The product estimate in Lemma~\ref{Gamma n m=sqrt} gives
\begin{equation}\label{1 n1(K) pi n le 1}
\frac{\pi_n}{\pi_m}
\le O(1)\frac{K}{n}\exp\left\{K\left[F\left(\frac{n}{K}\right)-F\left(\frac{m}{K}\right)\right]\right\},
\quad1\le n\le n_{\frac12}(K).
\end{equation}
For $n>m$, the same formula follows by reversing the product.
Moreover, the definition of $u^0$ and the Laplace estimate
$1+\sum\limits_{j=1}^{k-1}w_j\asymp\sqrt K\,w_k$ imply
\begin{equation}\label{Phi n le 1lenlen1K}
\Phi_n=\frac{1+\sum\limits_{j=1}^{n-1}w_j}{1+\sum\limits_{j=1}^{k-1}w_j}
\le O(1)\frac{n}{\sqrt K}
\exp\left\{K\left[F\left(\frac{k}{K}\right)-F\left(\frac{n}{K}\right)\right]\right\},
\quad1\le n\le k.
\end{equation}
Indeed, $w_j$ increases up to the threshold, so the numerator is
bounded by $O(1)nw_n$. Similarly, for $k\le n\le m$,
\begin{equation}\label{eq:Phi-flat-bound}
a_K-\Phi_n
=\frac{\sum\limits_{j=n}^{m-1}w_j}{1+\sum\limits_{j=1}^{k-1}w_j}
\le O(1)\sqrt K
\exp\left\{K\left[F\left(\frac{k}{K}\right)-F\left(\frac{n}{K}\right)\right]\right\}.
\end{equation}

Since $\Phi_n=a_K$ for $n\ge m$, we have
\[
a_KA_\Phi-B_\Phi
=
\sum_{n=1}^{m-1}
\pi_n\Phi_n(a_K-\Phi_n)
\ge0.
\]
We split the sum at $k$.

For $1\le n\le k$, using \eqref{1 n1(K) pi n le 1}, \eqref{Phi n le 1lenlen1K},
$a_K-\Phi_n\le a_K\le C$, and $F(\frac{k}{K})\le F(\frac{n}{K})$,
we obtain
\[
\begin{aligned}
	\pi_n\Phi_n(a_K-\Phi_n)
	&\le
	O(1)\pi_m
	\frac{1}{K}
	\exp\left\lbrace K\left[F\left(\frac{n}{K}\right)-F\left(\frac{m}{K}\right)\right]\right\rbrace\times
	\frac{n}{\sqrt K}
	\exp\left\lbrace K\left[F\left(\frac{k}{K}\right)-F\left(\frac{n}{K}\right)\right]\right\rbrace\\
	&\le
	O(1)\pi_m\frac{n}{K^{\frac{3}{2}}}
	\exp\left\lbrace -K\left[F\left(\frac{m}{K}\right)-F\left(\frac{k}{K}\right)\right]\right\rbrace.
\end{aligned}
\]
Therefore,
\[
\sum_{n=1}^{k}
\pi_n\Phi_n(a_K-\Phi_n)
\le
O(1)K^{\frac{1}{2}}\pi_m
\exp\left\lbrace -K\left[F\left(\frac{m}{K}\right)-F\left(\frac{k}{K}\right)\right]\right\rbrace.
\]

For $k<n<m$, using \eqref{1 n1(K) pi n le 1},
$\Phi_n\le a_K\le C$, and \eqref{eq:Phi-flat-bound}, we obtain
\[
\begin{aligned}
	\pi_n\Phi_n(a_K-\Phi_n)
	&\le
	O(1)\pi_m
	\frac{1}{K}
	\exp\left\lbrace K\left[F\left(\frac{n}{K}\right)-F\left(\frac{m}{K}\right)\right]\right\rbrace\times
	\sqrt K
	\exp\left\lbrace K\left[F\left(\frac{k}{K}\right)-F\left(\frac{n}{K}\right)\right]\right\rbrace\\
	&\le
	\frac{O(1)}{\sqrt K}\pi_m
	\exp\left\lbrace -K\left[F\left(\frac{m}{K}\right)-F\left(\frac{k}{K}\right)\right]\right\rbrace.
\end{aligned}
\]
Summing over $k<n<m$ gives
\[
\sum_{n=k+1}^{m-1}
\pi_n\Phi_n(a_K-\Phi_n)
\le
O(1)\sqrt K\,\pi_m
\exp\left\lbrace -K\left[F\left(\frac{m}{K}\right)-F\left(\frac{k}{K}\right)\right]\right\rbrace.
\]
Consequently,
\[
0\le a_KA_\Phi-B_\Phi
\le
O(1)\sqrt K\,\pi_m
\exp\left\lbrace -K\left[F\left(\frac{m}{K}\right)-F\left(\frac{k}{K}\right)\right]\right\rbrace.
\]
Since $F=H(0)-H$, $H(x_2)=0$, and
$H'(x_1)=H'(x_2)=0$, the floor errors give
\[
F\left(\frac{m}{K}\right)-F\left(\frac{k}{K}\right)
=
H(x_1)+O\left(\frac{1}{K^2}\right).
\]
Hence
\[
0\le a_KA_\Phi-B_\Phi
\le
O(1)\sqrt K\,\pi_m e^{-KH(x_1)}.
\]
For $m\le n\le m+\lfloor\sqrt K\rfloor$, the local product estimate
and Taylor expansion at $x_2$ yield $\pi_n\ge c\pi_m$.
Consequently,
\[
B_\Phi\ge c\sqrt K\pi_m,
\quad
\left|\frac{A_\Phi}{B_\Phi}-\frac{1}{a_K}\right|
\le O(1)K e^{-KH(x_1)}\le O(1)K\rho_1(K).
\]
In particular, $A_\Phi/B_\Phi$ is bounded uniformly in $K$.
Finally, the relative eigenvector estimate implies
\[
A_\phi=A_\Phi(1+O(K\rho_1)),\quad
B_\phi=B_\Phi(1+O(K\rho_1)).
\]
Consequently,
\[
\frac{A_\varphi}{B_\varphi}
=
\frac{A_\Phi}{B_\Phi}
\left(1+O(K\rho_1)\right).
\]
Combining this estimate
with the preceding bound for $A_\Phi/B_\Phi$ yields the desired
conclusion.
\end{proof}

The pointwise estimate above depends on the initial state through
$\pi_n$. The next proposition uses the projection estimate and
hitting times to obtain bounds uniform in $n\in\mathbb{N}^*$.

\begin{proposition}\label{proposition7.3}
There exists a constant $C>0$ such that, for all sufficiently large $K$,
all $t\ge0$, and all $n\ge1$,
\[
\sup_{A\subseteq\mathbb{N}^*}
\left|P_t(n,A)-e^{-\rho_1t}d_n(K)\nu(A)\right|
\le C\left[e^{-\frac{c_0}{4}t}+R(t,K,\rho_2)\right],
\]
where
\begin{equation}\label{RKrho2}
R(t,K,\rho_2):=B_K e^{-\frac{\rho_2}{2}t}.
\end{equation}
The remainder also satisfies
\begin{equation}\label{eq:uniform-spectral-refined}
\begin{split}
&\sup_{A\subseteq\mathbb{N}^*}
\left|P_t(n,A)-e^{-\rho_1t}d_n(K)\nu(A)\right|\le C e^{-\rho_1t}
\left[e^{-\frac{c_0}{8}t}+B_K e^{-\frac{\rho_2-\rho_1}{2} t}\right].
\end{split}
\end{equation}
Furthermore, $c_0>0$ and $B_K\le O(1) e^{C' K}$ for some \(C'>0\).
\end{proposition}

\begin{proof}
Write $N=n_{\frac12}(K)$ and $T=T_N$.
The birth--death hitting-time identity gives
\[
\sup_{n\ge N}\mathbb{E}_n T
=\sum_{j=N}^\infty\frac{1}{\lambda_j\pi_j}
\sum_{p=j+1}^\infty\pi_p=M_K.
\]
For $j\ge N$, detailed balance and the inequality
$\lambda_p/\mu_p\le1/2$ for $p\ge N+1$ imply
\[
\frac{1}{\lambda_j\pi_j}\sum_{p=j+1}^\infty\pi_p
\le\frac{2}{\mu_{j+1}}.
\]
Since $x\mapsto[x\widetilde{\mu}(x)]^{-1}$ is decreasing,
\[
M_K\le2\sum_{p=N+1}^\infty\frac{1}{p\widetilde{\mu}(p/K)}
\le2\int_{\frac{N}{K}}^\infty\frac{dx}{x\widetilde{\mu}(x)}.
\]
After increasing the fixed $K_0$ if necessary, $N/K\ge x_{\frac12}/2$.
Hence \textbf{(A3)} gives
\[
\sup_{K\ge K_0}M_K<\infty.
\]
In particular, $c_0>0$.

The strong Markov property and induction give
$\sup\limits_{n\ge N}\mathbb{E}_n [T^m]\le m!M_K^m$ for every integer $m\ge1$.
Thus
\[
\sup_{n\ge N}\mathbb{E}_n \left[e^{\frac{3 c_0}{4}T}\right]\le4.
\]
For sufficiently large $K$, $\rho_1\le c_0/4$. Since $\phi$ is bounded,
optional stopping of $e^{\rho_1(t\wedge T)}\phi_{X_{t\wedge T}^K}$
and the preceding exponential-moment bound imply
\[
\mathbb{E}_n\left[ e^{\rho_1T}\right]=\frac{\phi_n}{\phi_N},\quad n\ge N.
\]
The same moment bound yields
\[
\mathbb{P}_n(T\ge\frac t2)\le4e^{-\frac{3c_0}{8}t},\quad
\mathbb{E}_n\left[e^{\rho_1T}\mathbbm{1}_{\{T\ge \frac t2\}}\right]
\le4e^{-\frac{c_0}{4}t}.
\]
Theorem~\ref{theorem of principal eigenvalue and eigenfunction} and
Proposition~\ref{le mathcal O (1) rho_1(K)K} imply
$\sup\limits_{n\ge1}d_n(K)\le C$.
For $A\subseteq\mathbb{N}^*$ and $s\ge0$, write
\[
r_s(N,A):=P_s(N,A)-e^{-\rho_1s}d_N(K)\nu(A).
\]
The pointwise spectral estimate established in the proof of
Proposition~\ref{proposition7.2} gives
\[
|r_s(N,A)|\le e^{-\rho_2s}\sqrt{\frac{S_K}{\pi_N}}
\le B_K e^{-\rho_2s}.
\]
For $n\ge N$, the strong Markov property at $T$ and the identity
$d_n(K)=d_N(K)\mathbb{E}_n \left[e^{\rho_1T}\right]$ yield the exact decomposition
\[
\begin{split}
P_t(n,A)-e^{-\rho_1t}d_n(K)\nu(A)=&\mathbb{E}_n\left[
r_{t-T}(N,A)\mathbbm{1}_{\{T<\frac t2\}}\right]
+\mathbb{P}_n(X_t^K\in A,\ T\ge \frac t2)\\
&-e^{-\rho_1t}d_N(K)\nu(A)
\mathbb{E}_n\left[e^{\rho_1T}\mathbbm{1}_{\{T\ge \frac t2\}}\right].
\end{split}
\]
To retain the principal exponential factor, observe that
$\rho_2-\rho_1>0$ and $\mathbb{E}_n\left[ e^{\rho_1T}\right]\le4$. Consequently,
\[
\begin{split}
\left|\mathbb{E}_n\left[
r_{t-T}(N,A)\mathbbm{1}_{\{T<\frac t2\}}\right]\right|\le B_K e^{-\rho_1t}
\mathbb{E}_n\left[
e^{\rho_1T}e^{-(\rho_2-\rho_1)(t-T)}
\mathbbm{1}_{\{T<\frac t2\}}\right]
\le&4B_K e^{-\rho_1t}e^{-\frac{\rho_2-\rho_1}{2}t}.
\end{split}
\]
The two remaining terms are bounded using the preceding tail
estimates and $d_N(K)\le C$. We obtain, uniformly in $n\ge N$,
\[
\begin{split}
&\sup_{A\subseteq\mathbb{N}^*}
\left|P_t(n,A)-e^{-\rho_1t}d_n(K)\nu(A)\right|\le C\left[
e^{-\frac{3c_0}{8}t}+e^{-\rho_1t}e^{-\frac{c_0}{4}t}
+B_K e^{-\rho_1t}e^{-\frac{\rho_2-\rho_1}{2} t}
\right].
\end{split}
\]
This proves the first asserted bound for $n\ge N$, because
$(\rho_1+\rho_2)/2\ge\rho_2/2$. It also proves
\eqref{eq:uniform-spectral-refined}: after extracting
$e^{-\rho_1t}$, the first term decays at rate
$3c_0/8-\rho_1\ge c_0/8$, and the second decays at rate $c_0/4$.

For $n\le N$, the pointwise spectral estimate gives directly
\[
\sup_{A\subseteq\mathbb{N}^*}
\left|P_t(n,A)-e^{-\rho_1t}d_n(K)\nu(A)\right|
\le e^{-\rho_2t}\sqrt{\frac{S_K}{\pi_n}}
\le B_K e^{-\rho_2t}.
\]
Since $\rho_2>\rho_1$, this implies both asserted bounds for
$n\le N$ and completes the uniform estimates.

Finally, on the compact interval $[0,x_{\frac12}+1]$, the positive
rate functions and their ratio have fixed upper and lower bounds.
The product formula for $\pi_n$ consequently gives constants
$C_1,C_2>0$ such that
\[
\min_{1\le n\le N}\pi_n\ge C_1 K^{-1}e^{-C_2K},
\quad\sum_{n=1}^{N+1}\pi_n\le C_1^{-1}K e^{C_2K}.
\]
The tail beyond $N+1$ is geometric, since
$\pi_{n+1}/\pi_n\le1/2$ for $n\ge N+1$.
These bounds imply $B_K\le O(1) e^{C' K}$.
\end{proof}

\begin{remark}
Proposition~\ref{proposition7.2} retains the dependence on the initial
state and the decay rate $\rho_2$. Proposition~\ref{proposition7.3}
provides a bound uniform in the initial state at the cost of a larger
amplitude and the rate $\rho_2/2$.
\end{remark}

We now rewrite the preceding uniform estimate as a total variation estimate on the full state space, including the absorbing state.

\begin{corollary}\label{TV}
There exists $C>0$ such that, for all sufficiently large $K$,
all $t\ge0$, and all $n\ge1$,
\[
\begin{split}
&\left\|\mathbb{P}_n(X_t^K\in\cdot)
-\left[e^{-\rho_1t}d_n(K)\nu
+(1-e^{-\rho_1t}d_n(K))\delta_0\right]\right\|_{\mathrm{TV}}\le C\left[e^{-\frac{c_0}{4}t}+R(t,K,\rho_2)\right].
\end{split}
\]
Here the bracketed expression has total mass one but may be signed;
its difference from the probability law is interpreted using the
zero-total-mass convention of Section~\ref{subsec:notation}.
\end{corollary}
\begin{proof}
The difference of the two measures has total mass zero. For any
$A\subseteq\mathbb{N}$, its value on $A$ equals either its value on
$A\cap\mathbb{N}^*$, when $0\notin A$, or the negative of its value on
$A^c\subseteq\mathbb{N}^*$, when $0\in A$.
Thus its total variation equals the supremum of the absolute values
over subsets of $\mathbb{N}^*$, and the claim follows from
Proposition~\ref{proposition7.3}.
\end{proof}

The preceding estimates provide both the uniform spectral remainder
and the approximation of its principal coefficient. We now combine
them to prove the total variation bounds.

\subsection{Proof of Theorem \ref{theorem of TV}}

We are now ready to prove Theorem \ref{theorem of TV}.

\begin{proof}[Proof of Theorem \ref{theorem of TV}]
Write $a_K=u_{n_2(K)}^0$ and
$A_K=\langle\phi,\mathbbm{1}\rangle_\pi/\|\phi\|_\pi^2$.
Proposition~\ref{le mathcal O (1) rho_1(K)K} and the uniform relative
eigenvector estimate imply
\[
d_n(K)=A_K\phi_n
=\frac{\Phi_n}{a_K}(1+O(K\rho_1))
=\alpha_n(K)(1+O(K\rho_1)),
\]
uniformly in $n$. Since $0<\alpha_n(K)\le1$,
\[
|e^{-\rho_1t}d_n(K)-\alpha_n(K)|
\le O(1)K\rho_1e^{-\rho_1t}+O(1)\left(1-e^{-\rho_1t}\right).
\]
Corollary~\ref{TV} and the triangle inequality prove
\eqref{equation TV}.

For the conditional estimate, set
\[
\varepsilon_K(t):=C e^{-\rho_1t}
\left[e^{-\frac{c_0}{8}t}+B_K e^{-\frac{\rho_2-\rho_1}{2}t}\right],
\quad a=e^{-\rho_1t}d_n(K)>0,
\]
where $C$ is chosen to dominate the bound in
\eqref{eq:uniform-spectral-refined}. If
$\varepsilon_K(t)<a/2$, then $P_t(n,\mathbb{N}^*)\ge a/2$, and for
$A\subseteq\mathbb{N}^*$,
\[
\begin{split}
\left|\frac{P_t(n,A)}{P_t(n,\mathbb{N}^*)}-\nu(A)\right|
&\le\frac{|P_t(n,A)-a\nu(A)|
+\nu(A)|a-P_t(n,\mathbb{N}^*)|}{P_t(n,\mathbb{N}^*)}\le\frac{4\varepsilon_K(t)}{a}.
\end{split}
\]
If $\varepsilon_K(t)\ge a/2$, the total variation distance is at most
one and $4\varepsilon_K(t)/a\ge2$. Taking the supremum over $A$ and
combining the two cases gives
\[
\left\|\frac{P_t(n,\cdot)}{P_t(n,\mathbb{N}^*)}-\nu\right\|_{\mathrm{TV}}
\le\min\left\{1,\frac{4\varepsilon_K(t)}{a}\right\}.
\]
Cancelling $e^{-\rho_1t}$ in $\varepsilon_K(t)/a$ and increasing
$C$ proves \eqref{ptnAptnN} for every $t\ge0$.
\end{proof}
\begin{remark}\label{remark pointwise}
For each initial state $n$, the pointwise spectral estimate yields
\[
\left\|\frac{P_t(n,\cdot)}{P_t(n,\mathbb{N}^*)}-\nu\right\|_{\mathrm{TV}}
\le\min\left\{1,
\frac{4}{d_n(K)}\sqrt{\frac{S_K}{\pi_n}}
e^{-(\rho_2-\rho_1)t}\right\}.
\]
The proof is the same two-case argument, with
$\varepsilon_K(t)$ replaced by $e^{-\rho_2t}\sqrt{S_K/\pi_n}$.
The factor $d_n(K)^{-1}$ records the effect of conditioning on survival
from an initial state with a small probability of reaching the
positive metastable region.
\end{remark}

\section{Gaussian approximation and proof of
Theorem~\ref{Gaussian approximation}}
\label{section of Gaussian approximation}

We first compare the quasi-stationary weights with a Gaussian profile
on a window of width $\sqrt K\ln K$ around $n_2(K)$. We then control
the low-population basin and the remaining tails using the
eigenvector and reversible-weight estimates. Normalizing the two
profiles gives the total variation bound.

\begin{proof}[Proof of Theorem \ref{Gaussian approximation}]
Write $k=n_1(K)$, $m=n_2(K)$, $N=n_{\frac12}(K)$,
$m_K=\lfloor\sqrt K\ln K\rfloor$, and
\[
W_K:=\{n\ge1:|n-m|\le m_K\},\quad
q_K(n):=\frac{\pi_n\phi_n}{\pi_m\phi_m},\quad
g_K(n):=\exp\left\{-\frac{(n-m)^2}{2K\sigma^2}\right\}.
\]
We first compare the two unnormalized sequences.
The rate functions are $C^2$ and positive near $x_2$, so uniformly for
$n\in W_K$,
\[
\widetilde{\lambda}\left(\frac nK\right)
=\widetilde{\lambda}(x_2)\left(1+O\left(\frac{\ln K}{\sqrt K}\right)\right),
\quad
\widetilde{\mu}\left(\frac nK\right)
=\widetilde{\mu}(x_2)\left(1+O\left(\frac{\ln K}{\sqrt K}\right)\right).
\]
For $n\le m$, the exact product identity is
\[
\frac{\pi_n}{\pi_m}=\frac{\mu_m}{\mu_n}\Gamma_{m,n}.
\]
Lemma~\ref{Gamma n m=sqrt}, the analogous identity for $n\ge m$,
and Taylor expansion of $H$ at $x_2$ therefore yield
\[
\frac{\pi_n}{\pi_m}
=g_K(n)\left(1+O\left(\frac{(\ln K)^3}{\sqrt K}\right)\right),
\quad n\in W_K.
\]
The use of $m=\lfloor Kx_2\rfloor$ contributes an error of order
$O(1)\ln K/\sqrt K$, which is included above.
On $W_K$, \eqref{eq:Phi-flat-bound} and the relative eigenvector
estimate imply
\[
\frac{\phi_n}{\phi_m}
=1+O(1)K\rho_1
+O\left(1\right)\sqrt K e^{-KH(x_1)+C(\ln K)^2}.
\]
Thus
\begin{equation}\label{pinpin2K-gKn n2K-mK+mK}
\sum_{n\in W_K}|q_K(n)-g_K(n)|\le O(1)(\ln K)^4.
\end{equation}

We next bound the tails. Equations \eqref{1 n1(K) pi n le 1} and
\eqref{Phi n le 1lenlen1K}, together with the relative eigenvector
estimate, give
\[
\sum_{n=1}^k q_K(n)\le O(1)K^{\frac32}e^{-KH(x_1)}.
\]
Choose $0<\delta<\frac12\min\{x_2-x_1,x_{\frac12}-x_2\}$ sufficiently
small. On
$[x_1,x_2-\delta]\cup[x_2+\delta,x_{\frac12}]$,
$F(x)\le F(x_2)-\eta$ for some $\eta>0$.
For $x$ close to $x_2$, the nondegeneracy $F''(x_2)<0$ gives
$F(x)\le F(x_2)-c(x-x_2)^2$.
Using \eqref{1 n1(K) pi n le 1} and the uniform bound
$\phi_n/\phi_m\le C$, we consequently obtain
\[
\sum_{\substack{k<n\le N\\n\notin W_K}}q_K(n)
\le O(1)K e^{-\eta K}+O(1)K e^{-c(\ln K)^2}.
\]
The one-step floor errors only change the constants.
Since $F(x_{\frac12})<F(x_2)$ and the reversible weights have a
geometric tail after $N+1$, the same product estimate gives
\[
\sum_{n>N}q_K(n)\le O(1)e^{-\eta'K}
\]
for some $\eta'>0$.
The Gaussian tail satisfies
\[
\sum_{n\notin W_K}g_K(n)\le O(1)\sqrt K e^{-c'(\ln K)^2}
\]
for some $c'>0$.
Combining these bounds with \eqref{pinpin2K-gKn n2K-mK+mK} gives
\[
E_K:=\sum_{n\ge1}|q_K(n)-g_K(n)|\le O(1)(\ln K)^4.
\]

Let $Q_K=\sum\limits_{n\ge1}q_K(n)$ and $Z(K)=\sum\limits_{n\ge1}g_K(n)$.
Then $|Q_K-Z(K)|\le E_K$. A Gaussian sum-integral comparison gives
\[
Z(K)=\sqrt{2\pi K}\,\sigma+O(1).
\]
Since $\nu_n^K=q_K(n)/Q_K$ and $G_n^K=g_K(n)/Z(K)$,
\[
\begin{split}
2\|\nu^K-G^K\|_{\mathrm{TV}}
&\le\frac{1}{Z(K)}\sum_{n\ge1}|q_K(n)-g_K(n)|
+Q_K\left|\frac1{Q_K}-\frac1{Z(K)}\right|\\
&\le\frac{2E_K}{Z(K)}
\le O\left(\frac{(\ln K)^4}{\sqrt K}\right).
\end{split}
\]
This proves the theorem.
\end{proof}

\section{Application to a logistic birth--death process with an Allee effect}\label{section of Application}

We conclude by applying the preceding results to a density-dependent
birth--death model exhibiting a strong Allee effect. The corresponding
deterministic dynamics are described by the classical cubic Allee
equation; see, for example, \cite{dennis1989allee,kramer2009evidence}. The
birth--death formulation is motivated by the Allee-type population
models studied in \cite{mulkey2006birth}. Let $K>1$ be the population-size scaling
parameter and let $X_t^K$ denote the population size at time $t$.
Fix constants
\[
        0<a<b,\quad r>0,\quad s>0.
\]
Here $aK$ represents the Allee threshold, whereas $bK$ represents
the carrying capacity. If one wishes to normalize the carrying
capacity to $K$, one may simply take $b=1$.

Consider the birth--death process $(X_t^K)_{t\geq0}$ with birth and
death rates
\begin{equation}\label{examplelambdan-mun}
        \lambda_n^K
        =
        n\widetilde{\lambda}\left(\frac{n}{K}\right),
        \quad
        \mu_n^K
        =
        n\widetilde{\mu}\left(\frac{n}{K}\right),
        \quad n\geq1,
\end{equation}
where
\begin{equation}\label{examplelambdamu}
        \widetilde{\lambda}(x)
        =
        r(a+b)x+s,
        \quad
        \widetilde{\mu}(x)
        =
        rx^2+rab+s.
\end{equation}
Notice that the functions $\widetilde{\lambda}$ and
$\widetilde{\mu}$ are independent of $K$. Moreover,
\[
\begin{aligned}
x\bigl(\widetilde{\lambda}(x)-\widetilde{\mu}(x)\bigr)
&=
rx\bigl((a+b)x-x^2-ab\bigr)=
rx(x-a)(b-x).
\end{aligned}
\]
Thus, the associated deterministic equation is
\begin{equation}\label{rxx--x}
        \frac{dx}{dt}
        =
        rx(x-a)(b-x).
\end{equation}
Its equilibria are $0$, $a$, and $b$. The equilibria $0$ and
$b$ are locally asymptotically stable, whereas $a$ is unstable.
Consequently, \eqref{rxx--x} exhibits the bistable structure associated
with a strong Allee effect.

More precisely, suppose that
\[
        \frac{X_0^K}{K}
        \xrightarrow[K\to\infty]{\mathbb{P}}
        x_0.
\]
Then the density-dependent Markov-chain limit theorem from
\cite{Kurtz} implies that, for every $T>0$,
\[
        \sup_{0\leq t\leq T}
        \left|
        \frac{X_t^K}{K}-x(t)
        \right|
        \xrightarrow[K\to\infty]{\mathbb{P}}
        0,
\]
where $x(t)$ is the solution of \eqref{rxx--x} with $x(0)=x_0$.

We now verify that the rate functions in
\eqref{examplelambdamu} satisfy assumptions
\textbf{(A1)}-\textbf{(A3)} and \textbf{(H)}.

\begin{lemma}\label{lem:Allee-example}
Let $\widetilde{\lambda}$ and $\widetilde{\mu}$ be defined by
\eqref{examplelambdamu}. Then assumptions
\textbf{(A1)}-\textbf{(A3)} and \textbf{(H)} are satisfied, with
\[
        x_1=a,\quad x_2=b.
\]
\end{lemma}

\begin{proof}
We verify the assumptions successively.

\medskip
\noindent\textbf{Verification of \textbf{(A1)} and \textbf{(A2)}.}
The functions $\widetilde{\lambda}$ and $\widetilde{\mu}$ belong
to $C^\infty(\mathbb{R}_+)$, are positive and strictly increasing,
and satisfy
\[
        \widetilde{\lambda}(0)=s
        <
        rab+s
        =
        \widetilde{\mu}(0).
\]
Furthermore,
\[
        \widetilde{\lambda}(x)-\widetilde{\mu}(x)
        =
        r(x-a)(b-x),
\]
and hence
\[
        \widetilde{\lambda}(a)=\widetilde{\mu}(a),
        \quad
        \widetilde{\lambda}(b)=\widetilde{\mu}(b).
\]
The two intersections are transversal because
\[
\begin{aligned}
\widetilde{\lambda}'(a)-\widetilde{\mu}'(a)
&=r(b-a)>0,\quad\widetilde{\lambda}'(b)-\widetilde{\mu}'(b)
=r(a-b)<0.
\end{aligned}
\]

It remains to verify the monotonicity condition in $(0,a)$.
Let
\[
        C:=rab+s.
\]
Then
\[
\begin{aligned}
\left(
\ln\frac{\widetilde{\mu}(x)}
          {\widetilde{\lambda}(x)}
\right)'
&=
\frac{2rx}{rx^2+C}
-
\frac{r(a+b)}{r(a+b)x+s}=
\frac{
r\bigl[
r(a+b)x^2+2sx-(a+b)C
\bigr]
}{
(rx^2+C)\bigl(r(a+b)x+s\bigr)
}.
\end{aligned}
\]
Set
\[
        q(x)
        :=
        r(a+b)x^2+2sx-(a+b)C.
\]
Since
\[
        q'(x)=2r(a+b)x+2s>0,
\]
the function $q$ is strictly increasing on $\mathbb{R}_+$.
Moreover,
\[
\begin{aligned}
q(a)
&=
r(a+b)a^2+2sa-(a+b)(rab+s)=
(a-b)\bigl(ra(a+b)+s\bigr)
<0.
\end{aligned}
\]
It follows that $q(x)<0$ for every $x\in(0,a)$. Therefore,
\[
        \left(
        \ln\frac{\widetilde{\mu}(x)}
                  {\widetilde{\lambda}(x)}
        \right)'
        <0,
        \quad x\in(0,a),
\]
which proves \textbf{(A1)} and \textbf{(A2)}.

\medskip
\noindent\textbf{Verification of \textbf{(A3)}.}
Since $\widetilde{\lambda}$ is affine and
$\widetilde{\mu}$ is quadratic,
\[
        \lim_{x\to\infty}
        \frac{\widetilde{\lambda}(x)}
             {\widetilde{\mu}(x)}
        =
        \lim_{x\to\infty}
        \frac{r(a+b)x+s}{rx^2+C}
        =
        0.
\]
In addition,
\[
        \frac{\widetilde{\mu}'(x)}
             {\widetilde{\mu}(x)}
        =
        \frac{2rx}{rx^2+C},
\]
and therefore
\[
        \sup_{x\in\mathbb{R}_+}
        \frac{\widetilde{\mu}'(x)}
             {\widetilde{\mu}(x)}
        <\infty.
\]
Finally,
\[
\begin{aligned}
\int_b^\infty
        \frac{dx}{x\widetilde{\mu}(x)}
&=
\int_b^\infty
        \frac{dx}{x(rx^2+C)}=
\frac{1}{2C}
\left.
\ln\frac{x^2}{rx^2+C}
\right|_{b}^{\infty}
<\infty.
\end{aligned}
\]
Thus \textbf{(A3)} holds.

\medskip
\noindent\textbf{Verification of \textbf{(H)}.}
Define
\[
        H(x)
        =
        \int_b^x
        \ln
        \frac{\widetilde{\mu}(u)}
             {\widetilde{\lambda}(u)}
        \,du
        =
        \int_b^x
        \ln
        \frac{ru^2+C}
             {r(a+b)u+s}
        \,du.
\]
Then
\[
        H'(x)
        =
        \ln
        \frac{rx^2+C}
             {r(a+b)x+s},\quad       
             H''(x)
        =
        \frac{2rx}{rx^2+C}
        -
        \frac{r(a+b)}{r(a+b)x+s}.
\]
Differentiating once more gives
\[
        H'''(x)
        =
        \frac{2r(C-rx^2)}
             {(rx^2+C)^2}
        +
        \frac{r^2(a+b)^2}
             {\bigl(r(a+b)x+s\bigr)^2}.
\]
Since $C>0$ and $s>0$, the function $H'''$ is continuous on
$\mathbb{R}_+$. Moreover,
\[
        \lim_{x\to\infty}
        (1+x^2)
        \frac{2r(C-rx^2)}
             {(rx^2+C)^2}
        =
        -2,\quad
        \lim_{x\to\infty}
        (1+x^2)
        \frac{r^2(a+b)^2}
             {\bigl(r(a+b)x+s\bigr)^2}
        =
        1.
\]
Therefore,
\[
        \sup_{x\in\mathbb{R}_+}
        (1+x^2)|H'''(x)|
        <\infty.
\]
Thus \textbf{(H)} is satisfied.
\end{proof}

Lemma~\ref{lem:Allee-example} shows that the model
\eqref{examplelambdan-mun}--\eqref{examplelambdamu} lies within the
framework developed in this paper. In particular, the Allee threshold
and the positive stable population level are given by
\[
        n_1(K)=\lfloor aK\rfloor,
        \quad
        n_2(K)=\lfloor bK\rfloor.
\]
Consequently, all the results established in
Section~\ref{M-result}, including the sharp asymptotics of
the principal eigenvalue and the mean extinction time, the quantitative
convergence to the unique quasi-stationary distribution, and its
Gaussian approximation near $bK$, apply to this model.

\appendix
\section{Auxiliary estimates}

Set
\[
M:=
\frac{1}{\widetilde{\mu}(x_2)}
+
\int_{x_2}^{\infty}
\frac{dx}{x\widetilde{\mu}(x)}
<\infty,
\]
where the finiteness follows from \textbf{(A3)}. The first result provides elementary estimates for the birth and
death rates.

\begin{lemma}\label{sum le M sum le CK}
	(\cite{Chazottes2016}, Lemma 9.1) There exists a constant $C\ge1$ such that for every $K\ge K_0$, it holds
	\[
		\sum_{p=n_2(K)+1}^{\infty}\frac{1}{\mu_p}\le M,\quad
		\sum_{j=n_2(K)}^{\infty}\frac{1}{\lambda_j\pi_j}\sum_{p=j+1}^{\infty}\pi_p\le CK.
	\]
\end{lemma}

Consider the linear equations
\begin{equation}\label{f_n}
	a_nx_{n+1}+b_nx_{n-1}-(a_n+b_n)x_n=c_n,
\end{equation}
where $(a_n)_{n\ge1}$, $(b_n)_{n\ge1}$ and $(c_n)_{n\ge1}$ are sequences of real numbers, and $a_n$, $b_n$ are positive. Define
\[
\Pi_{r,s}=\prod_{j=s}^{r-1}\frac{b_j}{a_j}\quad\text{for}\quad r>s\quad\text{and}\quad \Pi_{s,s}=1.
\]
For $t\le s\le r$,
\[
\Pi_{r,s}=\frac{\Pi_{r,t}}{\Pi_{s,t}}.
\]

\begin{lemma}\label{solution of x_n}
(\cite{Chazottes2016}, Lemma 9.7) The solution of the homogeneous equation \eqref{f_n} when $c_n=0$ for all $n\ge1$ satisfies
\[
x_n=x_s+(x_{s+1}-x_s)\sum_{j=s}^{n-1}\Pi_{j+1,s+1},\quad\forall n\ge s.
\]
The general solution of equation \eqref{f_n} is given by
\[
x_n=x_s+(x_{s+1}-x_s)\sum_{j=s}^{n-1}\Pi_{j+1,s+1}+\sum_{j=s}^{n-1}\sum_{r=s+1}^{j}\frac{c_r}{a_r}\Pi_{j+1,r+1},\quad\forall n\ge s.
\]
If the series $\sum\limits_{r=s}^{\infty}\frac{c_r}{a_r\Pi_{r+1,s}}$ converges, it can also be written as
\[
x_n=x_s+\widetilde{A}_s\sum_{j=s}^{n-1}\Pi_{j+1,s}-\sum_{j=s}^{n-1}\sum_{r=j+1}^{\infty}\frac{c_r}{a_r\Pi_{r+1,j+1}},\quad\forall n\ge s,
\]
for a constant $\widetilde{A}_s$.
\end{lemma}

For $n> m$ let
\[
\Gamma_{n,m}=\prod_{j=m}^{n-1}\frac{\mu_j}{\lambda_j}\quad\text{with}\quad\Gamma_{m,m}=1.
\]

\begin{lemma}\label{Gamma n m=sqrt}
(\cite{Chazottes2016} Lemma 9.4) For all $m,n\in\mathbb{N}^*$, $m< n$, there holds
\[
\Gamma_{n,m}=\sqrt{\frac{\mu_m}{\lambda_m}\frac{\lambda_n}{\mu_n}}e^{K\left(H\left(\frac{n}{K}\right)-H\left(\frac{m}{K}\right)\right)+\frac{1}{K}c(m,n,K)}.
\]
where $H$ is defined in \eqref{Hx} and $\sup\limits_{m,n,K}|c(m,n,K)|<\infty.$
\end{lemma}

Recall that $F(x)=H(0)-H(x)$, so $F(0)=0$, $F'(x_1)=F'(x_2)=0$,
$F''(x_1)>0$, and $F''(x_2)<0$. Under \textbf{(H)}, $F$ is
three times continuously differentiable. Let $m_K:=\lfloor\sqrt K\ln K\rfloor$.
We first estimate $\sum\limits_{n=1}^{\infty}\pi_n$ without imposing a
sign condition on $F(x_2)$. The two cases needed in the main text
then follow as immediate consequences.

\begin{lemma}\label{lem:total-reversible-weight}
Set $c:=F(x_2)$. There exist constants $c_*,C_*>0$ such that, for
all sufficiently large $K$,
\begin{equation}\label{eq:total-weight-general}
c_*\left(1+\frac{e^{cK}}{\sqrt K}\right)
\leq\sum_{n=1}^{\infty}\pi_n
\leq C_*\left(1+\frac{e^{cK}}{\sqrt K}\right).
\end{equation}
\end{lemma}

\begin{proof}
Since $F$ decreases on $(0,x_1)$ and increases on $(x_1,x_2)$,
choose $\epsilon\in(0,x_1)$ close enough to $x_1$ that
$F(\epsilon)<F(x_2)=c$. On $[0,\epsilon]$, there is a constant
$q\in(0,1)$ such that $\widetilde{\lambda}/\widetilde{\mu}\leq q$.
Consequently,
\[
\sum_{1\leq n<\lceil\epsilon K\rceil}\pi_n
=\sum_{1\leq n<\lceil\epsilon K\rceil}
\frac1{\lambda_n}\prod_{j=1}^{n}\frac{\lambda_j}{\mu_j}
\leq\frac1{\widetilde{\lambda}(0)}\sum_{n=1}^{\infty}\frac{q^n}{n}
\leq O(1).
\]
Also, $\pi_1=1/\widetilde{\mu}(1/K)$ is bounded below by a positive
constant for sufficiently large $K$.

Choose $R>x_2$ sufficiently large that
$\widetilde{\lambda}(x)/\widetilde{\mu}(x)\leq1/2$ for $x\geq R$.
For $\epsilon\leq n/K\leq R+1$, Lemma~\ref{Gamma n m=sqrt} gives
\begin{equation}\label{pi n=le}
\frac{c_*}{K}e^{KF\left(\frac nK\right)}
\leq\pi_n\leq\frac{C_*}{K}e^{KF\left(\frac nK\right)},
\end{equation}
where $c_*,C_*>0$ are independent of $n$ and $K$.
To see this, use $\pi_n=(\mu_n\Gamma_{n,1})^{-1}$,
$H(1/K)=H(0)+O(1/K)$, and the positive upper and lower bounds
for $\widetilde{\lambda}$ and $\widetilde{\mu}$ on the stated compact
interval.

Since $F'(x_2)=0$ and $F''(x_2)<0$, there are constants
$\delta,\beta_1,\beta_2>0$ such that
\[
c-\beta_2(x-x_2)^2\leq F(x)
\leq c-\beta_1(x-x_2)^2,
\quad |x-x_2|\leq\delta.
\]
Take $\delta$ small enough that $x_1<x_2-\delta$ and
$x_2+\delta<R$. By \eqref{pi n=le} and Gaussian-sum estimates,
\begin{equation}\label{eq:stable-weight-mass}
 c_*\frac{e^{cK}}{\sqrt K}
\leq\sum_{|n/K-x_2|\leq\delta}\pi_n
\leq C_*\frac{e^{cK}}{\sqrt K}.
\end{equation}
For the upper bound, sum
$\exp(-\beta_1(n-Kx_2)^2/K)$ over all integers, obtaining $O(\sqrt K)$.
For the lower bound, restrict to $|n-Kx_2|\leq\sqrt K$; there are
at least $\sqrt K$ such positive integers for large $K$, and each
Gaussian factor is bounded below by $e^{-\beta_2}$.

For later use, the same quadratic bound also implies, with
$N=n_2(K)$ and $m=m_K$, that there is $c'_2>0$ such that
\begin{equation}\label{Fx le x2-delta n2K-mK}
F(x)\leq c-\frac{c'_2(\ln K)^2}{K},
\quad
x\in\left[x_2-\delta,\frac{N-m}{K}\right]
\cup\left[\frac{N+m}{K},x_2+\delta\right],
\end{equation}
for all sufficiently large $K$. The integer rounding changes the
distance to $x_2$ by at most $1/K$ and is absorbed by reducing
$c'_2$.

On the remaining compact set
$[\epsilon,x_2-\delta]\cup[x_2+\delta,R+1]$, $F$ is bounded above
by $c-\eta$ for some $\eta>0$. Indeed, $F(\epsilon)<c$, $F$ is
strictly decreasing on $(0,x_1)$ and $(x_2,\infty)$, and it is
strictly increasing on $(x_1,x_2)$. Thus the sum of $\pi_n$ over
this compact region is at most $C e^{(c-\eta)K}$.

Finally, for any integer cutoff $r$ such that
$\lambda_j/\mu_j\leq1/2$ for $j>r$, monotonicity of $\lambda_n$
gives
\begin{equation}\label{pi n le x frac 1 2 infty}
\pi_n=\pi_r\frac{\lambda_r}{\lambda_n}
\prod_{j=r+1}^{n}\frac{\lambda_j}{\mu_j}
\leq\pi_r2^{-(n-r)},\quad n\geq r,
\quad\sum_{n=r}^{\infty}\pi_n\leq2\pi_r.
\end{equation}
This applies in particular to the cutoff $n_{\frac12}(K)$ used
in the main text. Taking $r=\lceil RK\rceil$ in the present proof,
\eqref{pi n=le} and $F(r/K)\leq F(R)<c$ show that this tail is at
most $O(1) 1/Ke^{(c-\eta)K}$, after reducing $\eta$ if necessary.

Combining the boundary estimate, the stable-region estimate
\eqref{eq:stable-weight-mass}, and the bounds on the remaining compact
region and the infinite tail proves \eqref{eq:total-weight-general}.
\end{proof}

\begin{lemma}\label{sum pi n F(x 2)>0}
Assume $c:=F(x_2)>0$. There exists $C_1>0$ such that, for all
sufficiently large $K$,
\[
C_1^{-1}\frac{e^{cK}}{\sqrt K}
\leq\sum_{n=1}^{\infty}\pi_n
\leq C_1\frac{e^{cK}}{\sqrt K}.
\]
\end{lemma}

\begin{proof}
By Lemma~\ref{lem:total-reversible-weight}, it is enough to note that
$K^{-1/2}e^{cK}\to\infty$ when $c>0$. Thus the constant term in
\eqref{eq:total-weight-general} is absorbed by $K^{-1/2}e^{cK}$.
\end{proof}

\begin{lemma}\label{sum pi n F(x 2)le0}
Assume $F(x_2)\leq0$. There is a constant $C_2>0$ such that, for all
sufficiently large $K$,
\[
C_2^{-1}\leq\sum_{n=1}^{\infty}\pi_n\leq C_2.
\]
\end{lemma}

\begin{proof}
Apply Lemma~\ref{lem:total-reversible-weight} with $c=F(x_2)\leq0$.
Since $K^{-1/2}e^{cK}\leq K^{-1/2}$, the quantity
$1+K^{-1/2}e^{cK}$ is bounded above and below by positive constants
for $K\geq1$. The conclusion follows from
\eqref{eq:total-weight-general}.
\end{proof}

%\cite{allee1938social,amarasekare1998allee,kramer2009evidence,lewis1993allee,Thattai2001,Yaglom1947,yule1925mathematical}

% Add acknowledgments here, if applicable.

\end{document}